\documentclass[12pt,reqno]{amsart}

\usepackage{amssymb}
\usepackage{amsfonts}
\usepackage{amsmath}
\usepackage{graphicx}
\usepackage{float}
\usepackage{color}
\usepackage{mathrsfs}
\usepackage[all]{xy}
\usepackage{tikz-cd}
\usepackage{url}
\usepackage{hyperref}
\usepackage{graphicx}
\usepackage[capitalize,noabbrev]{cleveref}

\newtheorem{theorem}{Theorem}[section]
\newtheorem{proposition}[theorem]{Proposition}
\newtheorem{lemma}[theorem]{Lemma}
\newtheorem{corollary}[theorem]{Corollary}
\newtheorem{claim}[theorem]{Claim}
\newtheorem{conjecture}[theorem]{Conjecture}
\theoremstyle{definition}
\newtheorem{definition}[theorem]{Definition}
\newtheorem{remark}[theorem]{Remark}

\newtheorem{convention}[theorem]{Convention}

\newcommand{\F}{\mathfrak F}
\newcommand{\G}{\mathfrak G}

\newcommand{\Grz}{{\sf S4.Grz}}

\newcommand{\grz}{{\sf grz}}

\newcommand{\ccc}{{\sf cl}}
\newcommand{\ii}{{\sf int}}

\newcommand{\bd}{{\sf bd}}

\newcommand{\grzform}{\Box(\Box(p\to\Box p)\to p)\to p}

\def\Cech{\v{C}ech}
\def\CS{\Cech-Stone}

\makeatletter
\@namedef{subjclassname@2020}{\textup{2020} Mathematics Subject Classification}
\makeatother

\begin{document}

\title{On Shehtman's two problems}

\author{G. Bezhanishvili}
\address{New Mexico State University}
\email{guram@nmsu.edu}

\author{N. Bezhanishvili}
\address{University of Amsterdam}
\email{N.Bezhanishvili@uva.nl}

\author{J. Lucero-Bryan}
\address{New Mexico State University}
\email{jglb@nmsu.edu}

\author{J. van Mill}
\address{University of Amsterdam}
\email{j.vanMill@uva.nl}

\begin{abstract}
We provide partial solutions to two problems  posed by 
Shehtman concerning the modal logic of the \CS\ compactification of an ordinal space. 
We use the Continuum Hypothesis to give a finite axiomatization of the modal logic of 
$\beta(\omega^2)$, thus 
resolving Shehtman's first problem for $n=2$. 
We also characterize 
modal logics arising from the \CS\ compactification of an ordinal $\gamma$ provided the Cantor normal form of $\gamma$ satisfies an additional condition. This gives a partial solution of Shehtman's second problem. 
\end{abstract}



\maketitle

\tableofcontents

\section{Introduction}

In 2020, during a virtual logic seminar hosted by the Steklov Math.~Institute, Valentin Shehtman posed the following two problems:
\begin{itemize}
\item[{\bf P1}:] For each nonzero $n\in\omega$, axiomatize the modal logic of the \CS\ compactification $\beta(\omega^n)$ of the ordinal space $\omega^n$.
\item[{\bf P2}:] Describe the modal logics that arise as the logic of the \CS\ compactification of some ordinal space.
\end{itemize} 
We use the Continuum Hypothesis (CH) to solve 
{\bf P1} for $n=2$ (see \cref{thm: P1 for n=2}). 
We also give a partial solution of {\bf P2} by characterizing modal logics 
arising as the logic of the \CS\ compactification of some ordinal space $\gamma$ provided the Cantor normal form of $\gamma$ satisfies an additional condition (see \cref{thm: partial soln of P2}). At the end of the paper, we conjecture full solutions for {\bf P1} and {\bf P2} (see \cref{sec: conclusions}).

To give context, we recall that topological semantics for intuitionistic and modal logics originated in the late 1930s/early 1940s in the works of Stone \cite{Sto37b}, Tarski \cite{Tar38}, Tsao-Chen \cite{TC38}, McKinsey \cite{McK41}, and McKinsey and Tarski \cite{MT44,MT46,MT48}. The key feature is that we interpret the necessity connective ($\Box$) as topological interior, and hence the possibility connective ($\Diamond$) as topological closure. Consequently, the logic of all topological spaces is Lewis' well-known modal system $\sf S4$. However, much more is true.
It is a celebrated result of McKinsey and Tarski \cite{MT44} that {\sf S4} is the logic of any separable dense-in-itself metric space. Rasiowa and Sikorski \cite{RS63} proved that the hypothesis of separability may be omitted in the presence of the Axiom of Choice. But dropping the dense-in-itself hypothesis may produce a logic different from $\sf S4$. A full description of the logics that arise as the logic of a metric space was given more recently in \cite{BGLB15b}.

A canonical example of a non-metrizable space is the \CS\ compactification $\beta(\omega)$. 
In \cite{BH09} it was shown that the logic of $\beta(\omega)$ is {\sf S4.1.2}.
We point out that the original proof utilized a set-theoretic assumption beyond ZFC which controls the cardinality of MAD families of subsets of $\omega$. Recently Alan Dow showed that this result can be proved within ZFC \cite{Dow19}. 

A natural next step is to determine the modal logic of $\beta(\omega^2)$, which is exactly the special case of Shehtman's first problem for $n=2$. Our answer relies on Parovi\v{c}enko's 
characterization of the remainder $\omega^*$ of $\beta(\omega)$ (see \cite{Par63}), and on its consequence known as the Homeomorphism Extension Theorem (see \cite[Thm.~1.1]{vDvM93}). 
Both of these results use the Continuum Hypothesis, hence our dependance on CH. 

Our strategy is to determine the logic of $\beta(\omega^2)$ by utilizing relational semantics (also known as Kripke semantics) of modal logic. 
We introduce a new class of frames, which we term `roaches' based on their shape. For each nonzero $n\in\omega$, an $n$-roach has 
a splitting point $s$ of depth at most $n$, 
where $s$ serves as the head, the downset ${\downarrow}s$ as the body, and the upset ${\uparrow}s$ as the antennae 
(see \cref{fig: n-roach}). 
\begin{figure}[h]
\begin{picture}(120,110)(-60,-60)
\setlength{\unitlength}{.25mm}
\qbezier(0,-70)(120,-20)(50,50)
\qbezier(0,-70)(-120,-20)(-50,50)
\put(-50,-20){\makebox(0,0){$\bullet$}}
\put(-50,-20){\line(1,-1){25}}
\qbezier(-50,-20)(-70,15)(-25,25)
\qbezier(-50,-20)(0,0)(-25,25)
\put(45,-30){\makebox(0,0){$\bullet$}}
\put(45,-30){\line(-4,-3){36}}
\qbezier(45,-30)(90,-5)(50,50)
\qbezier(45,-30)(0,0)(50,50)
\color{magenta}
\qbezier(0,-70)(70,-35)(0,0)
\qbezier(0,-70)(-70,-35)(0,0)
\put(-25,-45){\makebox(0,0){$\bullet$}}
\put(0,-70){\makebox(0,0){$\bullet$}}
\put(8,-58){\makebox(0,0){$\bullet$}}
\put(0,-25){\makebox(0,0){${\downarrow}s$}}
\put(0,0){\makebox(0,0){$\bullet$}}
\multiput(-25,25)(50,0){2}{\makebox(0,0){$\bullet$}}
\multiput(-50,50)(50,0){3}{\makebox(0,0){$\bullet$}}
\put(0,0){\line(-1,1){50}}
\put(0,0){\line(1,1){50}}
\put(-25,25){\line(1,1){25}}
\put(25,25){\line(-1,1){25}}
\put(0,13){\makebox(0,0){$s$}}
\put(0,65){\makebox(0,0){${\uparrow}s$}}
\multiput(-25,50)(50,0){2}{\makebox(0,0){$\cdots$}}
\put(1,25){\makebox(0,0){$\cdots$}}
\end{picture}
\caption{An $n$-roach.}
\label{fig: n-roach}
\end{figure}
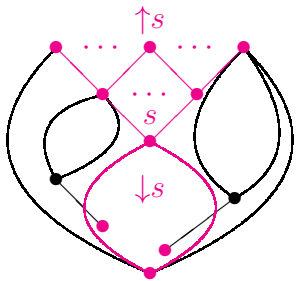
Herein, we focus on $2$-roaches (see \cref{fig2}), and show this class defines the logic of $\beta(\omega^2)$. This we do by first axiomatizing the modal logic of 2-roaches (see \cref{sec: logic of 2 roaches}) and then 
obtaining each $2$-roach as an interior image of $\beta(\omega^2)$ (see \cref{logic of beta omega squared}). 
Defining such interior maps requires the machinery of partition and mapping lemmas
(see \cref{sec: partition lemmas for omega star,sec: mappings of omega star}). These lemmas are consequences of Parovi\v{c}enko's Theorem
and hence of CH. 

The 
axiomatization of the class of 2-roaches is accomplished via 
the machinery of Fine-Jankov formulas (see, e.g., \cite[Sec.~3.4]{BRV01}) in conjunction with Zakharyaschev's sufficient condition for a logic to have the finite model property (see, e.g., \cite[Thm.~11.58]{CZ97}). 
This we do by identifying three forbidden configurations for $\beta(\omega^2)$ (see \cref{sec: beta omega squared and 3 forbid frames}) which yield our desired finite axiomatization.

The paper is organized as follows. \cref{sec: prelims} recalls some basic notions and results from both modal logic and topology. We investigate the logics arising from $\beta(\omega^n)$ in \cref{sec: P2},  which together with the Structural Theorem of \cref{sec: partial soln P2} leads to a partial solution of {\bf P2}. The remainder of the paper is dedicated to describing the logic of $\beta(\omega^2)$. 

In \cref{sec: beta omega squared and 3 forbid frames} we set the scene for the axiomatization by exhibiting the three minimal forbidden configurations.
\cref{roaches section} introduces the family of $n$-roaches for each nonzero $n\in\omega$ and presents some basic results. We study 
the class of $2$-roaches in \cref{sec: logic of 2 roaches}. 
\cref{willow trees} refines our focus by introducing the class of willow trees (see \cref{fig4}) and showing these define the same logic. 
Parovi\v{c}enko's Theorem and the Homeomorphism Extension Theorem are
employed in 
\cref{sec: partition lemmas for omega star,sec: mappings of omega star} to prove the partition and mapping lemmas for the remainder $\omega^*$
of $\beta(\omega)$. These are then used in \cref{logic of beta omega squared} 
to prove that each willow tree is an interior image of $\beta(\omega^2)$, thus completing our finite axiomatization of $\beta(\omega^2)$.
Finally, \cref{sec: conclusions} concludes with our conjectures regarding complete solutions
of Shehtman's two problems.

\section{Preliminaries}\label{sec: prelims}

In this section we recall some basic facts about the relational and topological semantics of modal logics. 
We use \cite{CZ97,BRV01} as our main references for modal logic and \cite{Eng89} for topology. 
We assume the reader is familiar with the modal logic {\sf S4}. We first recall relational semantics for {\sf S4}, which is given by pairs 
$\mathfrak F=(W,\le)$, where $W$ is a set and $\le$ is a quasi ordering of $W$ (that is, $\le$ is reflexive and transitive). Such pairs are known as {\em {\sf S4}-frames}, and it is well known that {\sf S4} is the logic of the class of all (finite) {\sf S4}-frames. 

Let $\mathfrak F=(W,\le)$ be an {\sf S4}-frame. For $A\subseteq W$, let 
\[
{\uparrow}A=\{w\in W\mid \exists a\in A \mbox{ with } a\le w\},
\] 
and define ${\downarrow}A$ dually. Then $A$ is an {\em upset} if $A={\uparrow}A$ and a {\em downset} if $A={\downarrow}A$. For $w\in W$, we write ${\uparrow}w$ for ${\uparrow}\{w\}$ and similarly for ${\downarrow}w$. Call $\mathfrak F$ {\em rooted} provided there is $r\in W$ (called a {\em root} of $\mathfrak F$) such that $W={\uparrow}r$.

We call $\G=(V,\preceq)$ a \emph{subframe} of $\F$ if $V\subseteq W$ and $\preceq$ is the restriction of $\le$ to $V$. The subframe $\G$ is \emph{generated} when $V$ is an upset, 
and \emph{point-generated} when $V={\uparrow}w$ for some $w\in W$.

The {\em cluster} of $w$ is the set $C_w={\uparrow}w\cap{\downarrow}w$. 
We recall that $w\in A$ is a {\em quasi-maximal point} of $A$ if ${\uparrow}w\cap A\subseteq C_w$, and a {\em maximal point} 
if ${\uparrow}w\cap A=\{w\}$; {\em quasi-minimal} and {\em minimal} points of $A$ are defined dually. We write $\max A$ and $\min A$ for the sets of maximal and minimal points of $A$, respectively.  

The following two extensions of {\sf S4} will be important for us.
The logic {\sf S4.1} is obtained from {\sf S4} by postulating the {\em McKinsey axiom} ${\sf ma}=\Box\Diamond p\to\Diamond \Box p$, and the logic {\sf S4.1.2} is obtained from {\sf S4.1} by postulating the {\em Geach axiom} ${\sf ga}=\Diamond\Box p\to\Box\Diamond p$. 

Let $\mathfrak F=(W,\le)$ be a finite rooted {\sf S4}-frame. Then $\mathfrak F$ is an {\sf S4.1}-frame provided for each $w\in W$ there is $m\in\mathrm{max}(\mathfrak F)$ such that $w\le m$, and 
$\mathfrak F$ is an {\sf S4.1.2}-frame provided $\mathrm{max}(\mathfrak F)$ is a singleton. The following completeness results are well known (see, e.g., \cite[Sec.~5.3]{CZ97}):
 
\begin{theorem}\label{thm: reln complnss s4 1 and 2}
\hfill
\begin{itemize}
\item {\sf S4.1} is the logic of all finite rooted {\sf S4.1}-frames.
\item {\sf S4.1.2} is the logic of all finite rooted {\sf S4.1.2}-frames.
\end{itemize}
\end{theorem}

We next recall topological semantics for {\sf S4}. Let $X$ be a topological space. For $A \subseteq X$, we write $\ii(A)$ for the interior and $\ccc(A)$ for the closure of $A$.  
Each modal formula is evaluated in $X$ by assigning a subset of $X$ to each propositional variable, interpreting the classical connectives 
as the corresponding boolean operations, 
$\Box$ as $\ii$, and hence 
$\Diamond$ as $\ccc$.
A formula $\varphi$ is {\em valid} in $X$, written $X\vDash\varphi$, provided it evaluates to $X$ for each assignment of the propositional variables. The logic ${\sf L}(X)$ of $X$ is the set $\{\varphi\mid X\vDash\varphi\}$ of modal formulas valid in $X$, and we have that ${\sf S4}\subseteq{\sf L}(X)$ for each space $X$ (see \cite{MT44}).

Topological semantics generalizes relational semantics for {\sf S4} since the family of upsets of an {\sf S4}-frame $(W,\le)$ forms a topology $\tau_\le$ on $W$, known as the {\em Alexandroff topology}, in which $\ccc(A)$ is ${\downarrow}A$ for each $A \subseteq W$. Therefore,
a modal formula is valid in an {\sf S4}-frame $(W,\le)$ iff it is valid in $(W,\tau_\le)$. Hence, relational completeness of a logic immediately yields topological completeness. But such completeness is with respect to spaces that don't satisfy high separation axioms. Thus, to obtain 
completeness with respect to widely used spaces in topology, further mapping results are required; 
early examples are presented in \cite{MT44,WP47,RS63}, whereas a contemporary 
approach can be found, e.g., in \cite[Ch.~5]{APvB07} 
(see also \cite{BBBM17a,BBBM19,BBBM21b}). 

For a topological space $X$, we have that $X\vDash{\sf ma}$ iff $\ii(\ccc A)\subseteq\ccc(\ii A)$ and $X\vDash{\sf ga}$ iff $\ccc(\ii A)\subseteq\ii(\ccc A)$ for each $A\subseteq X$. It is well known that the latter condition is equivalent to the closure of each open set being open. Such spaces are called {\em extremally disconnected} ({\sf ED} for short). Therefore, $X\vDash{\sf ga}$ iff $X$ is extremally disconnected. 

Following the terminology of \cite{BBBM21b} (which was suggested by Archangel'skii), we call $X$ {\em densely discrete} ({\sf DD} for short)
provided the set of isolated points of $X$ is dense in $X$. It is well known and easy to verify that if $X$ is {\sf DD}, then $X\vDash{\sf ma}$ (the converse however is not true). We thus obtain that if $X$ is {\sf DD}, then ${\sf S4.1}\subseteq{\sf L}(X)$ and if $X$ is {\sf ED}, then ${\sf S4.2}\subseteq{\sf L}(X)$. 
Because topological semantics generalizes relational semantics for {\sf S4}, Theorem~\ref{thm: reln complnss s4 1 and 2} yields: 

\begin{theorem}\label{thm: top complnss s4 1 and 2}
\hfill
\begin{itemize}
\item The logic of the class of all {\sf DD}-spaces is {\sf S4.1}.
\item The logic of the class of all extremally disconnected {\sf DD}-spaces is {\sf S4.1.2}.
\end{itemize}
\end{theorem}

We conclude this preliminary section by recalling the truth-preserving operations between topological spaces, which will play an important role in what follows. 
Let $f:X\to Y$ be a map between topological spaces. Then $f$ is {\em continuous} if $f^{-1}[V]$ is open in $X$ for each open $V \subseteq Y$, $f$ is {\em open} if $f[U]$ is open in $Y$ for each open $U \subseteq X$, and $f$ is {\em interior} if it is both continuous and open. It is well known (see \cite[p.~99]{RS63}) that $f$ is interior iff $\ccc f^{-1}[A]=f^{-1}[\ccc A]$ for each $A\subseteq Y$. If $f$ is an onto interior map, then we say that $Y$ is an {\em interior image} of $X$. We 
have (see, e.g., \cite[Thm.~5.49]{APvB07}):

\begin{lemma}\label{lem: truth preserving ops}
Let $X,Y,Z$ be spaces.
\begin{enumerate}
\item If $Y$ is an interior image of $X$, then ${\sf L}(X)\subseteq{\sf L}(Y)$. \label{lem item: interior image}
\item If $Y$ is an open subspace of $X$, then ${\sf L}(X)\subseteq{\sf L}(Y)$. \label{lem item: open sub}
\item If $X$ is the disjoint union $Y\oplus Z$, then ${\sf L}(X)={\sf L}(Y)\cap {\sf L}(Z)$. \label{lem item: disjoint union}
\end{enumerate}
\end{lemma}

Viewing an {\sf S4}-frame $(W,\le)$ as the topological space $(W,\tau_\le)$, 
we may speak of interior mappings/images when the domain or codomain of a mapping is an {\sf S4}-frame. Indeed, an interior map between {\sf S4}-frames is known under the name of a {\em p-morphism} (see, e.g., \cite[p.~30]{CZ97}); that is, a map $f:W\to V$ 
between two {\sf S4}-frames
$\F=(W,\le)$ and $\G=(V,\preceq)$ satisfying: 
\begin{itemize}
\item $w\le u$ implies $f(w)\preceq f(u)$ ({\em forth condition});
\item $f(w)\preceq v$ implies there is $u\in{\uparrow}w$ such that $f(u)=v$ ({\em back condition}). 
\end{itemize} 

\section{The logics ${\sf L}_n$}
\label{sec: P2} 

In this section we introduce the logics ${\sf L}_n := {\sf L}(\beta(\omega^n)$ and show that they form a strictly decreasing chain of logics above {\sf S4.1}. 

We view each ordinal $\gamma$ as a topological space in the interval topology, 
so $\gamma$ is locally compact, normal, and scattered (see, e.g., \cite[p.~151]{Sem71}). 
Therefore, the \CS\ compactification $\beta(\gamma)$ 
is a zero-dimensional {\sf DD}-space and $\gamma$ is (homeomorphic to) an open subspace of $\beta(\gamma)$. 
Thus, as a consequence of Theorem~\ref{thm: top complnss s4 1 and 2}, we obtain:

\begin{lemma}\label{lem: s4.1 in L beta gamma}
For any ordinal space $\gamma$, we have that ${\sf S4.1}\subseteq {\sf L}(\beta(\gamma))$. 
\end{lemma}

\begin{definition}\label{def: logics Ln}
For nonzero $n\in\omega$, let ${\sf L}_n$ be the logic of $\beta(\omega^n)$. 
\end{definition}

To show that ${\sf L}_{n+1}\subseteq{\sf L}_n$, it is convenient to view each $\omega^{n+1}$
as the product space $(\omega^{n}+1)\times\omega$. To give geometric intuition, $\beta(\omega^2)$ is depicted below.

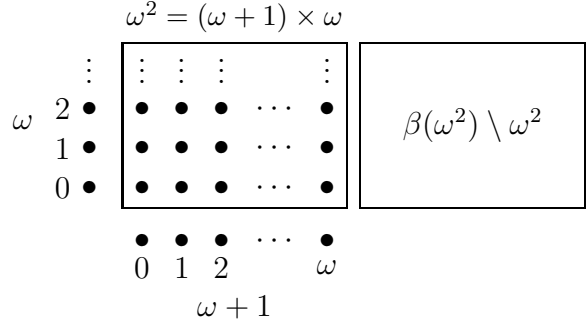
\begin{figure}[h]
\begin{center}
\begin{picture}(212.5,120)(0,-10)
\multiput(37.5,37.5)(0,62.5){2}{\multiput(0,0)(90,0){2}{\line(1,0){85}}}
\multiput(37.5,37.5)(85,0){2}{\multiput(0,0)(90,0){2}{\line(0,1){62.5}}}
\put(80,110){\makebox(0,0){\small $\omega^2=(\omega+1)\times\omega$}}
\put(170,68.75){\makebox(0,0){$\beta (\omega^2) \setminus \omega^2$}}
\put(0,70){\makebox(0,0){$\omega$}}
\put(15,45){\makebox(0,0){$0$}}
\put(15,60){\makebox(0,0){$1$}}
\put(15,75){\makebox(0,0){$2$}}
\multiput(25,45)(0,15){3}{\multiput(0,0)(90,0){2}{\makebox(0,0){$\bullet$}}}
\multiput(25,92.5)(90,0){2}{\makebox(0,0){$\vdots$}}
\put(80,0){\makebox(0,0){$\omega+1$}}
\put(45,15){\makebox(0,0){$0$}}
\put(60,15){\makebox(0,0){$1$}}
\put(75,15){\makebox(0,0){$2$}}
\put(115,15){\makebox(0,0){$\omega$}}
\multiput(45,25)(15,0){3}{\makebox(0,0){$\bullet$}}
\put(115,25){\makebox(0,0){$\bullet$}}
\put(96,25){\makebox(0,0){$\cdots$}}
\multiput(45,45)(15,0){3}{\multiput(0,0)(0,15){3}{\makebox(0,0){$\bullet$}}}
\multiput(96,45)(0,15){3}{\makebox(0,0){$\cdots$}}
\multiput(45,92.5)(15,0){3}{\makebox(0,0){$\vdots$}}
\end{picture}
\end{center}
\caption{Depiction of $\beta(\omega^2)$.}
\label{fig: omega plus one times omega}
\end{figure}

\begin{lemma}\label{lem: omega^n is clopen in omega^n+1}
For each nonzero $n\in\omega$, we have that ${\sf L}_{n+1}\subseteq{\sf L}_n$.
\end{lemma}

\begin{proof} 
Since $(\omega^{n-1}+1)\times\omega$ is a clopen subset of $(\omega^{n}+1)\times\omega$,
we have that $\omega^n$ is homeomorphic to a clopen subset of $\omega^{n+1}$.
Therefore, 
$\ccc(\omega^n)$ is clopen in $\beta(\omega^{n+1})$ \cite[Cor.~3.6.5]{Eng89}. 
Because $\omega^{n+1}$ is normal, $\ccc(\omega^n)$ is homeomorphic to $\beta(\omega^n)$ \cite[Cor.~3.6.8]{Eng89}.
Thus, up to homeomorphism, $\beta(\omega^n)$ is a clopen subspace of $\beta(\omega^{n+1})$, which by Lemma~\ref{lem: truth preserving ops}(\ref{lem item: open sub}) yields that 
$
{\sf L}_n={\sf L}(\beta(\omega^n))\subseteq{\sf L}(\beta(\omega^{n+1}))={\sf L}_{n+1}.
$ 
\end{proof}

We now focus on showing that the containment ${\sf L}_{n+1}\subseteq {\sf L}_n$ is strict.

\begin{definition}\label{def: tall lop-sided two fork}
For each nonzero $n\in\omega$, let $\mathfrak T_n$ be the partial ordering of $\{r,v,w_1,\dots,w_n\}$ depicted in Figure~\ref{fig: frame separating n and n+1}.
\begin{figure}[h]
\begin{picture}(70,95)(-35,-15)
\setlength{\unitlength}{.28mm}
\put(0,-10){\makebox(0,0){$r$}}
\put(0,0){\makebox(0,0){$\bullet$}}
\put(0,0){\line(1,1){25}}
\put(0,0){\line(-1,1){25}}
\multiput(-25,25)(0,25){4}{\makebox(0,0){$\bullet$}}
\put(25,25){\makebox(0,0){$\bullet$}}
\multiput(-25,25)(0,50){2}{\line(0,1){25}}
\put(-25,66.5){\makebox(0,0){$\vdots$}}
\put(35,25){\makebox(0,0){$v$}}
\put(-40,25){\makebox(0,0){$w_1$}}
\put(-40,50){\makebox(0,0){$w_2$}}
\put(-47,75){\makebox(0,0){$w_{n-1}$}}
\put(-40,100){\makebox(0,0){$w_n$}}
\end{picture}
\caption{The partial ordering of $\mathfrak T_n$.} 
\label{fig: frame separating n and n+1}
\end{figure}
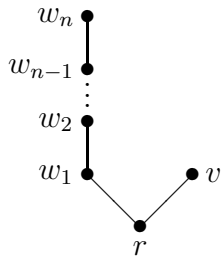
\end{definition}

We show that $\mathfrak T_n$ is an interior image of $\beta(\omega^{n+1})$, 
but not of $\beta(\omega^n)$. 
For this we require the following general
result concerning interior images of zero-dimensional spaces.

\begin{proposition}\label{lem: interior image of open sub of 0-dim is image of whole space}
Let $X$ be a zero-dimensional space and $\mathfrak F=(W,\le)$ a rooted {\sf S4}-frame containing a maximal point $m$. 
Then $\mathfrak F$ is an interior image of $X$ 
iff $\mathfrak F$ is an interior image of some open subspace of $X$.
\end{proposition}

\begin{proof}
One implication is clear. For the other implication, suppose that $\mathfrak F$ is an interior image of an open subspace $Y$ of $X$, say via $f$. Then there is $y\in Y$ such that $f(y)=r$, where $r$ is a root of $\mathfrak F$. Since $X$ is zero-dimensional, there is a clopen subset $U$ of $X$ such that $y\in U$ and $U\subseteq Y$. Because $U$ is open in $Y$ and contains a preimage of $r$, 
the restriction of $f$ to $U$ is an interior mapping onto $\mathfrak F$. 

Define $g:X\to W$ by 
\[
g(x) = \left\{
\begin{array}{ll}
f(x) & \text{if }x\in U,\\
m & \text{if }x\in X\setminus U.
\end{array}
\right.
\]
It is clear that $g$ is a well-defined surjection. To see that $g$ is continuous, let $w\in W$ and consider ${\uparrow}w$. Since $f^{-1}({\uparrow}w)$ is open in $Y$, which is open in $X$, we have that $f^{-1}({\uparrow}w)$ is open in $X$. If $m\not\in {\uparrow}w$, then 
$g^{-1}({\uparrow}w)=f^{-1}({\uparrow}w)\cap U$ is open in $X$ because $U$ is open in $X$. 
If $m\in {\uparrow}w
$, then 
$g^{-1}({\uparrow}w)=(X\setminus U)\cup f^{-1}({\uparrow}w)$ is open in $X$ since $U$ is closed in $X$. Thus, $g$ is continuous. 

To see that $g$ is open, let $V$ be open in $X$. Noting that $V\cap U$ is open in $U$ gives that $f(V\cap U)$ is an 
upset of $\mathfrak F$. Moreover, $g(V\setminus U)$ is either $\{m\}$ or $\varnothing$, both of which are 
upsets of $\mathfrak F$. Since $g(V)=g(V\cap U) \cup g(V\setminus U)=f(V\cap U)\cup g(V\setminus U)$, it follows that $g(V)$ is an 
upset of $\mathfrak F$. Thus, $g$ is open, and hence 
$\mathfrak F$ is an interior image of $X$.
\end{proof}

\begin{corollary}\label{lem: F image of beta iff F image of open sub of beta}
Let $n\in\omega$ be nonzero. A rooted {\sf S4}-frame $\mathfrak F$ is an interior image of $\beta(\omega^n)$ iff $\mathfrak F$ is an interior image of some open subspace of $\beta(\omega^n)$.
\end{corollary}

\begin{proof}
Since one implication is clear, suppose that $\mathfrak F$ is an interior image of 
an open subspace of $\beta(\omega^n)$. By Lemma~\ref{lem: s4.1 in L beta gamma}, 
$\beta(\omega^n)\vDash{\sf S4.1}$. Therefore, 
Lemma~\ref{lem: truth preserving ops} implies that 
$\mathfrak F$ is an {\sf S4.1}-frame. 
Thus, $\mathfrak F$ has a maximal point, and hence  
the result follows 
from Proposition~\ref{lem: interior image of open sub of 0-dim is image of whole space} because $\beta(\omega^n)$ is zero-dimensional.
\end{proof}

We require the concept of depth of 
a finite rooted {\sf S4}-frame $\mathfrak F=(W,\le)$. 
Recall that the {\em depth of $w\in W$}, written $d(w)$, is the greatest $n\in\omega$ such that there are $w_1,\dots,w_n\in W$ 
satisfying $w_1=w$, $w_i\le w_{i+1}$, and $w_{i+1}\not\le w_i$ for each $1\le i<n$. The {\em depth of $\mathfrak F$}, written $d(\mathfrak F)$, is the depth of a root of $\mathfrak F$. It is well known (see, e.g., \cite[Prop.~3.44]{CZ97}) that $d(\mathfrak F)\le n$ iff $\mathfrak F\vDash\bd_n$, where \begin{eqnarray*}
\bd_1 &=& \Diamond\Box p_1\to p_1, \\ 
\bd_{n+1} &=& \Diamond(\Box p_{n+1}\wedge\lnot\bd_n)\to p_{n+1}.
\end{eqnarray*}

We will repeatedly use the following well-known fact, which we state as a lemma for ease of reference.

\begin{lemma}\label{lem: closure of open intersect dense is closure of open}
Let $X$ be a topological space. If $D$ is dense 
and $U$ is open in $X$, then $\ccc (U) = \ccc(U\cap D)$.
\end{lemma}

\begin{lemma}\label{clm: frame sepating n and n=1}
Let 
$n\in\omega$ be nonzero and $\mathfrak T_n$ the 
{\sf S4}-frame depicted in Figure~\ref{fig: frame separating n and n+1}.
\begin{enumerate}
\item 
$\mathfrak T_n$ is an interior image of $\beta(\omega^{n+1})$. \label{clm. item: IS image n+1}
\item 
$\mathfrak T_n$ is not an interior image of any open subspace of $\beta(\omega^n)$. \label{clm. item: is NOT image n}
\end{enumerate}
\end{lemma}

\begin{proof}
(1) Since 
$\omega^{n+1}$ is metrizable 
and $d(\mathfrak T_n) = n+1$, we have that $\mathfrak T_n$ is an interior image of $\omega^{n+1}$ by \cite[Lem.~3.5]{BGLB15b}. Because $\omega^{n+1}$ is locally compact, $\omega^{n+1}$ is open in $\beta(\omega^{n+1})$. 
Thus, $\mathfrak T_n$ is an interior image of $\beta(\omega^{n+1})$ by \cref{lem: F image of beta iff F image of open sub of beta}. 

(2) By \cref{lem: F image of beta iff F image of open sub of beta}, it is sufficient to show that $\mathfrak T_n$ is not an interior image of $\beta(\omega^n)$. Suppose that $f$ is an interior mapping of $\beta(\omega^n)$ onto $\mathfrak T_n$. Since $\omega^n$ is open in $\beta(\omega^n)$, we have that $f(\omega^n)$ is an upset of $\mathfrak T_n$. 
Because 
$\bd_n$ is  
valid in $\omega^n$ (see \cite{Aba87} or \cref{thm: logic of ord}(\ref{lem item: AB finite rank}) below), 
\cref{lem: truth preserving ops}(\ref{lem item: interior image}) implies that 
$\bd_n$ is 
valid in 
$f(\omega^n)$. 
Therefore, $f(\omega^n)$ is an upset of $\mathfrak T_n$ of depth at most $n$, and 
so $f(\omega^n)\subseteq W\setminus\{r\}$ (because $d(r)=n+1$). 
Thus, 
\[
\omega^n = \omega^n\cap f^{-1}(W\setminus\{r\}) = (\omega^n\cap f^{-1}({\uparrow}w_1))\cup(\omega^n\cap f^{-1}(v)).
\]
Since $f^{-1}({\uparrow}w_1)$ and $f^{-1}(v)=f^{-1}({\uparrow}v)$ are disjoint and open in $\beta(\omega^n)$, it follows from the above equality 
that $\{\omega^n\cap f^{-1}({\uparrow}w_1),\omega^n\cap f^{-1}(v)\}$ is 
a clopen partition of $\omega^n$. 
Hence, $\ccc\left(\omega^n\cap f^{-1}({\uparrow}w_1)\right)$ and $\ccc\left(\omega^n\cap f^{-1}(v)\right)$ are disjoint by \cite[Cor.~3.6.4]{Eng89}. Because $f^{-1}({\uparrow}w_1)$ and $f^{-1}(v)$ are open in $\beta(\omega^n)$ and $\omega^n$ is dense in $\beta(\omega^n)$, \cref{lem: closure of open intersect dense is closure of open} gives
\[
\ccc\left(\omega^n\cap f^{-1}({\uparrow}w_1)\right) = \ccc(f^{-1}({\uparrow}w_1)) = f^{-1}\left({\downarrow}{\uparrow}w_1\right) = f^{-1}\left(\left\{r,w_1,\dots,w_n\right\}\right)
\]
and
\[
\ccc\left(\omega^n\cap f^{-1}(v)\right) = \ccc(f^{-1}(v)) = f^{-1}\left({\downarrow}v\right) = f^{-1}\left(\left\{r,v\right\}\right).
\]
Therefore, since $f$ is onto, we have the following contradiction:
\begin{eqnarray*}
\varnothing &=& \ccc\left(\omega^n\cap f^{-1}({\uparrow}w_1)\right) \cap \ccc\left(\omega^n\cap f^{-1}(v)\right) \\
&=& f^{-1}\left(\left\{r,w_1,\dots,w_n\right\}\right) \cap f^{-1}\left(\left\{r,v\right\}\right) = f^{-1}(r)\neq\varnothing.
\end{eqnarray*}
Thus, $\mathfrak T_n$ is not an interior image of $\beta(\omega^n)$.
\end{proof}

It follows from Lemmas~\ref{lem: s4.1 in L beta gamma} and~\ref{lem: omega^n is clopen in omega^n+1} that ${\sf S4.1}\subseteq\bigcap_{n=1}^\infty{\sf L}_n$. We show that 
the containment is strict. 

\begin{definition}\label{def: G and F_2}
Let $\mathfrak G$ be the {\sf S4.1}-frame depicted in Figure~\ref{fig: frame separating s4.1 and n}. 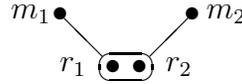
\begin{figure}[h]
\begin{picture}(50,40)(-25,-5)
\multiput(-5,5)(10,0){2}{\makebox(0,0){$\bullet$}}
\put(0,5){\oval(20,10)}
\put(-20,5){\makebox(0,0){$r_1$}}
\put(20,5){\makebox(0,0){$r_2$}}
\put(-25,25){\line(1,-1){16}}
\put(-36,25){\makebox(0,0){$m_1$}}
\put(25,25){\line(-1,-1){16}}
\put(38,25){\makebox(0,0){$m_2$}}
\multiput(-25,25)(50,0){2}{\makebox(0,0){$\bullet$}}
\end{picture}
\caption{The {\sf S4}-frame $\mathfrak G$.}
\label{fig: frame separating s4.1 and n}
\end{figure}
\end{definition}

\begin{lemma}\label{clm: NEVER an image}
$\mathfrak G$ 
is not an interior image of any open subspace of $\beta(\omega^n)$ for any $n$.
\end{lemma}

\begin{proof}
The proof is analogous to the proof of Lemma~\ref{clm: frame sepating n and n=1}(\ref{clm. item: is NOT image n}). 
Suppose that $\mathfrak G$ 
is an interior image of $\beta(\omega^n)$ for some $n$.  
Then $\left\{\omega^n\cap f^{-1}(m_1),\omega^n\cap f^{-1}(m_2)\right\}$ is a clopen partition of $\omega^n$, and hence
\begin{eqnarray*}
\varnothing &=& \ccc(\omega^n\cap f^{-1}(m_1)) \cap \ccc(\omega^n\cap f^{-1}(m_2)) = \ccc( f^{-1}(m_1)) \cap \ccc(f^{-1}(m_2)) \\&=& f^{-1}({\downarrow}m_1)\cap f^{-1}({\downarrow}m_2) = f^{-1}\left({\downarrow}m_1\cap{\downarrow}m_2\right)
= f^{-1}(\{r_1,r_2\}) \neq \varnothing.
\end{eqnarray*}
The obtained contradiction proves that $\mathfrak G$  
is not an interior image of $\beta(\omega^n)$ for any $n$.
\end{proof}

As a consequence of Lemma~\ref{clm: NEVER an image}, we obtain that ${\sf S4.1}\subset\bigcap_{n=1}^\infty{\sf L}_n$. This motivates the following definition.

\begin{definition}
Set ${\sf L}_\infty=\bigcap\nolimits_{n=1}^\infty{\sf L}_n$.
\end{definition}

We recall 
that with each finite rooted $\sf S4$-frame $\mathfrak F$ we may associate the formula $\chi_\mathfrak F$, known as the {\em Fine-Jankov formula} of $\mathfrak F$ (see, e.g., \cite[pp.~143--144]{BRV01}). For ease of reference, we recall Fine's result \cite[Lem.~1]{Fin74} and its generalization 
to the topological setting \cite[Lem.~3.5]{BBBM17}:

\begin{lemma}\label{lem: top vers of Fine}
Let $\F$ be a finite rooted {\sf S4}-frame, $X$ a topological space, and $\G$ an {\sf S4}-frame. 
\begin{enumerate}
\item 
$\G\vDash\chi_\F$ iff $\F$ is not a p-morphic image of any 
generated subframe of $
\G$.\label{lem item: FINE}
\item $X\vDash\chi_\mathfrak F$ iff $\mathfrak F$ is not an interior image of any open subspace of $X$. \label{lem item: FINE gen Top setting}
\end{enumerate}
\end{lemma}

\begin{theorem}\label{lem: L_n's chain}
$
{\sf S4.1}\subset 
{\sf L}_\infty \subset \cdots \subset {\sf L}_{n+1}\subset {\sf L}_n\subset \cdots\subset {\sf L_2}\subset {\sf L_1}={\sf S4.1.2}.
$ 
\end{theorem}

\begin{proof}
By \cite{BH09}, ${\sf L}_1={\sf S4.1.2}$. The inclusions 
follow from Lemmas~\ref{lem: s4.1 in L beta gamma} and~\ref{lem: omega^n is clopen in omega^n+1}. That the inclusions 
are strict follows from \cref{clm: frame sepating n and n=1}. Indeed, let $\chi_{\mathfrak T_n}$ be the Fine-Jankov formula of $\mathfrak T_n$. Because $\mathfrak T_n$ is an interior image of $\beta(\omega^{n+1})$, \cref{lem: top vers of Fine}(\ref{lem item: FINE gen Top setting}) yields that $\beta(\omega^{n+1})\not\vDash\chi_{\mathfrak T_n}$, so $\chi_{\mathfrak T_n}\not\in{\sf L}_{n+1}$. Since $\mathfrak T_n$ is not an interior image of $\beta(\omega^n)$, \cref{lem: F image of beta iff F image of open sub of beta} implies that $\mathfrak T_n$ is not an interior image of any open subspace of $\beta(\omega^n)$. Thus, \cref{lem: top vers of Fine}(\ref{lem item: FINE gen Top setting}) yields that $\beta(\omega^{n})\vDash\chi_{\mathfrak T_n}$, and hence $\chi_{\mathfrak T_n}\in{\sf L}_{n}$. Therefore, ${\sf L}_{n+1}\subset{\sf L}_n$. Consequently, ${\sf L}_\infty \subset {\sf L}_n$ for each $n$. Similarly, it follows from Lemmas~\ref{clm: NEVER an image} and~\ref{lem: top vers of Fine}(\ref{lem item: FINE gen Top setting}) that $\chi_{\mathfrak G
}\in{\sf L}_n$ for each $n\ge 1$, but $\chi_{\mathfrak G
}\not\in{\sf S4.1}$ (since $\mathfrak G$ is an {\sf S4.1}-frame). 
\end{proof}

\section{Logics arising as ${\sf L}(\beta(\gamma))$}
\label{sec: partial soln P2}

Having established that the logics ${\sf L}_n={\sf L}(\beta(\omega^n))$ form a decreasing chain in the interval $\left[{\sf S4.1},{\sf S4.1.2}\right]$ whose intersection ${\sf L}_\infty$ strictly contains {\sf S4.1}, we consider a more general question: what logics arise as ${\sf L}(\beta(\gamma))$ for an ordinal space $\gamma$? 
The following structural theorem is one of our main tools. Its proof utilizes the Cantor normal form of an ordinal, which we now recall (see, e.g., \cite[p.~153]{Sem71}).

\begin{definition}\label{def: CNF}
Each nonzero ordinal $\gamma$ can be uniquely written in the \emph{Cantor normal form} 
$
\gamma = \omega^{\alpha_k}n_k +\cdots +\omega^{\alpha_1}n_1,
$
where $0\ne k\in\omega$, 
each $n_i$ is nonzero and finite, and $0\le\alpha_1<\cdots<\alpha_k$ are ordinals.
\end{definition}

\begin{theorem}[Structural Theorem]\label{thm: strctr of CS noncompact ordinal}
For an infinite  
ordinal $\gamma$, we have that $\beta(\gamma)$ is homeomorphic to the disjoint union of a compact ordinal and the \CS\ compactification of a power of $\omega$.
\end{theorem}

\begin{proof}
If $\gamma$ is compact, 
then $\beta(\gamma)$ is homeomorphic to $\gamma$, which in turn is homeomorphic to $\gamma + 1$ (since $\gamma$ is infinite). Therefore, 
$\beta(\gamma)$ is homeomorphic to the disjoint union $\gamma\oplus\beta(\omega^0)$. 

Suppose that $\gamma$ is not compact, and let $\gamma = \omega^{\alpha_k}n_k +\cdots +\omega^{\alpha_1}n_1$ be in Cantor normal form. Then 
$\alpha_1\neq0$ because $\gamma$ is not compact.
The idea is to ``tear off'' one copy of the least occurring power of $\omega$ from ``the top'' of $\gamma$. Formally, write $\gamma$ as $\gamma' + \omega^{\alpha_1}$ by taking 
\[
\gamma' = \left\{\begin{array}{ll} 
0 & \text{if }n_1=1\text{ and }k=1,\\
\omega^{\alpha_k}n_k +\cdots +\omega^{\alpha_2}n_2 & \text{if }n_1=1\text{ and }k>1,\\
\omega^{\alpha_k}n_k +\cdots +\omega^{\alpha_2}n_2 + \omega^{\alpha_1}(n_1-1) & \text{if }n_1>1.
\end{array}
\right.
\]

Basic ordinal arithmetic yields that 
\[
\gamma = \gamma' +\omega^{\alpha_1} = \gamma' + (1+\omega^{\alpha_1}) = (\gamma'+1)+\omega^{\alpha_1}.
\] 
Since 
$\{\gamma'+1,\omega^{\alpha_1}\}$ is a clopen partition of $\gamma$ and $\gamma'+1$ is compact, $\{\gamma'+1,\ccc
(\omega^{\alpha_1})\}$ is a clopen partition of $\beta(\gamma)$. 
But $\gamma$ is a normal space, giving that $\ccc
(\omega^{\alpha_1})$ is homeomorphic to $\beta(\omega^{\alpha_1})$. 
Thus, $\beta(\gamma)$ is homeomorphic to the disjoint union $(\gamma'+1) \oplus \beta(\omega^{\alpha_1})$.
\end{proof}

We next recall the well-known Abashidze-Blass theorem. For this recall that the {\em Grzegorczyk logic} $\Grz$ is obtained from {\sf S4} by postulating 
the {\em Grzegorczyk axiom}
\[
\grz = \grzform.
\]
For each nonzero $n\in\omega$, let $\Grz_n$ be obtained from $\Grz$ by postulating $\bd_n$. We have that $\Grz=\bigcap_{n=1}^\infty \Grz_n$ and these logics form the following descending chain:
\[
\Grz\subset \cdots\subset\Grz_3\subset\Grz_2\subset\Grz_1.
\]

\begin{theorem}[Abashidze \cite{Aba87}, Blass \cite{Bla90}]
\label{thm: logic of ord}
Let $\gamma$ be a nonempty ordinal and $n\in\omega$ be nonzero. 
\begin{enumerate}
\item If $\omega^{n-1}< \gamma \le \omega^n$
, then ${\sf L}(\gamma)=\Grz_n$.
\label{lem item: AB finite rank}
\item If $\omega^\omega\le \gamma$, 
then ${\sf L}(\gamma)=\Grz$. 
\label{lem item: AB infinite rank}
\end{enumerate}
\end{theorem}
 
Intersecting logics appearing in \cref{thm: logic of ord} with those appearing in \cref{def: logics Ln}  
yields the family of logics depicted in Figure~\ref{fig: logics of beta gamma}. 
\begin{figure}[h]
\begin{center}
\begin{picture}(150,275)(35,-75)
\multiput(170,190)(-25,-25){5}{\multiput(0,0)(0,-25){4}{\makebox(0,0){$\bullet$}}}
\multiput(170,190)(-25,-25){5}{\line(0,-1){85}}
\multiput(170,95)(-25,-25){5}{\makebox(0,0){$\vdots$}}
\multiput(170,75)(-25,-25){5}{\makebox(0,0){$\bullet$}}
\multiput(170,190)(0,-25){4}{\line(-1,-1){110}}
\color{blue}
\put(25,-70){\makebox(0,0){$\bullet$}}
\color{black}
\multiput(50,70)(0,-25){4}{\multiput(0,0)(-3,-3){3}{\makebox(0,0){$\cdot$}}}
\multiput(50,-45)(-3,-3){3}{\makebox(0,0){$\cdot$}}
\put(170,75){\line(-1,-1){110}}
\put(10,-70){\makebox(0,0){${\sf L}_\infty$}}
\put(195,190){\makebox(0,0){$\Grz_1$}}
\put(195,165){\makebox(0,0){$\Grz_2$}}
\put(195,140){\makebox(0,0){$\Grz_3$}}
\put(195,115){\makebox(0,0){$\Grz_4$}}
\put(195,75){\makebox(0,0){$\Grz$}}
\put(145,180){\makebox(0,0){${\sf L}_1$}}
\put(120,155){\makebox(0,0){${\sf L}_2$}}
\put(95,130){\makebox(0,0){${\sf L}_3$}}
\put(70,105){\makebox(0,0){${\sf L}_4$}}
\put(175,40){\makebox(0,0){$\Grz\cap{\sf L}_1$}}
\put(150,15){\makebox(0,0){$\Grz\cap{\sf L}_2$}}
\put(125,-10){\makebox(0,0){$\Grz\cap {\sf L}_3$}}
\put(100,-35){\makebox(0,0){$\Grz\cap {\sf L}_4$}}
\put(30,170){$\Grz_3\cap{\sf L}_1$}
\put(70,165){\vector(3,-2){70}}
\color{magenta}
\put(120,115){\makebox(0,0){\textrm{$\bullet$}}}
\multiput(95,90)(0,-25){2}{\makebox(0,0){\textrm{$\bullet$}}}
\multiput(70,65)(0,-25){3}{\makebox(0,0){\textrm{$\bullet$}}}
\end{picture}
\end{center}
\caption{Logics arising as ${\sf L}(\beta(\gamma))$.}
\label{fig: logics of beta gamma}
\end{figure}
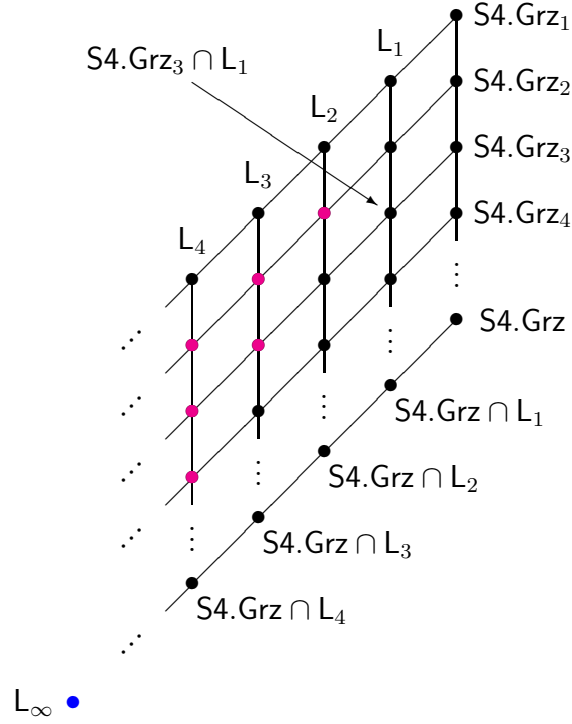

We next show that the logics corresponding to black bullets are of the form ${\sf L}(\beta(\gamma))$ for some ordinal $\gamma$, while the logics associated with magenta bullets do not arise in this fashion. Although we presently don't have a proof, we believe that 
${\sf L}_\infty$ also arises this way (see Conjecture~\ref{conj}).

\begin{theorem}\label{thm: most of our logics picked up by ordinal}
Each of the following logics arises as ${\sf L}(\beta(\gamma))$ for some ordinal $\gamma$.
\begin{enumerate}
\item ${\sf L}_m$ for each $0<m<\omega$. 
\item $\Grz$ and $\Grz_n$ for each $0<n<\omega$.
\item $\Grz\cap{\sf L}_m$ and $\Grz_n\cap {\sf L}_m$ for $0<m< n<\omega$. 
\end{enumerate}
All these logics are depicted by black bullets in Figure~\ref{fig: logics of beta gamma}. 
\end{theorem}

\begin{proof}
(1) We have ${\sf L}_m={\sf L}(\beta(\omega^m))$ for each $0<m<\omega$ by Definition~\ref{def: logics Ln}. 

(2) Let 
$0<n<\omega$. Since 
both $\omega^\omega+1$ and $\omega^{n-1}+1$ are compact, Theorem~\ref{thm: logic of ord} implies that 
\[
\Grz = {\sf L}(\omega^\omega+1)={\sf L}(\beta(\omega^\omega+1))
\]
and
\[
\Grz_n = {\sf L}(\omega^{n-1}+1)={\sf L}(\beta(\omega^{n-1}+1)).
\]

(3) 
By Theroem~\ref{thm: logic of ord}, Definition~\ref{def: logics Ln}, Lemma~\ref{lem: truth preserving ops}(\ref{lem item: disjoint union}), and
Theorem~\ref{thm: strctr of CS noncompact ordinal},  
\[
\Grz\cap{\sf L}_m = {\sf L}(\omega^\omega+1)\cap{\sf L}\left(\beta(\omega^m)\right) = {\sf L}\left((\omega^\omega+1)\oplus\beta(\omega^m)\right) = {\sf L}(\beta(\omega^\omega+\omega^m))
\]
and
\[
\Grz_n\cap{\sf L}_m = {\sf L}(\omega^{n-1}+1)\cap{\sf L}\left(\beta(\omega^m)\right) = {\sf L}(\left(\omega^{n-1}+1)\oplus\beta(\omega^m)\right) = {\sf L}(\beta(\omega^{n-1}+\omega^m))
\]
for each $0<m<n<\omega$.
\end{proof}

Notice that in the proof of Theorem~\ref{thm: most of our logics picked up by ordinal}, the ordinal giving rise to the logic of interest has $\alpha_1$ finite in its Cantor normal form. In such a case, 
we can determine ${\sf L}(\beta(\gamma))$ as shown in the next theorem. 

\begin{theorem}\label{thm: other direction for finite alpha_1}
Let $\gamma$ be a nonzero ordinal. 
If $\beta(\gamma)$ is an ordinal or $\alpha_1$ in the Cantor normal form of $\gamma$ is finite, then ${\sf L}(\beta(\gamma))$ is a logic appearing as a black bullet in Figure~\ref{fig: logics of beta gamma}.
\end{theorem}

\begin{proof} 
If $\beta(\gamma)$ is an ordinal, then Theorem~\ref{thm: logic of ord} yields that ${\sf L}(\beta(\gamma))$ appears as a black bullet in Figure~\ref{fig: logics of beta gamma}. Suppose that $\alpha_1$ is finite in the Cantor normal form
$
\gamma = \omega^{\alpha_k}n_k +\cdots +\omega^{\alpha_1}n_1,
$
where $0\ne k\in\omega$, each $n_i$ is nonzero and finite, and $0\le\alpha_1<\cdots<\alpha_k$ are ordinals. If $\alpha_1=0$, then $\gamma$ is compact, and hence $\beta(\gamma)=\gamma$ is an ordinal.
Thus, we may assume that $\alpha_1>0$. Write $\gamma = \gamma'+\omega^{\alpha_1}$ as in the proof of Theorem~\ref{thm: strctr of CS noncompact ordinal}. 

If $\gamma'=0$, then ${\sf L}(\beta(\gamma)) = {\sf L}(\beta(\omega^{\alpha_1}))={\sf L}_{\alpha_1}$ appears in Figure~\ref{fig: logics of beta gamma} by Definition~\ref{def: logics Ln}. Otherwise
Theorem~\ref{thm: strctr of CS noncompact ordinal}, Lemma~\ref{lem: truth preserving ops}(\ref{lem item: disjoint union}), and Definition~\ref{def: logics Ln}  imply that 
\[
{\sf L}(\beta(\gamma)) = {\sf L}\left((\gamma'+1)\oplus\beta(\omega^{\alpha_1})\right) = {\sf L}(\gamma'+1)\cap {\sf L}_{\alpha_1}.
\]
If $\alpha_k<\omega$, then ${\sf L}(\gamma'+1) = \Grz_{\alpha_k+1}$ and $\alpha_k+1>\alpha_k\ge\alpha_1$, so ${\sf L}(\beta(\gamma))=\Grz_{\alpha_k+1}\cap {\sf L}_{\alpha_1}$ appears as a black bullet in Figure~\ref{fig: logics of beta gamma}. 
If $\alpha_k\ge\omega$, then $k>1$ and $\gamma'+1>\gamma'\ge\omega^{\alpha_k}\ge \omega^\omega$, implying that ${\sf L}(\gamma'+1)=\Grz$. Therefore, ${\sf L}(\beta(\gamma))=\Grz\cap{\sf L}_{\alpha_1}$ again appears as a black bullet in Figure~\ref{fig: logics of beta gamma}.
\end{proof}

Putting Theorems~\ref{thm: most of our logics picked up by ordinal} and~\ref{thm: other direction for finite alpha_1} together yields 
the following partial solution of {\bf P2}: 

\begin{corollary}\label{thm: partial soln of P2}
Let {\sf L} be a logic above {\sf S4}. 
There exists a nonzero ordinal $\gamma$ such that ${\sf L}={\sf L}(\beta(\gamma))$ and $\alpha_1$ in the Cantor normal form of $\gamma$ is finite iff {\sf L} is one of the following logics: 
\begin{enumerate}
\item ${\sf L}_m$ for each $0<m<\omega$; 
\item $\Grz$ and $\Grz_n$ for each $0<n<\omega$;
\item $\Grz\cap{\sf L}_m$ and $\Grz_n\cap {\sf L}_m$ for $0<m< n<\omega$.
\end{enumerate}
\end{corollary} 

To answer {\bf P2} in full generality, we must consider ordinals $\gamma$ in which $\alpha_1$ in the Cantor normal form of $\gamma$ is infinite. The smallest such example is $\omega^\omega$. It is likely that
the situation mirrors the Abashidze-Blass Theorem in 
that ${\sf L}_\infty={\sf L}(\beta(\gamma))$ whenever $\gamma\ge \omega^\omega$ (see \cref{conj}). 

\section{Forbidden configurations for $\beta(\omega^2)$}
\label{sec: beta omega squared and 3 forbid frames}

In order to understand relational semantics of the 
logics ${\sf L}_n={\sf L}(\beta(\omega^n))$, we need to determine which $\sf S4$-frames are permissible and which are forbidden for $\beta(\omega^n)$ in the following sense.

\begin{definition}
\label{def: forbidden frame}

Let $\F$ be a finite rooted {\sf S4}-frame.
\begin{enumerate}
\item 
Call $\mathfrak F$ a \emph{permissible configuration} for a space $X$ if 
$\F$ is an interior image of an open subspace of $X$. 
Otherwise $\F$ is a \emph{forbidden configuration} for $X$. 

\item Call $\F$ a \emph{permissible configuration} for a class $\mathcal K$ of spaces if $\F$ is permissible for some $X\in \mathcal K$. Otherwise $\F$ is a 
\emph{forbidden configuration} for $\mathcal K$.
\end{enumerate}
\end{definition}

Since each {\sf S4}-frame 
can be viewed as a topological space, we will also speak about permissible/forbidden configurations for a frame or a class of frames. 

In this section we show that the following three frames are forbidden configurations for $\beta(\omega^2)$. 

\begin{definition}\label{def: four forbidden frames}
Let $\mathfrak F_1=(W_1,\le)$, $\mathfrak F_2=(W_2,\le)$, and $\mathfrak F_3=(W_3,\le)$ be the frames depicted in Figure~\ref{fig: forbidden frames}. 
\begin{figure}[h]
\begin{center}
\begin{picture}(200,100)(0,-33)
\put(25,-23){\makebox(0,0){$\mathfrak F_1$}}
\put(100,-23){\makebox(0,0){$\mathfrak F_2$}}
\put(175,-23){\makebox(0,0){$\mathfrak F_3$}}
\put(96,-5){\makebox(0,0){$r_1$}}
\put(109,-5){\makebox(0,0){$r_2$}}
\put(100,10){\oval(20,10)}
\put(95,10){\makebox(0,0){$\bullet$}}
\put(105,10){\makebox(0,0){$\bullet$}}
\put(75,35){\line(1,-1){20}}
\put(125,35){\line(-1,-1){20}}
\put(75,35){\makebox(0,0){$\bullet$}}
\put(125,35){\makebox(0,0){$\bullet$}}
\multiput(75,45)(75,0){1}{\makebox(0,0){$m_1$}}
\put(125,45){\makebox(0,0){$m_2$}}
\multiput(25,10)(150,0){2}{\makebox(0,0){$\bullet$}}
\put(0,35){\makebox(0,0){$\bullet$}}
\multiput(150,35)(75,0){1}{\multiput(0,0)(50,0){2}{\makebox(0,0){$\bullet$}}}
\multiput(200,60)(150,0){1}{\makebox(0,0){$\bullet$}}
\put(50,35){\makebox(0,0){$\bullet$}}
\multiput(25,10)(150,0){2}{\line(-1,1){25}}
\put(25,10){\line(1,1){25}}
\multiput(175,10)(75,0){1}{\line(1,1){25}}
\multiput(200,35)(75,0){1}{\line(0,1){25}}
\multiput(0,60)(150,0){2}{\makebox(0,0){$\bullet$}}
\multiput(0,60)(150,0){2}{\line(0,-1){25}}
\put(150,35){\line(2,1){50}}
\multiput(150,60)(75,0){1}{\line(2,-1){50}}
\multiput(25,-2)(150,0){2}{\makebox(0,0){$r$}}
\multiput(202,26)(75,0){1}{\makebox(0,0){$v_2$}}
\multiput(-2,26)(150,0){1}{\makebox(0,0){$v$}}\multiput(150,26)(150,0){1}{\makebox(0,0){$v_1$}}
\multiput(50,45)(150,25){2}{\makebox(0,0){$m_2$}}
\multiput(0,70)(150,0){2}{\makebox(0,0){$m_1$}}
\end{picture}
\end{center}
\caption{The frames $\mathfrak F_1$, $\mathfrak F_2$, and $\mathfrak F_3$.}
\label{fig: forbidden frames}
\end{figure}
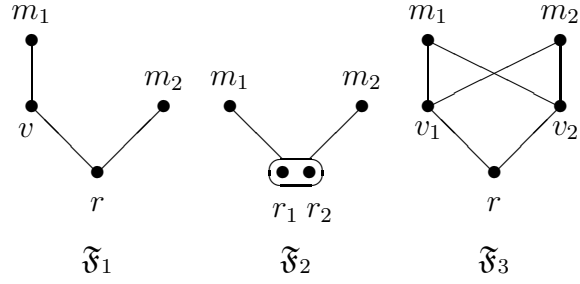
\end{definition}

\begin{remark}

Observe that $\mathfrak F_1$ is isomorphic to $\mathfrak T_2$ of \cref{def: tall lop-sided two fork} and $\mathfrak F_2$ is $\mathfrak G$ of \cref{def: G and F_2}. Despite being previously defined, it is  convenient to rename these frames to be packaged together in \cref{def: four forbidden frames}.

\end{remark}

We have that both $\mathfrak F_1$ and $\mathfrak F_2$ are forbidden configurations for $\beta(\omega^2)$ by \cref{clm: frame sepating n and n=1,clm: NEVER an image}, respectively. 
Thus, the main work is in showing that $\mathfrak F_3$ is also a forbidden configuration for $\beta(\omega^2)$. To do so, 
recall that we identify the ordinal space $\omega^2$ with the product space $(\omega+1)\times\omega$. Under this identification, the sets $A$ of isolated points and $B$ of limit points of $\omega^2$ are the products $\omega\times \omega$ and $\{\omega\}\times\omega$, respectively.
Clearly $A$ is open and $B$ is closed in $\omega^2$. 
Moreover, $B$ is homeomorphic to $\omega$, so 
the closure $\ccc(B)$ in $\beta(\omega^2)$ is homeomorphic to $\beta(\omega)$ (\cite[Cor.~3.6.8]{Eng89}) and 
$B^*:=\ccc(B)\setminus B$ 
to $\omega^*$. 
For $n\in\omega$, let $H_n=(\omega+1)\times\{n\}$ and $V_n=\{n\}\times\omega$. Then both $H_n$ and $V_n$ are clopen subspaces of $\omega^2$ such that $H_n$ is homeomorphic to $\omega+1$ and $V_n$ 
to $\omega$. 

\begin{figure}[h]
\begin{center}
\begin{picture}(127.5,107.5)(0,7.5)
\multiput(127.5,22.5)(-25,0){2}{{\line(0,1){92.5}}} 
\multiput(127.5,22.5)(0,92.5){2}{{\line(-1,0){25}}}
\put(115, 107.5){\makebox(0,0){$B$}}
\multiput(100,22.5)(-80,0){2}{{\line(0,1){92.5}}} 
\multiput(100,22.5)(0,92.5){2}{{\line(-1,0){80}}}
\put(60, 107.5){\makebox(0,0){$A$}}
\multiput(52.5,7.5)(-15,0){2}{{\line(0,1){92.5}}} 
\multiput(52.5,7.5)(0,92.5){2}{{\line(-1,0){15}}}
\put(45.5,15){\makebox(0,0){$V_1$}}
\multiput(0,52.5)(122.5,0){2}{{\line(0,1){15}}} 
\multiput(0,52.5)(0,15){2}{{\line(1,0){122.5}}}
\put(10,60){\makebox(0,0){$H_2$}}
\multiput(115,30)(0,15){4}{{\makebox(0,0){$\bullet$}}}
\put(115,92.5){\makebox(0,0){$\vdots$}}
\multiput(30,30)(15,0){4}{\multiput(0,0)(0,15){4}{\makebox(0,0){$\bullet$}}}
\put(89,60){\makebox(0,0){$\cdots$}}
\multiput(30,92.5)(15,0){4}{\makebox(0,0){$\vdots$}}
\end{picture}
\end{center}
\caption{Useful subsets of $\omega^2$.}
\label{fig: notations for subsets of omega squared}
\end{figure}
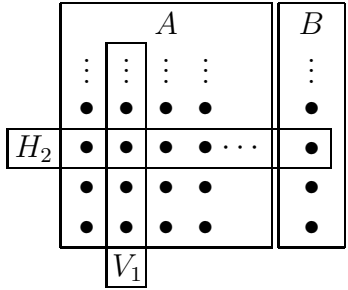

We are ready to prove the main result of this section.

\begin{theorem}\label{lem: F3 is not an interior image of beta X}
The frame $\mathfrak F_3$ is a forbidden configuration 
for $\beta(\omega^2)$. 
\end{theorem}

\begin{proof}
By \cref{lem: F image of beta iff F image of open sub of beta}, it is sufficient to show that $\mathfrak F_3$ is not an interior image of $\beta(\omega^2)$. 
Suppose on the contrary that $f$ is an interior mapping of $\beta(\omega^2)$ onto $\mathfrak F_3$. Let $U_i=f^{-1}(m_i)$ for $i=1,2$. 
As $f$ is interior and $\mathrm{max}(\mathfrak F_3)=\{m_1,m_2\}$, we have that $U_1$ and $U_2$ are disjoint and $U_1\cup U_2$ is a dense open subset of $\beta(\omega^2)$.

\begin{claim}\label{clm: complement of closure of limits of X is in U union V}
$\beta(\omega^2)\setminus\ccc (B)\subseteq U_1\cup U_2$.
\end{claim}

\begin{proof}
Suppose not and let $x\in\beta(\omega^2)\setminus\ccc (B)$ be such that $x\not\in U_1\cup U_2$. There is a clopen subset $C$ of $\beta(\omega^2)$ such that $x\in C\subseteq \beta(\omega^2)\setminus\ccc (B)$, yielding that $C\cap B=\varnothing$. Thus, 
\[
C\cap \omega^2 = C\cap(A \cup B) =C\cap A.
\]
Because $A$ is the set of isolated points of $\beta(\omega^2)$, we have 
\[
A\subseteq f^{-1}(\mathrm{max}(\mathfrak F_3)) = f^{-1}(\{m_1,m_2\}) = U_1\cup U_2.
\]
Therefore, 
\begin{eqnarray*}
C\cap \omega^2 &=& C\cap A = C\cap A \cap(U_1\cup U_2) = C\cap \omega^2 \cap (U_1\cup U_2) \\
&=& (C\cap \omega^2\cap U_1)\cup(C\cap \omega^2\cap U_2).
\end{eqnarray*}
Since $U_1$ and $U_2$ are disjoint open subsets of $\beta(\omega^2)$, we have that $C\cap \omega^2\cap U_1$ and $C\cap \omega^2\cap U_2$ are disjoint open subsets of $C\cap \omega^2$. Because $C\cap \omega^2 = (C\cap \omega^2\cap U_1)\cup(C\cap \omega^2\cap U_2)$, both $C\cap \omega^2\cap U_1$ and $C\cap \omega^2\cap U_2$ are clopen in $C\cap \omega^2$. Since $C$, $C\cap U_1$, and $C\cap U_2$ are open in $\beta(\omega^2)$, it follows from Lemma~\ref{lem: closure of open intersect dense is closure of open} that 
\begin{eqnarray*}
C&=&\ccc (C) = \ccc(C\cap \omega^2)=\ccc\left((C\cap \omega^2\cap U_1)\cup(C\cap \omega^2\cap U_2)\right)\\&=&\ccc(C\cap \omega^2\cap U_1)\cup\ccc(C\cap \omega^2\cap U_2)=\ccc(C\cap U_1)\cup\ccc(C\cap U_2).
\end{eqnarray*}
Because $C\cap \omega^2$ is closed in $\omega^2$, we have that $C=\ccc(C\cap \omega^2)$ is homeomorphic to $\beta(C\cap \omega^2)$ (\cite[Cor.~3.6.8]{Eng89}). Since $C\cap \omega^2$ is normal, 
$\ccc(C\cap \omega^2\cap U_1)$ and $\ccc(C\cap \omega^2\cap U_2)$ are disjoint (\cite[Cor.~3.6.4]{Eng89}), giving that $\ccc(C\cap U_1)$ and $\ccc(C\cap U_2)$ are disjoint (\cref{lem: closure of open intersect dense is closure of open}). Because $C=\ccc(C\cap U_1)\cup\ccc(C\cap U_2)$, we have that $\ccc(C\cap U_1)$ and $\ccc(C\cap U_2)$ are clopen in $C$, and hence clopen in $\beta(\omega^2)$. Since $x\in C$, either $x\in\ccc(C\cap U_1)$ or $x\in\ccc(C\cap U_2)$.

Suppose $x\in\ccc(C\cap U_i)$ for some $i\in\{1,2\}$ and let $j\in\{1,2\}\setminus\{i\}$. Since $x\not\in U_1\cup U_2$, we have that $f(x)\in\{r,v_1,v_2\}$, which implies that $m_j\in {\uparrow}f(x)$. Because $f$ is interior, $f(\ccc(C\cap U_i))$ is an upset of $\mathfrak F_3$ that contains $f(x)$, and so $m_j\in f(\ccc(C\cap U_i))$. This yields that
\[
\varnothing \neq \ccc(C\cap U_i)\cap U_j = \ccc(C\cap U_i)\cap C\cap U_j \subseteq \ccc(C\cap U_i)\cap \ccc(C\cap U_j) = \varnothing,
\]
which is a contradiction. 
\end{proof}

Set $B_i=U_i\cap B$ for $i=1,2$.

\begin{claim}\label{clm: closure of U intersect limits of X stays in U}
$\ccc(B_i)\subseteq U_i$ for each $i=1,2$.
\end{claim}

\begin{proof}
Let $x\in\ccc(B_i)$ for some $i\in\{1,2\}$ and $j\in\{1,2\}\setminus\{i\}$. Set $N=\{k\in\omega \mid \langle\omega,k\rangle\in B_i\}$. For each $k\in N$, let $C_k$ be a clopen subset of $\beta(\omega^2)$ such that $\langle\omega,k\rangle\in C_k\subseteq U_i\cap H_k$. Set $C=\bigcup\nolimits_{k\in N}C_k$. Then $C$ is an open subset of $\omega^2$ because each $C_k$ is open in $\omega^2$. Since $H_k$ is open in $\omega^2$ 
and $H_k\setminus C_k$ is open in $\omega^2$ for each $k\in N$, we have that 
\[
\omega^2\setminus C = \left(\bigcup\nolimits_{k\in\omega\setminus N}H_k\right)\cup\left(\bigcup\nolimits_{k\in N}H_k\setminus C_k\right)
\]
is open in $\omega^2$. Thus, $C$ is a clopen subset of $\omega^2$ such that $B_i\subseteq C \subseteq U_i$. By \cite[Cor.~3.6.5]{Eng89}, $\ccc (C)$ is a clopen subset of $\beta(\omega^2)$ such that $x\in\ccc(B_i)\subseteq\ccc (C)$. Because $f$ is an onto interior map, 
$f(\ccc (C))$ is a nonempty upset in $\mathfrak F_3$ such that 
\[
f(\ccc (C)) \subseteq f(\ccc (U_i)) = f(\ccc(f^{-1}(m_i))) = f(f^{-1}({\downarrow}m_i)) = {\downarrow}m_i.
\]
Since $m_j\not\in {\downarrow}
m_i$, we have
$m_j\not\in f(\ccc(C))$. 
But the only nonempty upset of $\mathfrak F_3$ not containing $m_j$ is $\{m_i\}$. Thus, 
$f(x)\in f(\ccc(C)) = \{m_i\}$, which implies that $x\in f^{-1}(m_i)=U_i$. Hence, $\ccc(B_i)\subseteq U_i$. \end{proof}

Let $x\in\beta(\omega^2)$ be such that $f(x)=r$. Restricting the domain of $f$ to 
$\omega^2$ 
and the codomain 
to 
$f(\omega^2)$ 
yields an interior mapping of $\omega^2$ onto 
$f(\omega^2)$. By \cref{thm: logic of ord}, 
$\omega^2\vDash \bd_2$, and hence $f(\omega^2)\vDash\bd_2$ by \cref{lem: truth preserving ops}(\ref{lem item: interior image}). Thus,
$f(\omega^2)$ is 
of depth $\le 2$ (see, e.g., \cite[Prop.~3.44]{CZ97}), so $f(\omega^2)\subseteq \{m_1,m_2,v_1,v_2\}$, and hence $x\not\in \omega^2$. Moreover, 
$x\not\in U_1\cup U_2$, 
so $x\in\ccc(B)$ by Claim~\ref{clm: complement of closure of limits of X is in U union V}, and thus $x\in\ccc(B)\setminus B$. 
By Claim~\ref{clm: closure of U intersect limits of X stays in U}, 
$x\not\in\ccc(B_1)\cup\ccc(B_2)=\ccc(B_1\cup B_2)$. Therefore, there is a clopen subset $C$ of $\beta(\omega^2)$ such that $x\in C\subseteq\beta(\omega^2)\setminus\ccc(B_1\cup B_2)\subseteq\beta(\omega^2)\setminus(B_1\cup B_2)$. Because $C$ is an open neighborhood of $x$, it follows 
that $C\cap B\neq\varnothing$. 
Let $y\in C\cap B$. Then $y\in \omega^2$, giving that $f(y)\neq r$. If $f(y)\in\{m_1,m_2\}$, then $y\in (U_1\cup U_2)\cap B=B_1\cup B_2$, which contradicts 
$C\subseteq\beta(\omega^2)\setminus(B_1\cup B_2)$. Therefore, $C\cap B\subseteq f^{-1}(\{v_1,v_2\})$.
Let $C_i=C\cap B\cap f^{-1}(v_i)$ for $i=1,2$. 
Then 
\begin{eqnarray*}
C\cap B&=& C\cap B\cap f^{-1}(\{v_1,v_2\}) = C\cap B\cap \left(f^{-1}(v_1)\cup f^{-1}(v_2)\right) \\
&=& \left(C\cap B\cap f^{-1}(v_1)\right)\cup \left(C\cap B\cap  f^{-1}(v_2)\right) = C_1\cup C_2.
\end{eqnarray*}
Let $i\in\{1,2\}$. Since $f^{-1}(r)$, $\omega^2$ and $f^{-1}(v_i)$, $A$ are disjoint, we have
\begin{eqnarray*}
C\cap f^{-1}({\downarrow}v_i) \cap \omega^2 &=& C \cap \left(f^{-1}(r)\cup f^{-1}(v_i)\right)\cap \omega^2 \\
&=& C \cap \left((f^{-1}(r)\cap \omega^2)\cup (f^{-1}(v_i)\cap \omega^2)\right) \\
&=& C \cap 
(f^{-1}(v_i)\cap \left[A \cup B \right])\\
&=& C\cap \left(\left[f^{-1}(v_i)\cap  A \right] \cup \left[f^{-1}(v_i)\cap B \right]\right) \\
&=& C\cap \left(
B\cap f^{-1}(v_i) 
\right) = C_i.
\end{eqnarray*}
Therefore, $C_i$ is closed in $\omega^2$ because $C\cap f^{-1}({\downarrow}v_i)$ is closed in $\beta(\omega^2)$.
Since $C_1$ and $C_2$ are disjoint closed subsets of $\omega^2$, \cite[Cor.~3.6.4]{Eng89} yields that $\ccc (C_1)\cap \ccc (C_2)=\varnothing$. 
Because 
$x\in C\cap \ccc (B) = \ccc(C\cap B)
$, we have that  
$x\in\ccc(C\cap B)
= \ccc (C_1) \cup \ccc(C_2)$. Since $\ccc (C_1)$ and $\ccc (C_2)$ are disjoint, $x$ belongs to
exactly one of $\ccc (C_1)$ or $\ccc (C_2)$.

Suppose that $x\in\ccc (C_i)$ for some $i\in\{1,2\}$ and let $j\in\{1,2\}\setminus\{i\}$. Then $x\not\in\ccc (C_j)$, so $x\in\beta(\omega^2)\setminus \ccc (C_j)$. Thus, there is a clopen neighborhood $F$ of $x$ 
such that $F\subseteq\beta(\omega^2)\setminus\ccc (C_j)$, so $F\cap \ccc (C_j) = \varnothing$. Noting that $F\cap C$ is a clopen neighborhood of $x$ such that $F\cap C \cap \ccc (C_j) = \varnothing$ and $F\cap C \cap \ccc (C_i)\neq\varnothing$, we may assume 
that $C_j=\varnothing$, and hence that $C\cap B = C_i$. 

Since $r\in f(C)$ and $C$ is clopen, we have that $f(C)=W$.
Therefore, $v_j\in f(C)$, so there is $z\in C$ such that $f(z)=v_j$. By the definition of $U_1$ and $U_2$, we have that $z\not\in U_1\cup U_2$, implying that $z\in\ccc (B)$ by Claim~\ref{clm: complement of closure of limits of X is in U union V}. Clearly $z\not\in B\setminus C$ as $z\in C$, and $z\not \in C_i$ because $f(z)=v_j\neq v_i$. Thus, $z\not\in (B\setminus C) \cup C_i = (B\setminus C)\cup (C\cap B) = B$, yielding that $z\in \ccc(B)\setminus B$. Since $f$ is interior, 
$f^{-1}({\uparrow}v_j)\cap C$ is an open neighborhood of $z$ in $\beta(\omega^2)$. Thus, there is a clopen subset $D$ of $\beta(\omega^2)$ such that $z\in D\subseteq f^{-1}({\uparrow}v_j)\cap C$. Because $z\in\ccc (B)$, we have that $D\cap C_i = D\cap C \cap B = D\cap B\neq\varnothing$. Taking $y\in D\cap C_i$ yields the contradiction that $v_i=f(y)\in {\uparrow}v_j$, completing the proof.
\end{proof}

\section{Roaches}
\label{roaches section}

In this section we describe the class of rooted {\sf S4.1}-frames that is defined by the three forbidden configurations for $\beta(\omega^2)$. 
Because of their shape, we term the frames in this class $2$-roaches. These frames readily generalize to $n$-roaches (see the Introduction). We believe that the logic of the class of $n$-roaches is exactly the logic of $\beta(\omega^n)$. While this remains a conjecture (see \cref{conj: big conj for L_n}), the main result of this paper verifies the conjecture for $n=2$ (we emphasize again that our proof uses CH). 
 
\begin{convention}
From now on, all frames are assumed to be finite rooted ${\sf S4.1}$-frames. 
\end{convention}

\begin{definition}\label{def: 3 roach}
Let $\mathfrak F=(W,\le)$ be a frame
and $n\in\omega$ nonzero. 
\begin{enumerate} 
\item Call $\mathfrak F$ an \emph{$n$-roach} if there is a \emph{splitting point} $s\in W$ 
such that
\begin{enumerate}
\item $d(s)\le n$;
\item ${\uparrow}s$ is partially ordered (by the restriction of $\le$); \item for each $w\in W$, there is $t_w\in {\uparrow} s$ such that ${\uparrow} w\cap {\uparrow} s = {\uparrow} t_w$. 
\end{enumerate}

\item Let $\mathcal R_n$ be the class of all $n$-roaches and let ${\sf L}(\mathcal R_n)$ be the logic of $\mathcal R_n$.

\item Call $\mathfrak F$ a \emph{roach} if $\mathfrak F$ is an $n$-roach for some nonzero $n\in\omega$. 

\item Let $\mathcal R_\infty$ be the the class of all roaches and let ${\sf L}(\mathcal R_\infty)$ be the logic of $\mathcal R_\infty$.
\end{enumerate}

\end{definition}

\cref{fig: n-roach} in the Introduction depicts an $n$-roach and \cref{fig2} below a 2-roach. In both cases, the head $\{s\}$, the body ${\downarrow}s\setminus\{s\}$, and the antennae ${\uparrow}s\setminus\{s\}$ are colored in magenta. 

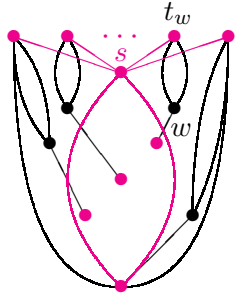
\begin{figure}[h]
\begin{center}
\begin{picture}(150,100)(-75,0)
\setlength{\unitlength}{0.19mm}
\qbezier(0,0)(-75,0)(-75,175)
\qbezier(0,0)(75,0)(75,175)
\multiput(-50,100)(100,0){1}{\makebox(0,0){$\bullet$}}
\put(50,50){\makebox(0,0){$\bullet$}}
\multiput(-37.5,125)(75,0){2}{\makebox(0,0){$\bullet$}}
\multiput(0,0)(75,0){2}{\qbezier(-37.5,125)(-55,150)(-37.5,175)}
\multiput(0,0)(75,0){2}{\qbezier(-37.5,125)(-20,150)(-37.5,175)}
\qbezier(-50,100)(-75,125)(-75,175)
\qbezier(-50,100)(-50,150)(-75,175)
\qbezier(50,50)(75,100)(75,175)
\qbezier(50,50)(50,130)(75,175)
\put(-25,50){\line(-1,2){25}}
\put(0,75){\line(-3,4){37.5}}
\put(25,100){\line(1,2){12.5}}
\put(0,0){\line(1,1){50}}
\put(42,110){\makebox(0,0){\small $w$}}
\put(40,192){\makebox(0,0){\small $t_w$}}
\color{magenta}
\multiput(-25,50)(25,25){3}{\makebox(0,0){$\bullet$}}
\multiput(0,0)(0,150){2}{\makebox(0,0){$\bullet$}}
\multiput(-75,175)(37.5,0){2}{\multiput(0,0)(112.5,0){2}{\makebox(0,0){$\bullet$}}}
\put(1,175){\makebox(0,0){$\cdots$}}
\put(0,163){\makebox(0,0){\small$s$}}
\qbezier(0,0)(-75,75)(0,150)
\qbezier(0,0)(75,75)(0,150)
\put(0,150){\line(-3,1){75}}
\put(0,150){\line(3,1){75}}
\put(0,150){\line(-3,2){37.5}}
\put(0,150){\line(3,2){37.5}}
\color{black}
\end{picture}
\end{center}
\caption{A $2$-roach.}
\label{fig2}
\end{figure}

\begin{remark} \label{remark about n-roaches}
We offer some intuition connecting the definition of an $n$-roach to the pursuit of the relational semantics for ${\sf L}_n$.  
Recall that finite interior images of $\omega^n$ are exactly posets of depth at most $n$. Thus, how the partially ordered upset ${\uparrow}s$ of a splitting point 
``sits inside'' an $n$-roach is a frame-theoretic analogue of how $\omega^n$ ``sits inside'' $\beta(\omega^n)$. 
\end{remark}

\begin{theorem}\label{n-roaches and logics}
\begin{enumerate}
\item[]

\item A frame $\F$ is a $1$-roach iff $\F$ is an {\sf S4.1.2}-frame. 
\label{item: L_1=L(R_1) sorta}

\item For $m,n\in\omega$, if $1\le n<m$ then $\mathcal R_n\subset \mathcal R_m$.
\label{R_n form strictly increasing chain}

\item ${\sf S4.1}\subset{\sf L}(\mathcal R_\infty)\subset \cdots \subset {\sf L}(\mathcal R_{n+1})\subset{\sf L}(\mathcal R_n)\subset \cdots \subset {\sf L}(\mathcal R_2)\subset{\sf L}(\mathcal R_1)={\sf S4.1.2}$.
\label{L(R_n) chain}

\end{enumerate}
\end{theorem}
 
\begin{proof}
(1) Suppose $\F\in\mathcal R_1$. Then there is a splitting point $s\in W$ of depth $1$, and hence ${\uparrow}s=\{s\}$ as ${\uparrow}s$ is partially ordered. For $w\in W$, we have $t_w\in{\uparrow}s$ and  
${\uparrow}w\cap {\uparrow}s={\uparrow}t_w={\uparrow}s$. Thus, $s$ is the unique maximal point of $\mathfrak F$. Consequently,  
$\F$ is an {\sf S4.1.2}-frame. 
Conversely, if $\F$ is an {\sf S4.1.2}-frame, the unique maximal point 
of $\F$ is a splitting point, yielding that $\F\in\mathcal R_1$. 

(2) Let $m,n\in\omega$ be such that $1\le n<m$. That $\mathcal R_n\subseteq\mathcal R_m$ is an immediate consequence of \cref{def: 3 roach}. To see that the containment is strict, consider 
$\mathfrak T_{n+1}$ (\cref{{def: tall lop-sided two fork}}). Notice that the root $r$ is the only candidate for a splitting point of $\mathfrak T_{n+1}$ because ${\uparrow}w_1\cap{\uparrow}v=\varnothing$. Since $n<d(r)=n+1\le m$, we have that 
$\mathfrak T_{n+1}\in\mathcal R_{m}\setminus\mathcal R_n$ as $\mathfrak T_{n+1}$ is partially ordered. 

(3) 
The right most equality follows from (\ref{item: L_1=L(R_1) sorta}) and \cref{thm: reln complnss s4 1 and 2}. Therefore, by (\ref{R_n form strictly increasing chain}), we only need to address the left most inclusion. Since each roach is an {\sf S4.1}-frame, we have ${\sf S4.1}\subseteq{\sf L}(\mathcal R_\infty)$. To see the inclusion is strict, we point out that the frame $\mathfrak G$ in \cref{def: G and F_2} 
is not a roach since the only candidates for a splitting point are $m_1$ and $m_2$, but there is no $t\in {\uparrow}m_i$ such that $m_j\le t$ for distinct $i,j\in\{1,2\}$. Because $\mathfrak G$ is an {\sf S4.1}-frame, the 
inclusion is strict.
\end{proof}

We now focus on establishing some basic facts about $2$-roaches, whose generalizations to $n\ge2$ are discussed at the end of the section. We begin with a recursive characterization of $2$-roaches which we employ frequently and without explicit reference.

\begin{lemma}\label{simple char of 2-roaches}
For a frame $\F$, 
we have that $\F\in\mathcal R_2$ iff 
\[
\exists s\in W \left(
{\uparrow}s=\{s\}\cup \mathrm{max}(\mathfrak F) 
\wedge  
\left(\forall w\in W\setminus {\downarrow}s\right)\left({\uparrow}w\in\mathcal R_1
\right)
\right).
\tag{\textrm{\dag}}\label{Cond}
\]
\end{lemma}

\begin{proof}
Let $\mathfrak F$ be a $2$-roach and $s\in W$ a splitting point of $\mathfrak F$. 
That ${\uparrow}s \subseteq \{s\}\cup\max(\F)$ is obvious since $d(s)\le 2$ and ${\uparrow}s$ is partially ordered. For the reverse inclusion,
let $w\in\max(\mathfrak F)$. Then
$t_w\in{\uparrow}s$ is such that ${\uparrow}w\cap{\uparrow}s = {\uparrow} t_w$, giving that $w\le t_w$. Since $w\in\max(\mathfrak F)$, we have
$w=t_w\in{\uparrow}s$. 

We next show 
that ${\uparrow}w\in\mathcal R_1$ 
for each $w\in W\setminus{\downarrow} s$. If not, then \cref{n-roaches and logics}(\ref{item: L_1=L(R_1) sorta}) implies ${\uparrow}w\not\vDash{\sf S4.1.2}$, giving that there are $w\in W\setminus{\downarrow}s$ and distinct
$m_1, m_2\in{\uparrow}w\cap\max(\mathfrak F)$.
Since $\mathfrak F\in\mathcal R_2$, 
there is $t_w\in{\uparrow}s$ such that $m_1,m_2\in{\uparrow}w\cap {\uparrow}s = {\uparrow}t_w$, implying that $d(t_w)\neq 1$ because $m_1\neq m_2$. Therefore, $t_w=s$, contradicting that $w\not\in{\downarrow} s$. 
Thus, ${\uparrow}w$ has a unique maximal point, and hence is an {\sf S4.1.2}-frame. Therefore, (\ref{Cond}) holds.

Conversely, suppose (\ref{Cond}) holds and let $s\in W$ be as in (\ref{Cond}).  
We show that $s$ is a splitting point realizing $\mathfrak F$ as a $2$-roach. Because $\mathfrak F\vDash{\sf S4.1}$ and ${\uparrow}s = \{s\}\cup\max(\mathfrak F)$, it follows that ${\uparrow}s$ is partially ordered and $d(s)\le 2$. Let $w\in W$. If $w\le s$, then taking $t_w=s\in{\uparrow}s$ yields ${\uparrow}w\cap{\uparrow}s={\uparrow}s={\uparrow}t_w$. Suppose $w\not\le s$. Then $w\in W\setminus{\downarrow}s$, which yields that ${\uparrow}w\in\mathcal R_1$ by (\ref{Cond}). Thus, ${\uparrow}w$ is an {\sf S4.1.2}-frame by \cref{n-roaches and logics}(\ref{item: L_1=L(R_1) sorta}). Let $t_w$ be the unique maximal point of ${\uparrow}w$. As ${\uparrow}w$ is a generated subframe, we have that $t_w\in\max({\uparrow}w)\subseteq\max(\F)\subseteq{\uparrow}s$. Clearly ${\uparrow}t_w=\{t_w\}\subseteq{\uparrow}w\cap{\uparrow}s$. Let $v\in{\uparrow}w\cap{\uparrow}s$. If $v=s$, then $w\le s$, a contradiction. Therefore, 
$v\in\max(\F)$. Because $t_w$ is the unique maximal point in ${\uparrow}w$, 
we conclude that $v=t_w$, and hence 
${\uparrow}w\cap {\uparrow}s={\uparrow}t_w$. Thus, 
$\mathfrak F$ is a $2$-roach. 
\end{proof}

\begin{remark}
Let $\F=(W,\le)$ be a $1$-roach. By \cref{n-roaches and logics}(\ref{item: L_1=L(R_1) sorta}), $\F$ is an {\sf S4.1.2}-frame. Thus, its unique maximal point is a splitting point realizing $\F\in\mathcal R_1$, and it is the unique splitting point of depth 1. On the other hand, we also have that $\F\in\mathcal R_2$ by \cref{n-roaches and logics}(\ref{R_n form strictly increasing chain}). Suppose there is $w\in W$ such that $d(w)=2$ and $C_w=\{w\}$. Then $w$ is also a splitting point realizing $\F\in\mathcal R_2$ (but $w$ does not realize $\F\in\mathcal R_1$.) Therefore, a splitting point of a $2$-roach need not be unique. 
\end{remark}

\begin{lemma}\label{splitting point of a 2-roach non-1-roach is unique}
If $\F\in\mathcal R_2\setminus\mathcal R_1$, then $\F$ has a unique splitting point of depth 2.
\end{lemma}

\begin{proof}
Suppose that $\mathfrak F\in\mathcal R_2\setminus\mathcal R_1$. By \cref{n-roaches and logics}(\ref{item: L_1=L(R_1) sorta}), $\F$ is not an {\sf S4.1.2}-frame. Therefore, $\mathrm{max}(\mathfrak F)$ consists of at least 2 points. Let $s,s'\in W$ be splitting points of $\mathfrak F$ of depth $2$. Then ${\uparrow}s=\{s\}\cup\mathrm{max}(\mathfrak F)$ and ${\uparrow}s'=\{s'\}\cup\mathrm{max}(\mathfrak F)$. 
Since ${\uparrow}s'$ is not an {\sf S4.1.2}-frame, 
${\uparrow}s'\notin\mathcal R_1$ by \cref{n-roaches and logics}(\ref{item: L_1=L(R_1) sorta}). Thus, since $\F\in\mathcal R_2$, we obtain that
$s'\not\in W\setminus {\downarrow}s$. Consequently, 
$s'\le s$, and hence
$s=s'$. 
\end{proof}

\begin{proposition}\label{lem: willow trees closed under p-mor im and pt gen subs}
The class $\mathcal R_2$ is closed under point-generated subframes and p-morphic images. 
\end{proposition}

\begin{proof}
Let $\mathfrak F=(W,\le)$ be a $2$-roach with splitting point $s$. First suppose that $\mathfrak G=(V,\le)$ is the subframe of $\mathfrak F$, where 
$V={\uparrow}w$. If $w\not\le s$, then ${\uparrow}w\in\mathcal R_1$ because $\F\in\mathcal R_2$, yielding that $\G\in\mathcal R_2$ by \cref{n-roaches and logics}(\ref{R_n form strictly increasing chain}). Suppose $w\le s$.
Then $s\in V$ and
\begin{eqnarray*}
{\uparrow}s &=& {\uparrow}s\cap V = (\{s\}\cup\mathrm{max}(\mathfrak F)) \cap V = \{s\}\cup\mathrm{max}(\mathfrak G).
\end{eqnarray*}
Let $v\in V\setminus {\downarrow}s$. Then $v\in W\setminus {\downarrow}s$, giving that ${\uparrow}v\in\mathcal R_1$. 
Thus, $\mathfrak G$ is a $2$-roach. 

Next suppose that $\mathfrak G=(V,\le)$ is a p-morphic image of $\mathfrak F$, say via $f:W\to V$.
\begin{claim}
${\uparrow}f(s)=\{f(s)\}\cup\mathrm{max}(\mathfrak G)$.
\end{claim}

\begin{proof}
Let $v\in {\uparrow}f(s)$. Then 
there is $w\in {\uparrow}s$ such that $f(w)=v$. If $w=s$, then $v=f(s)$.
If $w\neq s$, 
then $w\in\mathrm{max}(\mathfrak F)$. 
Therefore, $v\in\mathrm{max}(\mathfrak G)$.  
Thus, ${\uparrow}f(s)\subseteq \{f(s)\}\cup\mathrm{max}(\mathfrak G)$.

Conversely, let $v\in\{f(s)\}\cup\mathrm{max}(\mathfrak G)$. If $v=f(s)$, then $v\in{\uparrow}f(s)$.
Suppose that $v\in\mathrm{max}(\mathfrak G)$. Since $\mathfrak F$ is an {\sf S4.1}-frame and $f$ is an onto p-morphism, 
$v=f(w)$ for some $w\in\mathrm{max}(\mathfrak F)\subseteq 
{\uparrow}s$. Thus $s\le w$, so 
$f(s)\le v$, and hence $\{f(s)\}\cup\mathrm{max}(\mathfrak G)\subseteq {\uparrow}f(s)$. 
\end{proof}

To finish verifying that $\G\in\mathcal R_2$, let $v\in V\setminus {\downarrow}f(s)$. As $f$ is onto, there is $w\in W$ such that $f(w)=v$. 
If $w\le s$, then $v \le f(s)$, contradicting that $v\not\in {\downarrow}f(s)$. Thus, $w\in W\setminus {\downarrow}s$. Since $\mathfrak F$ is a $2$-roach, 
${\uparrow}w\in\mathcal R_1$, so ${\uparrow}w$ is an {\sf S4.1.2}-frame by \cref{n-roaches and logics}(\ref{item: L_1=L(R_1) sorta}). Since $f$ is a p-morphism, ${\uparrow}v=f({\uparrow}w)$.
Because ${\uparrow}w$ is an {\sf S4.1.2}-frame, so is 
${\uparrow}v$. 
By \cref{n-roaches and logics}(\ref{item: L_1=L(R_1) sorta}), ${\uparrow}v$ is a $1$-roach. Thus, $\mathfrak G$ is a $2$-roach. 
\end{proof}

In the next section, we will use the basic facts about 2-roaches established in this section to axiomatize ${\sf L}(\mathcal R_2)$. Some of these results also generalize to the setting of arbitrary $n$-roaches. For example, for each nonzero $n\in\omega$, $\mathcal R_n$ is closed under point-generated subframes (and hence $\mathcal R_\infty$ is also closed under point generated subframes). We conjecture that each $\mathcal R_n$ is closed under p-morphic images as well, and hence 
so is $\mathcal R_\infty$. 
Verifying this conjecture requires a careful analysis of the combinatorics of $n$-roaches.
 
\section{The logic of $2$-roaches}\label{sec: logic of 2 roaches}

A forbidden configuration $\F$ for $\mathcal K$ is \emph{minimal} provided for each forbidden configuration $\G$ of $\mathcal K$, if $\G$ is permissible for $\F$ then $\G$ is isomorphic to $\F$. 
In this section we show that $\mathcal R_2$ is exactly the class of {\sf S4.1}-frames for which $\F_1,\F_2,\F_3$ are the minimal forbidden configurations. We then utilize the machinery of Fine-Jankov formulas (see, e.g., \cite[Sec.~3.4]{BRV01}) in conjunction with a result of Zakharyaschev (see, e.g., \cite[Thm.~11.58]{CZ97}) to conclude that the logic ${\sf L}(\mathcal R_2)$ is finitely axiomatizable and decidable. 

\begin{proposition}\label{Cor: forbid for beta are forbid for R_2}
Each of $\F_1,\F_2,\F_3$ is forbidden for $\mathcal R_2$.
\end{proposition}

\begin{proof}
By \cref{lem: willow trees closed under p-mor im and pt gen subs}, 
it is sufficient to show that neither of $\mathfrak F_1$, $\mathfrak F_2$, $\mathfrak F_3$ is a $2$-roach. The only point 
in $\mathfrak F_1$ that sees both maximal points is the root $r$, which is of depth 3, so $\mathfrak F_1$ is not a $2$-roach. 
The points 
in $\mathfrak F_2$ that see both maximal points are the two roots $r_1$ and $r_2$. But ${\uparrow}r_i=W_2\neq\{r_i\}\cup\mathrm{max}(\mathfrak F_2)$ for $i=1,2$, so $\mathfrak F_2$ is not a $2$-roach.  Finally, the points 
in $\mathfrak F_3$ that see both maximal points are $r$, $v_1$, and $v_2$. But $r$ is of depth 3, 
$v_i\in W_3\setminus {\downarrow}v_j$, and ${\uparrow}v_i=\{m_1,m_2,v_i\}$ is not an {\sf S4.1.2}-frame for distinct $i,j\in\{1,2\}$. 
Thus, $\mathfrak F_3\not\in\mathcal R_2$.
\end{proof} 

To prove the minimality of these three frames, we utilize the following standard tool. 

\begin{lemma}\label{lemma with parts: 2 standard tools}
Let 
$\mathfrak F=(W,\le)$ be an {\sf S4}-frame. 
\begin{enumerate}
\item If $U_1,\dots,U_n$ are disjoint upsets of $\mathfrak F$, then the frame obtained by identifying each $U_i$ to a point 
is a p-morphic image of $\mathfrak F$. \label{lem: identifying an R-upset gives p-morphic image}

\item If $w_1,\dots,w_n\in W$ satisfy $C_{w_i}=\{w_i\}$ and ${\uparrow}w_i\setminus \{w_i\}={\uparrow}w_j\setminus \{w_j\}$ for each $i,j$, then the frame obtained by identifying $\{w_1,\dots,w_n\}$ to a point is a p-morphic image of $\mathfrak F$. \label{lem: identify points with same successors}
\end{enumerate}
\end{lemma}

The next mapping lemma paves the way to realize that $\F_1$ is a permissible configuration for certain frames that fail to be a $2$-roach.

\begin{lemma}\label{lem: tall two-fork condition}
Let $\mathfrak F=(W,\le)$ be an {\sf S4.1}-frame with root $q$. If $d(q)\ge3$, $q\not\models{\sf S4.1.2}$, and ${\uparrow}w\models{\sf S4.1.2}$ for each $w\in W\setminus C_q$, then $\mathfrak F_1$ is a p-morphic image of $\mathfrak F$. 
\end{lemma}

\begin{proof} 
Let $\mathrm{max}(\mathfrak F)=\{n_1,\dots,n_k\}$. Since $q\not\models{\sf S4.1.2}$, we have that $k\ge 2$. Because ${\uparrow}w\models{\sf S4.1.2}$ for each $w\in W\setminus C_q$, there is a unique $n_i\in\mathrm{max}(\mathfrak F)$ such that $w\le n_i$. Therefore, the sets ${\downarrow}n_i\setminus C_q$ for $i=1,\dots,k$ form a partition of $W\setminus C_q$. Moreover, each ${\downarrow}n_i\setminus C_q$ is an upset because ${\uparrow}w\models{\sf S4.1.2}$ for each $w\in W\setminus C_q$. 

Since $d(q)\ge 3$, there is $y\in W\setminus(C_q\cup\mathrm{max}(\mathfrak F))$. We may assume without loss of generality that $y\le n_1$. Then $y\in {\downarrow}n_1\setminus(C_q\cup\{n_1\})$; see Figure~\ref{fig: picture for tall two fork condition}.

\begin{figure}[h]
\begin{center}
\begin{picture}(200,90)(-100,-10)
\setlength{\unitlength}{.28mm}
\put(0,10){\makebox(0,0){$\bullet$}}
\put(0,10){\oval(35,12)}
\put(0,-5){\makebox(0,0){$C_q$}}
\put(-60,100){\makebox(0,0){$\bullet$}}
\put(-60,110){\makebox(0,0){$n_1$}}
\qbezier(-60,100)(-80,40)(-20,15) 
\qbezier(-60,100)(0,80)(-10,20) 
\qbezier(-20,15)(-20,20)(-10,20) 
\put(60,100){\makebox(0,0){$\bullet$}}
\put(60,110){\makebox(0,0){$n_k$}}
\qbezier(60,100)(80,40)(20,15) 
\qbezier(60,100)(0,80)(10,20) 
\qbezier(20,15)(20,20)(10,20) 
\put(1,100){\makebox(0,0){$\cdots$}}
\put(-75,7){\makebox(0,0){${\downarrow}n_1\setminus C_q$}}
\put(-70,20){\vector(2,3){25}}
\put(-30,50){\makebox(0,0){$\bullet$}}
\put(-30,40){\makebox(0,0){$y$}}
\put(73.75,7){\makebox(0,0){${\downarrow}n_k\setminus C_q$}}
\put(70,20){\vector(-2,3){25}}
\end{picture}
\end{center}
\caption{Depiction of $W=C_q\cup\bigcup\nolimits_{i=1}^k({\downarrow}n_i\setminus C_q)$.}
\label{fig: picture for tall two fork condition}
\end{figure}
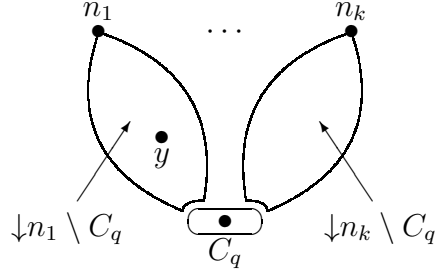

Define $f:W\to W_1$ by 
\[
f(x) =\left\{
\begin{array}{ll}
m_1& \text{if }x=n_1,\\
v
& \text{if }x\in {\downarrow}n_1\setminus (C_q\cup\{n_1\}),\\
r & \text{if }x\in C_q,\\
m_2& \text{otherwise.}
\end{array}
\right.
\]
Because $\left\{\{n_1\},{\downarrow}n_1\setminus (C_q\cup\{n_1\}),C_q,\bigcup\nolimits_{i=2}^k{\downarrow}n_i\setminus C_q\right\}$ is a partition of $W$, we have that $f$ is a well-defined mapping of $W$ onto $W_1$.

Suppose that $a,b\in W$ with $a\le b$. If $a\in C_q$, then $f(a)$ is the root of $\mathfrak F_1$, giving that $f(a)\le f(b)$. If $a\in {\downarrow}n_i\setminus C_q$ for some $i=2,\dots,k$, then $b\in {\downarrow}n_i\setminus C_q$, yielding that $f(a)=m_2=f(b)$, and so $f(a)\le f(b)$. If $a\in {\downarrow}n_1\setminus(C_q\cup\{n_1\})$, then $b\in {\downarrow}n_1\setminus C_q$ and $f(a)=v$, so $f(b)\in\{m_1,v\}$, and hence $f(a)\le f(b)$. Finally, if $a=n_1$, then $b=n_1$, implying that $f(a)\le f(b)$. Thus, $f$ satisfies the forth condition.

Suppose that $a\in W$, $u\in W_1$, and $f(a)\le u$. If $a\in C_q$, then ${\uparrow}a=W$, which implies 
that $u\in W_1=f(W) = f({\uparrow}a)$. Hence, there is $b\in {\uparrow}a$ with $f(b)=u$. Suppose that $a\in {\downarrow}n_1\setminus (C_q\cup\{n_1\})$. Then $f(a)=v$ and $u\in \{m_1,v\}$. If $u=v$, take $b=a$; and if $u=m_1$, take $b=n_1$. In either case, $a\le b$ and $f(b)=u$. Finally, if $a$ is neither in $C_q$ nor in ${\downarrow}n_1\setminus (C_q\cup\{n_1\})$, then $f(a)\in\{m_1,m_2\}=\mathrm{max}(\mathfrak F_1)$, giving that $f(a)=u$. Therefore, we may take $b=a$ to obtain that $a\le b$ and $f(b)=u$. Thus, $f$ satisfies the back condition, and hence is a p-morphism. Consequently, $\mathfrak F_1$ is a p-morphic image of $\mathfrak F$.
\end{proof}

We now prove a mapping lemma geared towards realizing that either $\F_2$ or $\F_3$ is permissible for 
the remaining frames that are not $2$-roaches.

\begin{lemma}\label{lem: crisscross and fat bottom fork condition}
Let $\mathfrak F=(W,\le)$ be a frame and $a,b\in W$ distinct points of depth 2. If both $M_a:={\uparrow}a\cap\mathrm{max}(\mathfrak F)$ and $M_b:={\uparrow}b\cap\mathrm{max}(\mathfrak F)$ consist of at least two points, then either $\mathfrak F_2$ or $\mathfrak F_3$ is a p-morphic image of a point-generated subframe of $\mathfrak F$. 
\end{lemma}

\begin{proof}
We have that $C_a=C_b$ or $C_a\ne C_b$. First suppose that 
$C_a=C_b$. In this case we show that $\mathfrak F_2$ is a p-morphic image of a generated subframe of $\mathfrak F$. Clearly 
$C_a\setminus\{a\}
\neq\varnothing$, ${\uparrow}a=C_a\cup M_a$, and $M_a$ contains at least two points. Denoting one of them by $m$,
it is straightforward to see that the following mapping realizes $\mathfrak F_2$ as a p-morphic image of the subframe ${\uparrow}a$ generated by $a$:
\[
f(x) = \left\{
\begin{array}{ll}
r_1 & \text{if }x=a,\\
r_2 & \text{if }x\in C_a\setminus\{a\},\\
m_1 & \text{if }x=m,\\
m_2 & \text{if }x\in M_a\setminus\{m\}.
\end{array}
\right.
\]

Next suppose that $C_a\neq C_b$. We show that $\mathfrak F_3$ is a p-morphic image of a point-generated subframe of $\mathfrak F$.
Since both $M_a$ and $M_b$ have more than one element, we may choose $n_1,n_2,N_1,N_2\in\mathrm{max}(\mathfrak F)$ based on the following cases:
\begin{itemize}
\item $M_a\cap M_b=\varnothing$: take distinct $n_1,n_2\in M_a$ and distinct $N_1,N_2\in M_b$.
\item $|M_a\cap M_b|=1$: take $n_1=N_1$ in $M_a\cap M_b$, $n_2\in M_a\setminus M_b$ and $N_2\in M_b\setminus M_a$.
\item $|M_a\cap M_b|\ge 2$: take $n_1=N_1$ and $n_2=N_2$ distinct in $M_a\cap M_b$.
\end{itemize}
Since each of $n_1,n_2,N_1,N_2$ is a maximal point, applying Lemma~\ref{lemma with parts: 2 standard tools}(\ref{lem: identifying an R-upset gives p-morphic image}), we may identify the set $\{n_1,N_1\}$ to a single point
and the set $\mathrm{max}(\mathfrak F)\setminus\{n_1,N_1\}$ to a single point to
obtain a p-morphic image of $\mathfrak F$ that has exactly two maximal points. Thus, without loss of generality we may assume that $\mathfrak F$ has exactly two maximal points, say $\mathrm{max}(\mathfrak F)=\{n_1,n_2\}$.

We next identify the upset $W\setminus {\downarrow}n_1$ to a point and the upset $W\setminus {\downarrow}n_2$ to a point. Applying Lemma~\ref{lemma with parts: 2 standard tools}(\ref{lem: identifying an R-upset gives p-morphic image}) again yields a p-morphic image of $\mathfrak F$ in which all points of depth 2 see both maximal points. Hence, we may assume without loss of generality that all points of $\mathfrak F$ of depth 2 
see both $n_1$ and $n_2$. Therefore, identifying the set of points of depth 2 that are distinct from $a$ to a point gives a p-morphic image of $\mathfrak F$ by Lemma~\ref{lemma with parts: 2 standard tools}(\ref{lem: identify points with same successors}). Thus, we may assume that $\mathfrak F$ has exactly two points of depth 2, namely $a$ and $b$; see Figure~\ref{fig: 2 max 2 depth 2}.
\begin{figure}[h]
\begin{center}
\begin{picture}(50,85)(-25,0)
\put(0,0){\makebox(0,0){$\bullet$}}
\multiput(-12.5,50)(25,0){2}{\multiput(0,0)(0,25){2}{\makebox(0,0){$\bullet$}}}
\multiput(-12.5,50)(25,0){2}{\line(0,1){25}}
\put(-12.5,50){\line(1,1){25}}
\put(12.5,50){\line(-1,1){25}}
\put(-12.5,85){\makebox(0,0){$n_1$}}
\put(12.5,85){\makebox(0,0){$n_2$}}
\put(-22.5,50){\makebox(0,0){$a$}}
\put(22.5,50){\makebox(0,0){$b$}}
\qbezier(0,0)(-35,35)(-12.5,50)
\qbezier(0,0)(35,35)(12.5,50)
\end{picture}
\end{center}
\caption{The frame $\mathfrak F$ with 2 maximal points and 2 points of depth 2.}
\label{fig: 2 max 2 depth 2}
\end{figure}
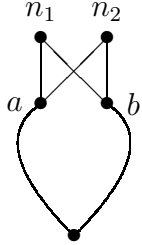

Since $\mathfrak F$ is finite and rooted, there is a quasi-maximal point 
$q$ in ${\downarrow}a\cap {\downarrow}b$. Let $\mathfrak G$ be the subframe of $\mathfrak F$ generated by $q$. Then the underlying set of $\mathfrak G$ is ${\uparrow}q$. We show that $\mathfrak F_3$ is a p-morphic image of $\mathfrak G$ via the mapping
\[
f(x) = \left\{\begin{array}{ll}
m_1 & \text{if }x=n_1, \\
m_2 & \text{if }x=n_2 ,\\
v_1 & \text{if }x\in ({\downarrow}a\cap {\uparrow}q)\setminus C_q, \\
v_2 & \text{if }x\in ({\downarrow}b\cap {\uparrow}q)\setminus C_q, \\
r & \text{if }x\in C_q.
\end{array}\right.
\]
Because $q$ is quasi-maximal in ${\downarrow}a\cap {\downarrow}b$, we have that $({\downarrow}a\cap {\uparrow}q)\setminus C_q$ and $({\downarrow}b\cap {\uparrow}q)\setminus C_q$ are disjoint. Thus, $f$ is a well-defined onto map.

Let $x,y\in {\uparrow}q$ be such that $x\le y$. If $x\in C_q$, then $f(x) = r$ is the root of $\mathfrak F_3$, and so $f(x)\le f(y)$. If $x\in {\downarrow}a \setminus C_q$, then $y\in \left( {\downarrow}a\cup\{n_1,n_2\} \right) \setminus C_q$, and hence $f(x)\le f(y)$. A symmetric argument gives that if $x\in {\downarrow}b\setminus C_q$, then $f(x)\le f(y)$. 
For $i=1,2$, if $x=n_i$, then $y=n_i$, yielding $f(x)\le 
f(y)$. Thus, $f$ satisfies the forth condition. 

To see that $f$ satisfies the back condition, let $x\in {\uparrow}q$, $u\in W_3$, and $f(x)\le u$. If $x\in C_q$, then $x$ and $q$ are in the same cluster, so ${\uparrow}x={\uparrow}q$. Therefore, $f({\uparrow}x)=W_3$, and so there is $y\in{\uparrow}x$ with $f(y)=u$. Suppose that $x\in {\downarrow}a\setminus C_q$. Then $f(x)=v_1$, giving that $u\in\{m_1,m_2,v_1\}$. Since $\{n_1,n_2,a\}\subseteq {\uparrow}x$, there is $y\in{\uparrow}x$ with 
$f(y)=u$. A symmetric argument yields a proof of the case when $x\in {\downarrow}b\setminus C_q$. Finally, if $x\in\{n_1,n_2\}$, then $$f(x)\in \{f(n_1),f(n_2)\}=\{m_1,m_2\}=\mathrm{max}(\mathfrak F_3),$$ so $f(x)=u$ and we can take $y=x$. Thus, $f$ is a p-morphism.
\end{proof}

We are ready to prove the aforementioned minimality of the forbidden configurations $\F_1,\F_2,\F_3$ with respect to $\mathcal R_2$. 

\begin{proposition}\label{lem: suf condition via F123 to be a WT}
If a frame $\mathfrak F$ is not a $2$-roach, then one of $\mathfrak F_1,\mathfrak F_2,\mathfrak F_3$ is a p-morphic image of a point-generated subframe of $\mathfrak F$, and hence is permissible for $\F$. 
\end{proposition}

\begin{proof}
Since $\mathfrak F$ is not a $2$-roach, $\mathfrak F$ is not an {\sf S4.1.2}-frame by \cref{n-roaches and logics}(\ref{R_n form strictly increasing chain}). Therefore, $\mathfrak F$ has at least 2 maximal points. Because $\mathfrak F$ is rooted, the depth of $\mathfrak F$ is at least 2. Since $\mathfrak F$ is finite, it contains a point of depth 2. If $\mathfrak F$ has at least two points of depth 2 such that each sees at least two maximal points, 
then Lemma~\ref{lem: crisscross and fat bottom fork condition} yields that either $\mathfrak F_2$ or $\mathfrak F_3$ is a p-morphic image of a point-generated subframe of $\mathfrak F$, and hence is permissible for $\F$. Thus, we may assume that there is at most one point $s$ in $\mathfrak F$ of depth 2 that can see more than one maximal point. 

First assume that $s$ sees all maximal points of $\mathfrak F$. Then we have that ${\uparrow}s=\{s\}\cup\mathrm{max}(\mathfrak F)$. Since $\mathfrak F$ is not a $2$-roach,  
there is $w\in W\setminus {\downarrow}s$ such that ${\uparrow}w\not\models{\sf S4.1.2}$. Because $\mathfrak F$ is finite, we may let $w$ be a quasi-maximal point satisfying this property. 
Since $w\neq s$ and every point of depth 2 other than $s$ sees exactly one maximal point, 
we have that $d(w)\ge 3$. Therefore, the generated subframe ${\uparrow}w$ of $\mathfrak F$ satisfies the hypothesis of Lemma~\ref{lem: tall two-fork condition}. Thus, $\mathfrak F_1$ is a p-morphic image of ${\uparrow}w$, and hence
$\mathfrak F_1$ is a p-morphic image of a point-generated subframe of $\mathfrak F$. 
Consequently, $\F_1$ is permissible for $\F$.

Next assume that $s$ does not see all maximal points of $\mathfrak F$. Since we obtain a p-morphic image of $\mathfrak F$ by identifying the nonempty sets $\mathrm{max}(\mathfrak F)\cap {\uparrow}s$ to a point and $\mathrm{max}(\mathfrak F)\setminus {\uparrow}s$ to a point, without loss of generality we may assume that $\mathfrak F$ has exactly two maximal points, say $n_1$ and $n_2$, with $s\le n_1$ and $s
\not\le n_2$. Therefore,
every point of depth 2 sees exactly one maximal point. Let $q$ be a quasi-maximal point in ${\downarrow}n_1\cap {\downarrow}n_2$. Since no point of depth $\le 2$ sees both $n_1$ and $n_2$, it follows that $d(q)\ge 3$. The maximality of $q$ implies that the generated subframe ${\uparrow}q$ satisfies the hypothesis of Lemma~\ref{lem: tall two-fork condition}. Thus, $\mathfrak F_1$ is a p-morphic image of ${\uparrow}q$, and hence $\mathfrak F_1$ is 
permissible for $\mathfrak F$.
\end{proof}

As an immediate consequence of \cref{Cor: forbid for beta are forbid for R_2,lem: suf condition via F123 to be a WT} we obtain:

\begin{theorem}\label{thm: char R_2 via F's}
A finite rooted {\sf S4.1}-frame $\mathfrak F$ is a $2$-roach iff each of $\F_1,\F_2,\F_3$ is forbidden for $\F$.
\end{theorem}

Having identified $\F_1,\F_2,\F_3$ as the minimal forbidden configurations for $\mathcal R_2$, 
we move towards axiomatizing the logic ${\sf L}(\mathcal R_2)$ of $2$-roaches via the Fine-Jankov formulas of $\F_1,\F_2,\F_3$. 

\begin{definition}\label{def: L(Phi)}
For $i\in\{1,2,3\}$, 
let $\chi_i$ be the Fine-Jankov formula of $\mathfrak F_i$. 
\end{definition}

We recall that a normal modal logic {\sf L} 
has the \emph{finite model property} (fmp) if {\sf L} is the logic of a class of finite frames. 
To see that ${\sf L}(\mathcal R_2)$ 
has the fmp, we utilize the following result of Zakharyaschev.

\begin{theorem} [{\cite[Thm.~11.58]{CZ97}}] \label{thm: Zak} 
Every normal extension of $\sf S4$ axiomatizable by a finite number
of formulas in one variable has the fmp.
\end{theorem}

We recall that an {\sf S4}-frame $(W,\le)$ is 
one-generated if there is a generator $U \subseteq W$ such that 
each subset of $W$ is obtained by applying 
the Boolean operations 
and ${\downarrow}$ 
to $U$. We are ready for the main result of this section:

\begin{theorem}\label{thm: the logic of 2-roaches}
The logic ${\sf L}(\mathcal R_2)$ of the class of $2$-roaches is obtained from ${\sf S4.1}$ by postulating $\chi_1,\chi_2,\chi_3$. 
\end{theorem}

\begin{proof}
Let {\sf L} be the logic obtained from {\sf S4.1} by postulating $\chi_1,\chi_2,\chi_3$. 
We first show that $\sf L$ has the fmp. 
Since 
the Fine-Jankov formula of a one-generated finite rooted {\sf S4}-frame can be written in one variable (see, e.g., \cite[Sec.~2]{BY23}), 
it is enough to verify that each of $\F_1,\F_2,\F_3$ is one-generated. It is straightforward to check that
$\{m_1\}$ is a generator of $\F_1$, $\{m_1,r_1\}$ is a generator of $\F_2$, and $\{m_1,v_1,r\}$ is a generator of $\F_3$. 
As {\sf ma} also contains only one variable, 
we may apply \cref{thm: Zak} to conclude that $\sf L$ has the fmp. 

We next show that ${\sf L}(\mathcal R_2)={\sf L}$. Suppose that ${\sf L}\not\vdash\varphi$. 
Since ${\sf L}$ has the fmp, there is a finite rooted ${\sf L}$-frame $\mathfrak F$ refuting $\varphi$. Because $\mathfrak F$ is an ${\sf L}$-frame, \cref{lem: top vers of Fine}(\ref{lem item: FINE}) 
implies that none of $\mathfrak F_1,\mathfrak F_2,\mathfrak F_3$ is a p-morphic image of a generated subframe of $\mathfrak F$. 
Thus, \cref{lem: suf condition via F123 to be a WT} yields that $\mathfrak F$ is a $2$-roach. 
Consequently, ${\sf L}(\mathcal R_2)\not\vdash\varphi$, and hence ${\sf L}(\mathcal R_2)\subseteq {\sf L}$.

For the other inclusion, let $i\in\{1,2,3\}$ and $\mathfrak F\in\mathcal R_2$. By \cref{Cor: forbid for beta are forbid for R_2}, $\F_i$ is forbidden for $\F$. Therefore, it follows from \cref{lem: top vers of Fine}(\ref{lem item: FINE}) that $\F\vDash\chi_i$. Since $\mathfrak F\in\mathcal R_2$ is arbitrary, ${\sf L}(\mathcal R_2)\vdash\chi_i$, which implies ${\sf L}\subseteq{\sf L}(\mathcal R_2)$ because each $2$-roach is an {\sf S4.1}-frame. Thus, ${\sf L}(\mathcal R_2)={\sf L}$.
\end{proof}

Each finitely axiomatizable logic with the fmp is decidable, and thus \cref{thm: the logic of 2-roaches} implies:
\begin{corollary}\label{R2 is decidable}
The logic ${\sf L}(\mathcal R_2)$ is decidable.
\end{corollary}

\section{Willow trees}\label{willow trees}
The remainder of the paper is dedicated to proving that ${\sf L}_2={\sf L}(\mathcal R_2)$, and hence solving {\bf P1} for $n=2$. One inclusion is easy: 

\begin{proposition}\label{L(R_2) contained in L(beta(omega_squared))}
${\sf L}(\mathcal R_2)\subseteq{\sf L}_2$.
\end{proposition}

\begin{proof}
Clearly ${\sf S4.1}\subseteq{\sf L}_2$ (see \cref{lem: L_n's chain}). 
Since $\F_1,\F_2,\F_3$ are forbidden configurations for $\beta(\omega^2)$, \cref{lem: top vers of Fine}(\ref{lem item: FINE gen Top setting}) implies that ${\sf L}_2\vdash\chi_1,\chi_2,\chi_3$, and hence ${\sf L}(\mathcal R_2)\subseteq{\sf L}_2$ by \cref{thm: the logic of 2-roaches}.
\end{proof}

To prove the reverse inclusion, by \cref{lem: truth preserving ops}(\ref{lem item: interior image}) it is enough to show that each 2-roach is an interior image of $\beta(\omega^2)$. For this purpose, we introduce a well behaved subclass of $2$-roaches, which we call willow trees because of their shape; see \cref{fig4}. The aim of this section is to show that each 2-roach can be unravelled into a willow tree. Thus, it is enough to obtain each willow tree as an interior image of $\beta(\omega^2)$. 

We recall (see, e.g., \cite[p.~71]{CZ97}) that a finite rooted {\sf S4}-frame $\F=(W,\le)$ is a \emph{quasi-tree} provided for each $w\in W$ and $u,v\in{\downarrow}w$, either $u\le v$ or $v\le u$. A partially ordered quasi-tree is simply a \emph{tree}. 

\begin{definition}\label{def: unrav WT}
Let $\mathfrak F=(W,\le)$ be a $2$-roach.  
Call $\mathfrak F$ a \emph{willow tree} provided the subframe $W\setminus {\uparrow}s$ 
is a quasi-tree for some splitting point $s\in W$. 
Let $\mathcal W$ be the class of willow trees and ${\sf L}(\mathcal W)$ the logic of $\mathcal W$. 
\end{definition}

\cref{fig4} depicts a willow tree with a splitting point $s$ in which $W\setminus{\uparrow}s$ is partitioned by the quasi-trees $T_s:={\downarrow}s\setminus\{s\}$ and $T_{i,j}$, where $i=1,\dots,n$ and $j=1,\dots,k_i$ for some $k_i\in\omega$. 

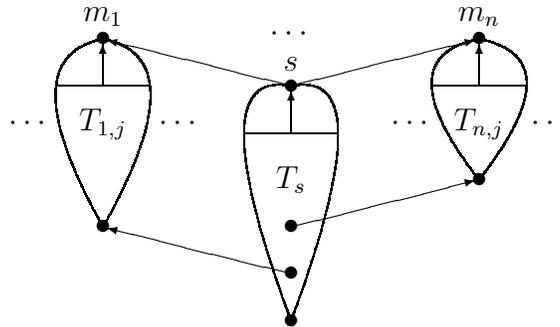
\begin{figure}[h]
\begin{center}
\begin{picture}(140,120)(-70,-5)
\setlength{\unitlength}{.25mm}
\multiput(0,0)(0,25){3}{\makebox(0,0){$\bullet$}}
\qbezier(0,0)(-50,135)(0,125)
\qbezier(0,0)(50,135)(0,125)
\qbezier(-100,50)(-49.5,140)(-100,150)
\qbezier(-100,50)(-150.5,140)(-100,150)
\qbezier(100,75)(49.5,140)(100,150)
\qbezier(100,75)(150.5,140)(100,150)
\put(0,125){\makebox(0,0){$\bullet$}}
\multiput(-100,50)(0,100){2}{\makebox(0,0){$\bullet$}}
\multiput(100,75)(0,75){2}{\makebox(0,0){$\bullet$}}
\multiput(0,25)(0,100){2}{\vector(-4,1){98}}
\multiput(0,50)(0,75){2}{\vector(4,1){98}}
\multiput(-100,125)(200,0){2}{\vector(0,1){23}}
\put(0,100){\vector(0,1){23}}
\put(-25,100){\line(1,0){50}}
\multiput(-139,105)(200,0){1}{\makebox(0,0){$\cdots$}}
\multiput(-59,105)(200,0){1}{\makebox(0,0){$\cdots$}}
\multiput(64.5,105)(200,0){1}{\makebox(0,0){$\cdots$}}
\multiput(139,105)(200,0){1}{\makebox(0,0){$\cdots$}}
\multiput(-125,125)(200,0){2}{\line(1,0){50}}
\put(-100,162){\makebox(0,0){$m_1$}}
\put(100,162){\makebox(0,0){$m_n$}}
\put(0,137){\makebox(0,0){$s$}}
\put(0,153){\makebox(0,0){$\cdots$}}
\put(0,75){\makebox(0,0){$T_s$}}
\put(-100,105){\makebox(0,0){$T_{1,j}$}}
\put(100,105){\makebox(0,0){$T_{n,j}$}}
\end{picture}
\end{center}
\caption{A willow tree.}
\label{fig4}
\end{figure}

By the classic unravelling technique in modal logic (see, e.g., \cite[p.~72]{CZ97}), every finite rooted {\sf S4}-frame is a p-morphic image of a quasi-tree. 
We will modify the unraveling technique by applying it to only a lower portion of a $2$-roach, 
thus yielding the following main result of this section.

\begin{theorem}\label{prop: each WT is p-mor image of unravelled WT} 
Every $2$-roach is a p-morphic image of a willow tree.
\end{theorem}

\begin{proof}
Let $\mathfrak F=(W,\le)$ be a 2-roach with splitting point $s$. We let $\mathfrak G$ be the subframe of $\mathfrak F$ whose underlying set is $W\setminus {\uparrow}s$. Then $\mathfrak G$ is a finite rooted {\sf S4}-frame, so is a p-morphic image of a finite quasi-tree 
$\mathfrak T=(T,\sqsubseteq)$ (see, e.g., \cite[Lem.~5]{BG05a}). Let $f$ be such a p-morphism. We may assume without loss of generality that $T\cap W=\varnothing$.
We construct a willow tree $\mathfrak H=(V,\preceq)$ such that $\mathfrak F$ is a p-morphic image of $\mathfrak H$. Intuitively, this is achieved by enlarging $\mathfrak T$ by appropriately attaching ${\uparrow}s$ ``above" $\mathfrak T$ and extending $f$ by the identity function $\mathrm{id}$ on~${\uparrow}s$. 

\begin{figure}[h]
\begin{center}
\begin{picture}(195,115)(-25,-5)
\multiput(25,0)(100,0){2}{\makebox(0,0){$\bullet$}}
\put(25,0){\line(-1,2){25}}
\put(25,0){\line(1,2){25}}
\put(0,50){\line(1,0){50}}
\put(25,50){\vector(0,1){13}}
\put(25,27){\makebox(0,0){\small $T_s$}}
\put(0,10){\makebox(0,0){\small $\mathfrak 
T$}}
\multiput(25,65)(100,0){2}{\makebox(0,0){$\bullet$}}
\multiput(25,65)(100,0){2}{\line(-2,1){20}}
\multiput(25,65)(100,0){2}{\line(2,1){20}}
\multiput(25,85)(100,0){2}{\makebox(0,0){$\cdots$}}
\multiput(25,75)(100,0){2}{\makebox(0,0){$s$}}
\multiput(5,75)(100,0){2}{\multiput(0,0)(40,0){2}{\makebox(0,0){$\bullet$}}}
\put(125,25){\makebox(0,0){\tiny${\downarrow}s\setminus\{s\}$}}
\put(175,50){\makebox(0,0){\small $\G=W\setminus{\uparrow}s$}}
\qbezier(125,0)(70,30)(103,70)
\qbezier(125,0)(155,20)(130,60)
\qbezier(103,70)(118,57)(130,60)
\qbezier(125,0)(95,17.5)(120,60)
\put(25,110){\makebox(0,0){$\mathfrak H=(V,\preceq)$}}
\put(125,110){\makebox(0,0){$\F=(W,\le)$}}
\multiput(25,37)(0,65){1}{\vector(1,0){95}}
\put(60,110){\vector(1,0){30}}
\put(75,30){\makebox(0,0){\small$f$}}
\put(-15,44){\vector(1,0){117.5}}
\put(-14.5,54.25){\scriptsize \makebox(0,0){$T_{1,j}$}}
\put(60,72){\vector(1,0){30}}
\put(75,63){\makebox(0,0){$\mathrm{id}$}}
\put(-15,35){\makebox(0,0){$\bullet$}}
\put(25,15){\makebox(0,0){$\bullet$}}
\put(25,15){\vector(-2,1){38}}
\put(-15,35){\line(-1,2){12.5}}
\put(-15,35){\line(1,2){12.5}}
\put(-27.5,60){\line(1,0){25}}
\put(-15,60){\vector(4,3){18}}
\multiput(5,85)(100,0){2}{\makebox(0,0){$m_1$}}
\multiput(45,85)(100,0){2}{\makebox(0,0){$m_n$}}
\end{picture}
\caption{The willow tree $\mathfrak H$ and extending $f:\mathfrak T\to\mathfrak G$ by $\mathrm{id}:{\uparrow}s\to{\uparrow}s$.}
\label{fig: constructing willow from 2-roach}
\end{center}
\end{figure}
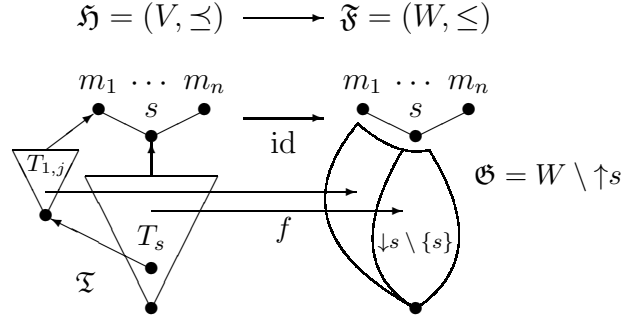

Put $V=T\cup {\uparrow}s$ and define a binary relation $\preceq$ on $V$ by setting $x\preceq y$ provided  
\begin{enumerate}
\item $x,y\in T$ and $x\sqsubseteq y$, 
\item $x\in T$, $y\in {\uparrow}s$, and $f(x)\le y$, or 
\item $x,y\in {\uparrow}s
$ and $x\le y$. 
\end{enumerate}

\begin{claim}
$\mathfrak H=(V,\preceq)$ is a willow tree.
\end{claim}

\begin{proof}
It is immediate from the construction that ${\uparrow}s$ in $\mathfrak H$ and ${\uparrow}s$ in $\F$ coincide and $\mathrm{max}(\mathfrak H)=\mathrm{max}(\mathfrak F)$.
Therefore, ${\uparrow}s=\{s\}\cup\mathrm{max}(\mathfrak F)=\{s\}\cup\mathrm{max}(\mathfrak H)$.

Let $x\in V\setminus {\downarrow}s$ and consider ${\uparrow}x$. If $x\in\mathrm{max}(\mathfrak H)$, then ${\uparrow}x=\{x\}$, so ${\uparrow}x\models{\sf S4.1.2}$. Suppose that $x\not\in\mathrm{max}(\mathfrak H)$. As $x\neq s$, we have that $x\not\in 
{\uparrow}s$. Therefore, $x\in T$, and hence $f(x)\in W\setminus {\uparrow}s$. If $f(x)\le s$, then condition (2) yields that $x\preceq s$. Thus, 
$f(x)\in W\setminus {\downarrow}s$. Because $\mathfrak F$ is a $2$-roach, ${\uparrow}f(x)\models{\sf S4.1.2}$. 
Consequently, there is a unique $m\in\mathrm{max}(\mathfrak F)$ such that $f(x)\le m$. Since $\mathrm{max}(\mathfrak F)\subseteq {\uparrow}s$, condition (2) implies that $x\preceq m$. From $\mathrm{max}(\mathfrak F)=\mathrm{max}(\mathfrak H)$ it follows that $m$ is the unique point of $\mathrm{max}(\mathfrak H)$ such that $x\preceq m$. Therefore, ${\uparrow}x\models{\sf S4.1.2}$, and hence $\mathfrak H$ is a $2$-roach. 
Finally, since the subframe $V\setminus {\uparrow}s$ of $\mathfrak H$ is the quasi-tree $\mathfrak T$, we conclude that $\mathfrak H$ is a willow tree. 
\end{proof}

It is left to show that $\mathfrak F$ is a p-morphic image of $\mathfrak H$. Extend $f$ to $g:V\to W$ 
by setting $g(x)=x$ for each $x\in {\uparrow}s$. Then $g$ is a well-defined onto map.
Let $x,y\in V$ with $x\preceq y$. First suppose that $x,y\in T$ and $x\sqsubseteq y$. Then $f(x)\le f(y)$ since $f$ is a p-morphism. Therefore, $g(x)\le g(y)$. 
Next suppose that $x\in T$, $y\in{\uparrow} s$, and $f(x)\le y$. Because $g(x)=f(x)$ and $g(y)=y$, we have that $g(x)\le g(y)$. Finally, suppose that $x,y\in {\uparrow}s$ and $x\le
y$. Then $g(x)\le g(y)$ since $g(x)=x$ and $g(y)=y$. Thus, $g$ satisfies the forth condition.

Let $x\in V$ and $w\in W$ be such that $g(x)\le w$. 
First suppose that $w\in W\setminus {\uparrow}s$. Then $g(x)\in W\setminus {\uparrow}s$ since $W\setminus {\uparrow}s$ is a $\le$-downset. Therefore, $x\in T$ and $f(x)\le w$. Because $f$ satisfies the back condition, there is $y\in T\subseteq V$ such that $x\sqsubseteq y$ and $f(y)=w$. Thus, $x\preceq y$ by condition (1), and $g(y)=f(y)=w$. Next suppose that $w\in {\uparrow}s$. 
If $x\in T$, then $f(x)=g(x)\le w$ implies that $x\preceq w$ by condition (2). 
If $x\in {\uparrow}s$, then $x=g(x)\le w$ implies that $x\preceq w$ by condition (3). In either case, $x\preceq w$ and $g(w)=w$. Thus, $g$ satisfies the back condition, and hence $g$ is a p-morphism. 
\end{proof}

As an immediate consequence of \cref{prop: each WT is p-mor image of unravelled WT} we obtain:

\begin{theorem}\label{thm: logic of unravelled willow trees is L(W)} 
${\sf L}(\mathcal W)={\sf L}(\mathcal R_2)$.
\end{theorem}

\begin{proof}
Since $\mathcal W\subseteq \mathcal R_2$, we have ${\sf L}(\mathcal R_2)\subseteq {\sf L}(\mathcal W)$. Suppose 
${\sf L}(\mathcal R_2)\not\vdash\varphi$. Then there is a $2$-roach $\mathfrak F$ refuting $\varphi$. By \cref{prop: each WT is p-mor image of unravelled WT}, $\mathfrak F$ is a p-morphic image of a willow tree $\mathfrak H$. By \cref{lem: truth preserving ops}(\ref{lem item: interior image}), 
$\mathfrak H$ refutes $\varphi$, giving that ${\sf L}(\mathcal W)\not\vdash\varphi$. Thus, ${\sf L}(\mathcal W)\subseteq {\sf L}(\mathcal R_2)$, and the result follows.
\end{proof}

Partially ordered willow trees will serve as a first stepping stone in the aforementioned goal of obtaining each willow tree as an interior image of $\beta(\omega^2)$. As such, we conclude this section with the following corollary to \cref{prop: each WT is p-mor image of unravelled WT}. 

\begin{corollary}\label{cor: poset version-each WT is an image of UnRavWT}
Each partially ordered $2$-roach is a p-morphic image of a partially ordered willow tree.
\end{corollary}

\begin{proof}
If in the proof of \cref{prop: each WT is p-mor image of unravelled WT} 
the 2-roach $\mathfrak F$ is
partially ordered, then $\mathfrak G$ is a finite rooted partially ordered {\sf S4}-frame. Therefore, the quasi-tree $\mathfrak T=(T,\sqsubseteq)$ is in fact a tree (see, e.g., \cite[Lem.~4]{BG05a}). Thus, the willow tree $\mathfrak H$ is partially ordered.
\end{proof}

\section{Partition lemmas for $\omega^*$}\label{sec: partition lemmas for omega star}

Proving that each willow tree is an interior image of $\beta(\omega^2)$ requires some preparation. In this section we develop two partition 
lemmas for the remainder $\omega^*$ of $\beta(\omega)$. These then will be used in the next section to prove 
two mapping lemmas for $\omega^*$. Both partition lemmas use CH. In particular, 
we freely use Parovi\v{c}enko's 
characterization of $\omega^*$ 
and one of its consequences known as the Homeomorphism Extension Theorem (see \cref{thm: 2 main tools} below), both of which assume CH. 

Let $X$ be a topological space. As usual, we call $N\subseteq X$ a neighborhood of $A\subseteq X$ provided $A\subseteq \ii(N)$. Then $A$ is a \emph{$P$-set} in $X$ if the intersection of an arbitrary countable family of neighborhoods of $A$ is a neighborhood of $A$. 

We recall that 
{\em $F_\sigma$-subsets} of 
$X$ are countable unions of closed sets, and {\em $G_\delta$-subsets} are countable intersections of open sets. Then $X$ is a \emph{$P$-space} if each $G_\delta$-set is open, and an \emph{almost $P$-space} if each nonempty $G_\delta$-set has nonempty interior. 
The concept of an \emph{$F$-space} was introduced in \cite[p.~208]{GJ76} for completely regular spaces. Since we are interested in this concept for normal spaces, we use the following characterization: 
a normal space $X$ is an $F$-space iff disjoint open $F_\sigma$-subsets of $X$ have disjoint closures (see, e.g., \cite[Lem.~1.2.2(b)]{vMi84}). We often utilize that a closed subspace of a normal $F$-space is a normal $F$-space (see, e.g., \cite[Lem.~1.2.2(d)]{vMi84}).

Our main example of an $F$-space that is an almost $P$-space is $\omega^*$. To identify further properties of $\omega^*$, 
we recall 
that a \emph{Stone space} is a zero-dimensional  compact Hausdorff space. It is \emph{crowded} if it has no isolated points.
Clearly both $\beta(\omega)$  
and $\omega^*$ are 
Stone spaces. Moreover, $\omega^*$ is crowded and we have the following characterization of 
$\omega^*$: 

\begin{theorem}\label{thm: 2 main tools}
Assume {\em CH}.
\begin{enumerate}
\item \emph{(Parovi\v{c}enko \cite{Par63})} 
Each crowded Stone space of weight $\omega_1$ that is both an $F$-space and an almost $P$-space  
is homeomorphic to $\omega^*$. \label{thm item: PT}
\item \emph{(van Douwen and van Mill \cite[Thm.~1.1]{vDvM93})} 
Every homeomorphism between closed nowhere dense $P$-sets in $\omega^*$ can be extended to a homeomorphism of $\omega^*$ with itself. \label{thm item: HET}
\end{enumerate}
\end{theorem} 

\subsection{The first partition lemma}
To prove the first partition lemma, we require some basic facts about $P$-sets and almost $P$-spaces, which are collected below. 

\begin{lemma} 
\label{lem: some properties of P-sets and almost P-spaces}
Let $X$ be a space and $A\subseteq X$.
\begin{enumerate}
\item $A$ is a $P$-set 
iff for each $F_\sigma$-subset $F$ of $X$, if $A\cap F = \varnothing$ then $A\cap\ccc(F)=\varnothing$. \label{lem Item: characterizing a P-set via F sigma sets}
\item If $A$ is a $P$-set in $X$ and $Y$ is a closed subspace of $X$ such that $A\subseteq Y$, then $A$ is a $P$-set in $Y$. \label{lem Item: being a P-set and contained in a closed set is a P-set in the closed subspace}
\item A finite union of $P$-sets 
is a $P$-set. 
\label{lem Item: finite union of P-sets is a P-set}
\item If $X$ is an almost $P$-space and $Y$ is an open subspace of $X$, then $Y$ is an almost $P$-space.\label{lem Item: almost P-space is an open hereditary property}
\end{enumerate}
\end{lemma}

\begin{lemma}[First Partition Lemma]\label{lem: Jan 0.4-open partition of the complement of nwd closed P-set}
Assume \emph{CH} and let $A$ be a nonempty closed nowhere dense $P$-set in $\omega^*$. For each nonzero $n\in\omega$, 
there is an open partition $\{U_1,\dots,U_n\}$ of $\omega^*\setminus A$ such that $\ccc(U_i)=U_i\cup A$ for each $1\le i\le n$.
\end{lemma}

\begin{proof}
Let $\mathscr C$ be the set of all nonempty clopen subsets of $\omega^*$. It follows from \cite[Prop.~7.23]{Kop89} that $|\mathscr C|=|\wp(\omega)|=\mathfrak c$, which equals $\omega_1$ by CH. Let $\mathscr C_1=\{C\in \mathscr C \mid C\subseteq\omega^*\setminus A\}$ and $\mathscr C_2=\{D\in\mathscr C \mid D\cap A\neq \varnothing\}$.

\begin{claim}\label{new claim 1}
Both $\mathscr C_1$ and $\mathscr C_2$ are uncountable.
\end{claim}

\begin{proof}
Suppose $\mathscr C_1$ is countable and set $U=\bigcup\mathscr C_1$. Then $U=\omega^*\setminus A$ is an open $F_\sigma$-set 
disjoint from $A$, so $\ccc(U)\cap A=\varnothing$ by
\cref{lem: some properties of P-sets and almost P-spaces}(\ref{lem Item: characterizing a P-set via F sigma sets}), and hence  
$U$ is clopen. Therefore, $A$ is clopen, which contradicts that $A$ is nonempty and  nowhere dense. 
Thus, $\mathscr C_1$ is uncountable, and it follows that $\mathscr C_2$ is uncountable as well. 
\end{proof}

Let $\{C_\alpha\mid \alpha<\omega_1\}$ and $\{D_\alpha \mid \alpha<\omega_1\}$ be enumerations of $\mathscr C_1$ and $\mathscr C_2$, respectively. 
For each $\alpha<\omega_1$, we recursively define a pairwise disjoint set of nonempty clopen subsets $U_\alpha^1,\dots,U_\alpha^n$ of $\omega^*$ such that for each $1\le i\le n$ the following properties hold: 
\begin{enumerate}
\item $U_\alpha^i\subseteq\omega^*\setminus A$,
\item $\beta<\alpha$ implies $U_\beta^i\subseteq U_\alpha^i$,
\item $D_\alpha\cap U_\alpha^i\neq\varnothing$, and
\item $C_\alpha\subseteq U_\alpha^1\cup\cdots\cup U_\alpha^n$.
\end{enumerate}

Let $1\le i\le n$ and $\alpha<\omega_1$. Suppose that the sets $U_\beta^1,\dots,U_\beta^n$ have been defined for each $\beta<\alpha$. Set $E_i=\bigcup_{\beta<\alpha}U_\beta^i$. Note that $\alpha$ is countable since $\alpha<\omega_1$. Thus, because $U_\beta^i$ is a clopen subset of $\omega^*$ that is contained in $\omega^*\setminus A$, it follows that $E_i$ is a countable union of clopen subsets of $\omega^*$ which are contained in $\omega^*\setminus A$. Therefore, $E_i$ is an open $F_\sigma$-subset of $\omega^*$ that is contained in $\omega^*\setminus A$. Because $A$ is a $P$-set in $\omega^*$, 
\cref{lem: some properties of P-sets and almost P-spaces}(\ref{lem Item: characterizing a P-set via F sigma sets}) implies that $A\cap \ccc(E_i)=\varnothing$.

We now show that the family $E_1,\dots,E_n$ is pairwise disjoint. Let $x\in E_i\cap E_j$ where $1\le i,j\le n$. Then there are $\beta_1,\beta_2<\alpha$ such that $x\in U_{\beta_1}^i$ and $x\in U_{\beta_2}^j$. Without loss of generality suppose that $\beta_2\le \beta_1$. Therefore,
$x\in U_{\beta_2}^j\subseteq U_{\beta_1}^j$, and hence $U_{\beta_1}^i\cap U_{\beta_1}^j\neq\varnothing$. Since the family $U_{\beta_1}^1,\dots,U_{\beta_1}^n$ is pairwise disjoint, we have that $i=j$, and hence $E_1,\dots,E_n$ is pairwise disjoint. Because $\omega^*$ is an $F$-space, it follows that $\ccc(E_1),\dots,\ccc(E_n)$ is pairwise disjoint.

The previous considerations yield that 
$A,\ccc(E_1),\dots,\ccc(E_n)$ form a finite pairwise disjoint family of
closed subsets of $\omega^*$. Because $\omega^*$ is a Stone space, there are pairwise disjoint clopen subsets $G,F_1,\dots,F_n$ of $\omega^*$ such that $A\subseteq G$ and $\ccc(E_i)\subseteq F_i$ 
for each $1\le i\le n$. 
Also, being pairwise disjoint implies $F_i\subseteq \omega^*\setminus A$ for each $1\le i\le n$ (since $A\subseteq G$).

Let $F=F_n\cup\left(C_\alpha\setminus(F_1\cup\dots\cup F_{n-1})\right)$. Then $F$ is clopen in $\omega^*$, $\ccc(E_n)\subseteq F$,
$F$ is disjoint from $A$ and $F_i$ for each $1\le i\le n-1$, 
and $C_\alpha\subseteq F_1\cup\dots\cup F_{n-1}\cup F$.
Therefore, we may replace $F_n$ by $F$ and assume without loss of generality that $C_\alpha\subseteq F_1\cup\dots\cup F_n$.

Clearly $D_\alpha\setminus\left(F_1\cup\dots\cup F_n\right)$ is clopen in $\omega^*$. Since $A$ is disjoint from $F_1\cup\dots\cup F_n$, 
we obtain that $\varnothing \neq D_\alpha\cap A\subseteq D_\alpha\setminus\left(F_1\cup\dots\cup F_n\right)$. Because $A$ is closed and nowhere dense, 
$D_\alpha\setminus(A\cup F_1\cup\dots\cup F_n)$
is nonempty and open, so it is infinite. 
Thus, there are pairwise disjoint nonempty clopens 
$H_1,\dots,H_n$ 
such that $H_i\subseteq D_\alpha\setminus(A\cup F_1\cup\dots\cup F_n)$ for each $1\le i\le n$.

For each $1\le i\le n$, set $U_\alpha^i=F_i\cup H_i$. We show that $U_\alpha^1,\dots,U_\alpha^n$ has the desired properties. Observe that $C_\alpha\subseteq F_1\cup\dots\cup F_n\subseteq U_\alpha^1\cup\dots\cup U_\alpha^n$. Let $1\le i\le n$. Since both $F_i$ and $H_i$ are clopen, 
$U_\alpha^i$ is clopen. Moreover, $U_\alpha^i$ is nonempty because $H_i$ is nonempty. Furthermore, 
$U_\alpha^i$ is disjoint from $A$.
For each $\beta<\alpha$, we have that $U_\beta^i\subseteq E_i\subseteq\ccc(E_i)\subseteq F_i\subseteq U_\alpha^i$. Since $H_i\subseteq D_\alpha\setminus(A\cup F_1\cup\dots\cup F_n)$, it follows that $\varnothing\neq H_i = D_\alpha \cap H_i\subseteq D_\alpha\cap U_\alpha^i$. Lastly, suppose $1\le j\le n$ and $i\neq j$. Then 
\[
U_\alpha^i\cap U_\alpha ^j = (F_i\cup H_i)\cap(F_j\cup H_j) 
= (F_i\cap F_j) \cup (F_i\cap H_j) \cup (H_i\cap F_j) \cup (H_i \cap H_j) 
= \varnothing
\]  
because both families $F_1,\dots,F_n$ and $H_1\dots,H_n$ are pairwise disjoint, $$H_i\subseteq D_\alpha\setminus\left(A\cup F_0\cup\dots\cup F_n\right)\subseteq\omega^*\setminus\left(A\cup F_0\cup\dots\cup F_n\right)\subseteq\omega^*\setminus F_j,$$
and similarly $H_j\subseteq\omega^*\setminus F_i$.

For each $1\le i\le n$, put $U_i=\bigcup_{\alpha<\omega_1}U_\alpha^i$. We show that $\{U_1,\dots,U_n\}$ is an open partition of $\omega^*\setminus A$ such that $\ccc(U_i)=U_i\cup A$ for each $1\le i \le n$. To see that $\{U_1,\dots,U_n\}$ is pairwise disjoint, let $x\in U_i\cap U_j$ for some $1\le i,j\le n$. Then there are $\alpha_1,\alpha_2<\omega_1$ such that $x\in U_{\alpha_1}^i$ and $x\in U_{\alpha_2}^j$. Without loss of generality suppose that $\alpha_2\le \alpha_1$. Then $x\in U_{\alpha_2}^j\subseteq U_{\alpha_1}^j$, and hence $U_{\alpha_1}^i\cap U_{\alpha_1}^j\neq\varnothing$, which yields that $i=j$. Let $1\le i\le n$. For each $\alpha<\omega_1$ we have that $U_\alpha^i$ is a nonempty clopen set
disjoint from $A$.
Therefore, $U_i=\bigcup_{\alpha<\omega_1}U_\alpha^i$ is a nonempty open subset of $\omega^*$ that is disjoint from $A$.
Let $x\notin A$.
Since $\omega^*\setminus A$ is open, there is clopen $C$ such that $x\in C\subseteq \omega^*\setminus A$. Thus, there is $\alpha_x<\omega_1$ such that $C=C_{\alpha_x}$. Consequently, \[
x\in C_{\alpha_x} \subseteq U_{\alpha_x}^1\cup\dots\cup U_{\alpha_x}^n\subseteq U_1\cup\dots\cup U_n,
\]
and hence $\{U_1,\dots,U_n\}$ is an open partition of $\omega^*\setminus A$.

Let $1\le i\le n$. To see that $\ccc(U_i)=U_i\cup A$, 
let $x\in A$ and $C$ be a clopen neighborhood of $x$. Then $C\cap A\neq\varnothing$, so there is $\alpha_x<\omega_1$ such that $C=D_{\alpha_x}$. We have 
\[
\varnothing\neq D_{\alpha_x}\cap U_{\alpha_x}^i\subseteq D_{\alpha_x}\cap\bigcup\nolimits_{\alpha<\omega_1}U_\alpha^i=D_{\alpha_x}\cap U_i=C\cap U_i,
\]
which implies that $A\subseteq\ccc(U_i)$, and hence $\ccc(U_i)\supseteq U_i\cap A$. Conversely, because $\{U_1,\dots,U_n\}$ is an open partition of $\omega^*\setminus A$, we have that
\[
\omega^*\setminus(U_i\cup A) = (\omega^*\setminus U_i)\cap (\omega^*\setminus A) = U_1\cup\dots\cup U_{i-1}\cup U_{i+1}\cup\dots\cup U_n
\]
is an open subset of $\omega^*$. Therefore, $U_i\cup A$ is a closed subset of $\omega^*$ that contains $U_i$. Thus, $\ccc(U_i)\subseteq U_i\cup A$, giving that $\ccc(U_i)=U_i\cup A$.
\end{proof}

\subsection{The second partition lemma}\label{2nd Part lem}
To prove the second partition lemma, we utilize the following three lemmas
about $P$-sets and almost $P$-spaces. 
We recall that each nonempty $G_\delta$-subset of a regular space $X$ contains a nonempty closed $G_\delta$-subset $F$ of $X$. More precisely, there is a 
family $\{V_n\}_{n\in\omega}$ of open subsets of $X$ such that 
\[
F=\bigcap\nolimits_{n\in\omega}V_n \ \text{ and } \ \ccc(V_{n+1})\subseteq V_n.
\tag{\textrm{\dag\dag}}\label{Cond2}
\]

\begin{lemma}\label{lem: jan lemma 0.1}
Suppose that $X$ is regular and $A$ is a closed $P$-set in $X$. Let $U$ be an open subset of $X$ such that $U \cap A =\varnothing$ and $\ccc(U)=U\cup A$. 
\begin{enumerate}
\item $\ccc(U)$ is a $P$-set in $X$. \label{lem item: jan lemma 0.1-closure of U is P-set}
\item If $\ccc(U)$ is compact and $U$ is an almost $P$-space, then $\ccc(U)$ is an almost $P$-space. \label{lem item: jan lemma 0.1-compact and almost P implies almost P}
\end{enumerate}
\end{lemma}

\begin{proof}
(1) Let $F$ be an $F_\sigma$-subset of $X$ such that $\ccc(U) \cap F = \varnothing$. Then $(U\cup A) \cap F = \varnothing$, so $A \cap F = \varnothing$. Since $A$ is a $P$-set, 
Lemma~\ref{lem: some properties of P-sets and almost P-spaces}(\ref{lem Item: characterizing a P-set via F sigma sets}) implies that $A\cap\ccc(F)=\varnothing$, which yields that $A\subseteq X\setminus\ccc(F) = \ii(X\setminus F)$. Clearly $F\cap \ii(X\setminus F)=\varnothing=F\cap U$.  
Therefore, $F\cap\left(U\cup\ii(X\setminus F)\right)=
\varnothing$, and hence
$\ccc(F)\cap\left(U\cup\ii(X\setminus F)\right)=\varnothing$. Thus, 
\[
\ccc(U)\cap\ccc(F) = (U\cup A) \cap \ccc(F) \subseteq (U\cup\ii(X\setminus F)) \cap \ccc(F) = \varnothing.
\]
We conclude that $\ccc(U)$ is a $P$-set 
by Lemma~\ref{lem: some properties of P-sets and almost P-spaces}(\ref{lem Item: characterizing a P-set via F sigma sets}). 

(2) Suppose that $\ccc(U)$ is compact and $U$ is an almost $P$-space. 
Let $G$ be a nonempty $G_\delta$-subset of $\ccc(U)$. Since $\ccc(U)$ is a regular space, by (\ref{Cond2}) 
there is a nonempty $F\subseteq G$ and a decreasing family $\{V_n\}_{n\in\omega}$ of open subsets of $\ccc(U)$ such that $F=\bigcap_{n\in\omega}V_n$ and $\ccc(V_{n+1})\subseteq V_n$ for each $n\in \omega$. 
To see that $\ccc(U)$ is an almost $P$-space, it is sufficient to prove that the interior of $F$ relative to $\ccc(U)$ is nonempty.

\begin{claim}\label{clm: F cap U is not empty}
$F\cap U\neq\varnothing$.
\end{claim}

\begin{proof}
Since $U$ is an open dense subset of $\ccc(U)$, 
we have that $A=\ccc(U)\setminus U$ is a closed nowhere dense subset of $\ccc(U)$. 
Therefore, $V_n\not\subseteq A$ for each $n\in\omega$, and hence there is $x_n\in V_n\setminus A = V_n \cap U$. Set $C_n=\ccc\left(\{x_m\mid m\ge n\}\right)$ for each $n\in\omega$. Then $\{C_n\mid n\in\omega\}$ is a nested family of closed subsets of $\ccc(U)$. 
Because $\ccc(U)$ is compact, there is $x\in\bigcap_{n\in\omega}C_n$.  
Clearly $\{x_m\mid m\ge0\}$
is an $F_\sigma$-subset of $X$. Since $A$ is a $P$-set 
and $\{x_m\mid m\ge 0\}\cap A \subseteq U\cap A = \varnothing$, 
Lemma~\ref{lem: some properties of P-sets and almost P-spaces}(\ref{lem Item: characterizing a P-set via F sigma sets}) implies that 
\[
A\cap\bigcap\nolimits_{n\in\omega}C_n\subseteq A\cap C_0 = A\cap \ccc\left(\{x_m\mid m\ge0\}\right) = \varnothing.
\]
Therefore, $x\not\in A$, so
$x\in U$. Because $\{V_n\}_{n\in\omega}$ is decreasing, $\{x_m\mid m\ge n\}\subseteq V_n$, and hence $C_n=\ccc\left(\{x_m\mid m\ge n\}\right)\subseteq \ccc(V_n)$ for each $n\in \omega$. Thus, 
\[
x\in\bigcap\nolimits_{n\in\omega}C_n\subseteq \bigcap\nolimits_{n\in\omega}\ccc(V_n)=F.
\] 
Consequently, $x\in F\cap U$, showing that 
$F\cap U\neq\varnothing$.
\end{proof}

Since $F$ is a $G_\delta$-subset of $\ccc(U)$, we have that $F\cap U$ is a $G_\delta$-subset of $U$. By Claim~\ref{clm: F cap U is not empty}, $F\cap U$ is nonempty. Because $U$ is an almost $P$-space, 
the interior of $F\cap U$ relative to $U$ is nonempty. Since $U$ is open in $\ccc(U)$, 
the interior of $F\cap U$ relative to $\ccc(U)$ is nonempty, which yields that the interior of $F$ relative to $\ccc(U)$ is nonempty. 
\end{proof} 

\begin{lemma}\label{lem: jans lemma 0.2}
Let $X$ be a space, $A\subseteq X$ a closed $P$-set, and $U$ an open subset of $X\setminus A$. 
Set $B=\ccc(U)\setminus U$. If $B\subseteq A$, then $B$ is a nowhere dense closed $P$-set in $\ccc(U)$.
\end{lemma}

\begin{proof}
Since $U$ is an open dense subset of $\ccc(U)$, we have that 
$B$ is a closed nowhere dense subset of $\ccc(U)$. To see that $B$ is a $P$-set in $\ccc(U)$, let $F$ be an $F_\sigma$-subset of $\ccc(U)$ such that $B \cap F = \varnothing$. Then $F\subseteq\ccc(U)\setminus B = U \subseteq X\setminus A$, so $A \cap F = \varnothing$. 
Because $\ccc(U)$ is closed in $X$, we have that $F$ is an $F_\sigma$-subset of $X$. 
Therefore, Lemma~\ref{lem: some properties of P-sets and almost P-spaces}(\ref{lem Item: characterizing a P-set via F sigma sets}) implies that $A\cap \ccc(F)=\varnothing$. 
Thus, since $B\subseteq A$, we have that $B\cap\ccc(F) = \varnothing$.
Consequently, $B$ is a $P$-set in $\ccc(U)$ by Lemma~\ref{lem: some properties of P-sets and almost P-spaces}(\ref{lem Item: characterizing a P-set via F sigma sets}).
\end{proof}

\begin{lemma}\label{lem: Jan 0.3-covered by closed almost P-spaces is an almost P-space}
Let $X$ be a space and $Y_1,\dots,Y_n$ closed subspaces of $X$ for some nonzero $n\in\omega$. If 
$\bigcup_{i=1}^{n}Y_i=X$ and each $Y_i$ is an almost $P$-space, then $X$ is an almost $P$-space.
\end{lemma}

\begin{proof}
We proceed by induction on $n\ge1$. 
The base case is clear. 
Suppose that the result holds for some $n\ge1$. 
Let $Y_1,\dots,Y_{n+1}$ be closed subspaces of $X$ such that $\bigcup_{i=1}^{n+1}Y_i=X$ and each $Y_i$ is an almost $P$-space. Let 
$A=\bigcup_{i=1}^{n}Y_i$ 
and $B=Y_{n+1}$. 
Then $X=A\cup B$ and both $A$ and $B$ are closed in $X$. Moreover, $B$ is an almost $P$-space by assumption and $A$ is an almost $P$-space by the inductive hypothesis.

Let $G$ be a nonempty $G_\delta$-subset of $X$. 
Suppose that $G\cap(A\setminus B)\neq\varnothing$. As $B$ is closed in $X$, we have that $A\setminus B$ is open in $A$. Thus, $G\cap(A\setminus B)$ is a $G_\delta$-subset of $A$. 
Since $A$ is an almost $P$-space, there is an open subset $U$ of $X$ such that $\varnothing\neq U\cap A\subseteq G\cap(A\setminus B)$. Therefore, 
$U\cap A\subseteq U\setminus B$. To see the reverse inclusion, let $x\in U\setminus B$. Then $x\in A$ because $X=A\cup B$ and $x\notin B$. This gives that $x\in U\cap A$, so
$U\setminus B=U\cap A$, yielding that $U\setminus B$ is a nonempty open subset of $X$ that is contained in $G$. Thus, $G$ has nonempty interior. A symmetric argument shows that if $G\cap(B\setminus A)\neq\varnothing$ then $G$ has nonempty interior.

Finally, suppose that $G\subseteq A\cap B$. Then $G$
is a nonempty $G_\delta$-subset of $A$. Since $A$ is an almost $P$-space, there is an open subset $U$ of $X$ such that $\varnothing\neq U\cap A\subseteq G$. Therefore,
$G\cap U=(G\cap A)\cap U 
=U\cap A\neq\varnothing$. Because $G\cap U
\subseteq B$, we have that $G\cap U$
is a nonempty $G_\delta$-subset of $B$. Since $B$ is an almost $P$-space, there is an open subset $V$ of $X$ such that $\varnothing\neq V\cap B\subseteq G\cap U$. Thus,
$V\cap B\subseteq U$, and so 
$\varnothing\neq V\cap B\subseteq U\cap V$. Consequently, $U\cap V$ is a nonempty open subset of $X$. Let $x\in U\cap V$. If $x\in B$, then $x\in V\cap B\subseteq G\cap U\subseteq G$. If $x\not\in B$, then $x\in A$, which implies that $x\in U\cap A\subseteq G$. In either case, $x\in G$, yielding that $U\cap V\subseteq G$. Hence, $G$ has nonempty interior, and thus $X$ is an almost $P$-space.
\end{proof}

We also need the following lemma, which is a consequence of CH.

\begin{lemma}\label{lem: Jan 0.5-F closed in nwd P-set in omega* can be picked up via open}
\emph{({CH})} Let $A$ be a closed nowhere dense $P$-set in $\omega^*$ and $F$ a nonempty closed subset of $A$. Then there is an open subset $U$ of $\omega^*\setminus A$ such that $\ccc(U) = U\cup F$.
\end{lemma}

\begin{proof}
We recall (see the first paragraph of the proof of the First Partition Lemma)
that $\mathscr C$ is the family of all nonempty clopen subsets of $\omega^*$. Let also
$\{C_\alpha \mid \alpha<\omega_1\}$ be the family $\{C\in\mathscr C \mid C\subseteq\omega^*\setminus A\}$, 
and $\{D_\alpha \mid \alpha<\omega_1\}$ the family $\{ D\in \mathscr C \mid D\cap F\neq\varnothing \}$. 
Since $F$ is closed, for each $x\in A\setminus F$ there is $C\in \mathscr C$ such that $x\in C\subseteq \omega^*\setminus F$, and hence $C\cap A$ is compact and 
$x\in C\cap A\subseteq A\setminus F$. Therefore, there is a family $\{K_\alpha \mid \alpha<\omega_1\}$ consisting of compact sets such that $A\setminus F=\bigcup_{\alpha<\omega_1}K_\alpha$. 
(For distinct $\alpha,\beta<\omega_1$, we may have $K_\alpha=\varnothing$ or $K_\alpha=K_\beta$.)

For each $\alpha<\omega_1$, we recursively define an open $F_\sigma$-subset $P_\alpha$ of $\omega^*$ and a clopen subset $Q_\alpha$ of $\omega^*$ disjoint from $A$
satisfying the following properties:
\begin{enumerate}
\item $\beta<\alpha$ implies $P_\beta\subseteq P_\alpha$ and $Q_\beta\subseteq Q_\alpha$,
\item $P_\alpha\cap Q_\alpha = \varnothing$,
\item $C_\alpha\subseteq P_\alpha \cup Q_\alpha$,
\item $D_\alpha \cap Q_\alpha \neq\varnothing$, and
\item $K_\alpha \subseteq P_\alpha \subseteq \omega^*\setminus F$.
\end{enumerate}
Let $\alpha<\omega_1$. Suppose that $P_\beta$ and $Q_\beta$ have been defined for all $\beta<\alpha$. Set $P=\bigcup_{\beta<\alpha}P_\beta$ and $Q=\bigcup_{\beta<\alpha}Q_\beta$. Because $\alpha<\omega_1$, we have that $\alpha$ is countable. Since $P$ is a countable union of open $F_\sigma$-sets 
and $Q$ is a countable union of clopen sets, 
both $P$ and $Q$ are open $F_\sigma$-sets. 
Conditions (1) and (2) above imply that $P\cap Q=\varnothing$, from which it follows that $\ccc(P)\cap\ccc(Q)=\varnothing$ (since $\omega^*$ is a normal $F$-space). Because $Q_\beta\subseteq \omega^*\setminus A$ for each $\beta<\alpha$, we have that $Q\subseteq\omega^*\setminus A$. Thus, Lemma~\ref{lem: some properties of P-sets and almost P-spaces}(\ref{lem Item: characterizing a P-set via F sigma sets}) gives that $A\cap \ccc(Q)=\varnothing$ (since $A$ is a $P$-set in $\omega^*$). Therefore, the closed sets $A \cup \ccc(P)$ and $\ccc(Q)$ are disjoint, and hence there are disjoint clopen sets $G$ and $H'$ 
such that $A\cup\ccc(P)\subseteq G$ and $\ccc(Q)\subseteq H'$. Set $H=H'\cup (C_\alpha\setminus G)$. Then $H$ is clopen 
such that $G\cap H=\varnothing$, $\ccc(Q)\subseteq H$, and $C_\alpha\subseteq G\cup H$. 

Condition (5) implies that 
$P_\beta\subseteq\omega^*\setminus F$ for each $\beta<\alpha$, and so
$P\subseteq\omega^*\setminus F$. Moreover, 
$C_\alpha\cap G\subseteq C_\alpha\subseteq\omega^*\setminus A\subseteq \omega^*\setminus F$, yielding that
$P\cup(C_\alpha\cap G)\subseteq\omega^*\setminus F$, or equivalently, $F\subseteq\omega^*\setminus\left(P\cup(C_\alpha\cap G)\right)$. Therefore, since $P$ is an $F_\sigma$-set 
and 
$D_\alpha,C_\alpha\cap G$ are clopen, 
$D_\alpha\setminus\left(P\cup(C_\alpha\cap G)\right)$ is a $G_\delta$-set 
that is nonempty since 
$$\varnothing\neq D_\alpha \cap F\subseteq D_\alpha\setminus\left(P\cup(C_\alpha\cap G)\right).$$
Thus, the interior of $D_\alpha\setminus\left(P\cup(C_\alpha\cap G)\right)$ is nonempty because $\omega^*$ is an almost $P$-space. 
Furthermore, since $A$ is closed and nowhere dense, 
the interior of $D_\alpha\setminus\left(P\cup(C_\alpha\cap G)\right)$ is not contained in $A$, which implies that there is a nonempty clopen $R$
that is disjoint from $A$ and contained in the interior of $D_\alpha\setminus\left(P\cup(C_\alpha\cap G)\right)$. 
We now define $Q_\alpha= H\cup R$ and $P_\alpha=P\cup(C_\alpha\cap G)\cup(B\setminus Q_\alpha)$ where $B$ is clopen 
such that $K_\alpha\subseteq B\subseteq \omega^*\setminus F$.

We 
show that $P_\alpha$ and $Q_\alpha$ are as desired. Since both $H$ and $R$ are clopen 
contained in $\omega^*\setminus A$, we have that $Q_\alpha$ is clopen contained in $\omega^*\setminus A$. Because $P$ is an open $F_\sigma$-set 
and $C_\alpha\cap G,B\setminus Q_\alpha$ are clopen, 
$P_\alpha$ is an open $F_\sigma$-set. 
To see that (1) holds, observe that $P_\beta\subseteq P\subseteq P_\alpha$ and $Q_\beta\subseteq Q\subseteq\ccc(Q)\subseteq H\subseteq Q_\alpha$ for each $\beta<\alpha$. For (2), since $P \cup(C_\alpha\cap G) \subseteq G$, we have 
$\left(P\cup(C_\alpha\cap G)\right)\cap H\subseteq G\cap H=\varnothing$. Therefore, because $\left(P\cup(C_\alpha\cap G)\right)\cap R=\varnothing$, we obtain
\begin{eqnarray*}
P_\alpha\cap Q_\alpha &=& \left[(P\cup(C_\alpha\cap G)\cup(B\setminus Q_\alpha)\right]\cap Q_\alpha \\
&=& \left[\left(P\cup(C_\alpha\cap G)\right)\cap Q_\alpha\right]\cup\left[(B\setminus Q_\alpha)\cap Q_\alpha\right] \\
&=& \left[\left(P\cup(C_\alpha\cap G)\right)\cap (H\cup R)\right]\\
&=& \left[\left(P\cup(C_\alpha\cap G)\right)\cap H\right]\cup\left[\left(P\cup(C_\alpha\cap G)\right)\cap R\right] = 
\varnothing.
\end{eqnarray*}
Condition (3) is satisfied since 
\[
C_\alpha = (C_\alpha\cap G)\cup(C_\alpha\setminus G) \subseteq P_\alpha\cup H \subseteq P_\alpha\cup Q_\alpha.
\]
Because $R$ is nonempty and contained in the interior of $D_\alpha\setminus\left(P\cup(C_\alpha\cap G)\right)$, and hence 
in $D_\alpha$, we see that (4) holds since $\varnothing\neq R=D_\alpha\cap R\subseteq D_\alpha\cap Q_\alpha$.
To see 
(5), 
$B\setminus Q_\alpha\subseteq B\subseteq\omega^*\setminus F$ and $P\cup(C_\alpha\cap G)\subseteq \omega^*\setminus F$ imply that $P_\alpha \subseteq \omega^* \setminus F$. Also,
$K_\alpha\subseteq B\setminus Q_\alpha$ since $K_\alpha\subseteq B$, $K_\alpha\subseteq A\setminus F$, and $Q_\alpha\subseteq \omega^*\setminus A$, which imples that $K_\alpha\subseteq P_\alpha$. Therefore, conditions (1)--(5)
hold for $P_\alpha$ and $Q_\alpha$.

To complete the proof, set $U=\bigcup_{\alpha<\omega_1}Q_\alpha$. Because each $Q_\alpha$ is clopen disjoint from $A$, 
we have that $U$ is an open set disjoint from $A$.
Let $x\in F$ and consider an arbitrary clopen neighborhood $D$ of $x$. Then there is $\alpha_x<\omega_1$ such that $D=D_{\alpha_x}$. By condition (4), $\varnothing \neq D_{\alpha_x}\cap Q_{\alpha_x} \subseteq D\cap U$, and hence $F\subseteq\ccc(U)$, yielding that $U\cup F\subseteq\ccc(U)$. For the other inclusion it suffices to show that $U\cup F$ is closed. 

Set $V=\bigcup_{\alpha<\omega_1}P_\alpha$, an open set.
By conditions (1) and (2), $U$ and $V$ are disjoint.
By condition (5), 
$V\subseteq\omega^*\setminus F$, which implies that $F\subseteq\omega^*\setminus V$, and thus $U\cup F\subseteq\omega^*\setminus V$. Conversely, suppose that $x\not\in U\cup F$. 
If $x\in A$, then $x\in A\setminus F$, so there is $\alpha_x<\omega_1$ such that $x\in K_{\alpha_x}$. By condition (5), 
$K_{\alpha_x}\subseteq P_{\alpha_x}\subseteq
V$, which implies that $x\not\in\omega^*\setminus V$. 
If $x\notin A$,
then there is $\alpha_x<\omega_1$ such that $x\in C_{\alpha_x}$. By condition (3), 
$C_{\alpha_x}\subseteq P_{\alpha_x}\cup Q_{\alpha_x}$. Because $x\not\in U$, it follows that $x\not\in Q_{\alpha_x}$, so $x\in P_{\alpha_x}\subseteq V$, and hence 
$x\not\in\omega^*\setminus V$. Thus, $\omega^*\setminus V\subseteq U\cup F$, which proves
that $U\cup F=\omega^*\setminus V$, so $U\cup F$ is closed. 
\end{proof}

We are now ready to prove the second partition lemma. 

\begin{lemma}[Second Partition Lemma]
\label{lem: Jan 0.6-good disjoint open cover of complement of closed nwd P set}
\emph{(CH)} 
Let $A$ be a nonempty closed nowhere dense $P$-set in $\omega^*$ and $A=\bigcup\mathscr F$, where $
\mathscr F$ is a ﬁnite set of 
nonempty closed subsets of $A$. 
Then there is an open partition 
$\{U_F\mid F\in\mathscr F\}$ of $\omega^*\setminus A$ such that for each $F\in\mathscr F$, the following conditions hold:
\begin{enumerate}
\item $\ccc(U_F)=U_F\cup F$,
\item $\ccc(U_F)$ is homeomorphic to $\omega^*$, 
\item $F$ is a closed nowhere dense $P$-set in $\ccc(U_F)$.
\end{enumerate}
\end{lemma}

\begin{proof}
Let 
$\mathscr F=\{F_1,\dots,F_n\}$. By the First Partition Lemma,
there is an open partition $\{U_1,\dots,U_n\}$ of $\omega^*\setminus A$ such that $\ccc(U_i)=U_i\cup A$ for each $1\le i\le n$. 

\begin{claim}\label{clm: closure of Ui is homeo to omega star}
The subspace $\ccc(U_i)$ is homeomorphic to $\omega^*$ for each $0\le i\le n$.
\end{claim}

\begin{proof}
Because $\ccc(U_i)$ is a closed subspace of a Stone space, it is a Stone space. Moreover, $\ccc(U_i)$ is an 
$F$-space because it is a closed subset of a normal $F$-space. 
Furthermore, since $U_i$ is a nonempty open subset of a crowded space, 
$\ccc(U_i)$ is crowded and infinite. Being an infinite closed subspace of $\omega^*$, and hence of $\beta(\omega)$, it follows from \cite[Thm.~3.6.14]{Eng89} that the weight of $\ccc(U_i)$ is $\mathfrak c$, which equals $\omega_1$ by CH. 
The open subspace $U_i$ of $\omega^*$ is an almost $P$-space by Lemma~\ref{lem: some properties of P-sets and almost P-spaces}(\ref{lem Item: almost P-space is an open hereditary property}). Therefore, we may apply \cref{lem: jan lemma 0.1}(\ref{lem item: jan lemma 0.1-compact and almost P implies almost P}) to obtain that $\ccc(U_i)$ is an almost $P$-space. Hence,
$\ccc(U_i)$ is homeomorphic to $\omega^*$ by Parovi\v{c}enko's Theorem.
\end{proof}

Let $1\le i\le n$. Because $U_i$ is an open dense subset of $\ccc(U_i)=U_i\cup A$ that is contained in $\omega^*\setminus A$, we have that $A$ is a closed nowhere dense subset of $\ccc(U_i)$. Moreover, $A$ is a $P$-set in $\ccc(U_i)$ by Lemma~\ref{lem: some properties of P-sets and almost P-spaces}(\ref{lem Item: being a P-set and contained in a closed set is a P-set in the closed subspace}). Since $F_i$ is a nonempty closed subset of $A$, it follows from Claim~\ref{clm: closure of Ui is homeo to omega star} and \cref{lem: Jan 0.5-F closed in nwd P-set in omega* can be picked up via open} that there is an open subset $V_i$ of $U_i$ such that $\ccc(V_i)=V_i\cup F_i$. Because $\ccc(U_i)$ is homeomorphic to $\omega^*$, we obtain 
that $\ccc(V_i)$ is homeomorphic to $\omega^*$ by replacing each occurrence of $U_i$ in the proof of Claim~\ref{clm: closure of Ui is homeo to omega star} by $V_i$. Since $F_i$ is closed in $A$, which is a closed $P$-set in $\ccc(U_i)$, and $V_i$ is an open subset of $\ccc(U_i)\setminus A$ such that $\ccc(V_i)=V_i\cup F_i$, Lemma~\ref{lem: jans lemma 0.2} yields that $F_i$ is a closed nowhere dense $P$-set in $\ccc(V_i)$. Set $Z=A\cup\bigcup_{i=1}^nV_i$.

\begin{claim}\label{last label}
The subspace $Z$ is homeomorphic to $\omega^*$.
\end{claim}

\begin{proof}
We have that $Z$ is a closed subspace of $\omega^*$ since
$$
Z=A\cup\bigcup\nolimits_{i=1}^nV_i=\bigcup\nolimits_{i=1}^n F_i\cup\bigcup\nolimits_{i=1}^nV_i=\bigcup\nolimits_{i=1}^n(V_i\cup F_i) = \bigcup\nolimits_{i=1}^n\ccc(V_i).
$$
Therefore, $Z$ is a Stone space that is an 
$F$-space. Since for each $1\le i\le n$ we have that $V_i$ is nonempty and open in $U_i$, which is open in $\omega^*$, it follows that $\bigcup_{i=1}^nV_i$ is a nonempty open subset of $\omega^*$. Thus, since $Z=
\ccc\left(\bigcup\nolimits_{i=1}^nV_i\right)$, we have that $Z$ is crowded and infinite. As $Z$ is an infinite closed subset of $\omega^*$, it follows from \cite[Thm.~3.6.14]{Eng89} that the weight of $Z$ is $\omega_1$. Recalling that $\ccc(V_i)$ is homeomorphic to $\omega^*$, and hence is an almost $P$-space for each $1\le i\le n$, we have that $Z$ is an almost $P$-space by \cref{lem: Jan 0.3-covered by closed almost P-spaces is an almost P-space}. Applying Parovi\v{c}enko's Theorem the yields that $Z$ is homeomorphic to $\omega^*$.
\end{proof}

We observed above that $Z=
\ccc\left(\bigcup\nolimits_{i=1}^nV_i\right)$. Thus, $\bigcup\nolimits_{i=1}^nV_i$ is an open dense subset of $Z$, giving that $A$, which is its complement 
in $Z$, is a closed nowhere dense subset of $Z$. By \cref{lem: some properties of P-sets and almost P-spaces}(\ref{lem Item: being a P-set and contained in a closed set is a P-set in the closed subspace}), $A$ is a $P$-set in $Z$. 
By \cref{last label}, there is a homeomorphism  $F:Z\to\omega^*$. 
Therefore, the image $F(A)$ of $A$ is a closed nowhere dense $P$-set in $\omega^*$ that is homeomorphic to $A$ via the restriction $f:A\to F(A)$ of $F$. Let $g:F(A)\to A$ be the inverse 
of $f$.
By the Homeomorphism Extension Theorem,
there is a homeomorphism $G:\omega^*\to\omega^*$ extending $g$; see \cref{mapping figure}. 

\begin{figure}[h]
\begin{center}
\begin{picture}(375,70)
\multiput(0,0)(125,0){3}{\multiput(0,0)(75,0){2}{\line(0,1){50}}}
\multiput(0,0)(125,0){3}{\multiput(0,0)(0,50){2}{\line(1,0){75}}}
\multiput(3,33.3)(3,0){24}{\multiput(0,0)(250,0){2}{\makebox(0,0){.}}}
\multiput(0,25)(125,0){3}{\oval(70,50)[r]}
\multiput(30,20)(125,0){2}{\vector(1,0){105}}
\multiput(50,60)(125,0){2}{\vector(1,0){97.5}}
\put(37.5,60){\makebox(0,0){$Z$}}
\multiput(162.5,60)(125,0){2}{\makebox(0,0){$\omega^*$}}
\multiput(14.5,10)(250,0){2}{\makebox(0,0){\small $A$}}
\put(139.5,10){\makebox(0,0){\small $F(A)$}}
\multiput(15,42)(250,0){2}{\makebox(0,0){\footnotesize $F_1$}}
\put(55,42){\makebox(0,0){\footnotesize $V_1$}}
\put(305,42){\makebox(0,0){\footnotesize $GF[V_1]$}}
\put(95,12.5){\makebox(0,0){\footnotesize$f$}}
\put(225,12.5){\makebox(0,0){\footnotesize$g=f^{-1}$}}
\put(95,70){\makebox(0,0){$F$}}
\put(220,70){\makebox(0,0){$G$}}
\end{picture}
\end{center}
\caption{The mapping $G:\omega^*\to\omega^*$.}
\label{mapping figure}
\end{figure}
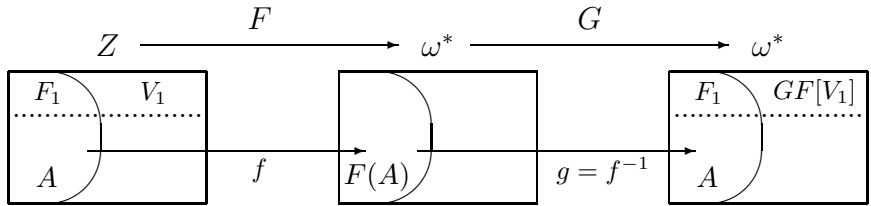

The mapping $G\circ F:Z\to \omega^*$ is a homeomorphism such that $GF(x)=x$ for each $x\in A$. Let $F\in\mathscr F$. Then
$F=F_i$ for some $1\le i\le n$. Set $U_F=GF[V_i]$. The family $\{U_F \mid F\in\mathscr F\}$ is then as desired because
$G\circ F$ is a homeomorphism that is the identity on $A$. 
\end{proof}

\section{Mapping lemmas for $\omega^*$}\label{sec: mappings of omega star}

In this 
section we utilize two partition lemmas of the previous section 
to prove two mapping lemmas for $\omega^*$. 
As we will see in the next section, these mapping lemmas are instrumental in proving that each willow tree is an interior image of $\beta(\omega^2)$. 

\subsection{The first mapping lemma}

We will be interested in special interior maps whose domain is homeomorphic to $\omega^*$. 

\begin{definition}\label{def: ws-interior map}
Let $\mathfrak T=(W,\le)$ be a tree, $X$ a space that is homeomorphic to $\omega^*$, and $f:X\to W$ an interior mapping of $X$ onto $\mathfrak T$. We call $f$ \emph{well suited}, or simply a {\em ws-map}, provided the following two conditions hold for each $w,v\in W$:
\begin{enumerate}
\item $f^{-1}({\downarrow}w)$ is a $P$-set in $X$ homeomorphic to $\omega^*$; 
\item if 
$v<w$, then $f^{-1}({\downarrow}v)$ is a $P$-set in $f^{-1}({\downarrow}w)$. 
\end{enumerate}
\end{definition}

The following lemma reduces the work required to see that a mapping is well suited. For this we recall that if $v$ is the 
immediate predecessor of $w$ in a tree, we say that $v$ is the \emph{parent} of $w$ and
$w$ is a \emph{child} of $v$.

\begin{lemma}
\label{lem: OLD ws-condition 2}
Let $f$ be an interior mapping of a space $X$ homeomorphic to $\omega^*$ onto a tree $\mathfrak T=(W,\le)$. Then $f$ is a ws-map iff the following conditions hold for each $w,v\in W:$
\begin{enumerate}
\item 
$f^{-1}({\downarrow}w)$ is homeomorphic to $\omega^*;$ \label{10.2(1)}
\item if $w\in W\setminus\mathrm{max}(\mathfrak T)$, then $f^{-1}({\downarrow}w)$ is a $P$-set in $X;$ \label{10.2(2)} 
\item if $v<w$, then $f^{-1}({\downarrow}v)$ is a 
$P$-set in $f^{-1}({\downarrow}w).$ \label{10.2(3)}
\end{enumerate}
\end{lemma}

\begin{proof}
The left-to-right implication is obvious. For the other implication, we only need to verify that $f^{-1}({\downarrow}w)$ is a $P$-set in $X$ whenever $w\in\mathrm{max}(\mathfrak T)$. 
If $w$ is the root of $\mathfrak T$, then $f^{-1}({\downarrow}w)=f^{-1}(W)=X$ is clearly a $P$-set in $X$. 
Otherwise $w$ has the parent $v\in 
W\setminus\mathrm{max}(\mathfrak T)$, and hence $f^{-1}({\downarrow}v)$ is a $P$-set, 
which is closed in $X$ since ${\downarrow}v$ is closed in $W$ and $f$ is an interior map. As $w\in \mathrm{max}(\mathfrak T)$, we have that $f^{-1}(w)$ is an open subset of $X$ that is contained in $X\setminus f^{-1}({\downarrow}v)$. Because $f$ is interior, we have 
\[
\ccc(f^{-1}(w)) = f^{-1}({\downarrow}w) = f^{-1}(\{w\}\cup{\downarrow}v) = f^{-1}(w)\cup f^{-1}({\downarrow}v).
\]
Thus, \cref{lem: jan lemma 0.1}(\ref{lem item: jan lemma 0.1-closure of U is P-set}) applied to $A=f^{-1}({\downarrow}v)$ and $U=f^{-1}(w)$ yields $f^{-1}({\downarrow}w)
$ is a $P$-set.
\end{proof}

\begin{lemma}\label{lem: closed P-set is a transitive property for normal spaces}
Let $X$ be a normal space, $A$ a closed $P$-set in $X$, and $B$ a closed $P$-set in $A$. Then $B$ is a closed $P$-set in $X$.
\end{lemma}

\begin{proof}
By Lemma~\ref{lem: some properties of P-sets and almost P-spaces}(\ref{lem Item: characterizing a P-set via F sigma sets}), it is sufficient to show that if $F$ is an $F_\sigma$-subset of $X$ disjoint from $B$, then $B\cap\ccc(F)=\varnothing$. Since $A$ is closed, 
$F\cap A$ is an $F_\sigma$-subset of $A$. 
Because $B$ is a $P$-set in $A$, 
Lemma~\ref{lem: some properties of P-sets and almost P-spaces}(\ref{lem Item: characterizing a P-set via F sigma sets}) implies that $B\cap\ccc(F\cap A)=\varnothing$. 
Since $X$ is normal, there are disjoint open subsets $U$ and $V$ of $X$ such that $\ccc(F\cap A)\subseteq U$ and $B\subseteq V$. Therefore,
$B\cap \ccc(U) = \varnothing$. Because $F$ is an $F_\sigma$-subset of $X$ and $U$ is open, 
$F\setminus U$ is an $F_\sigma$-subset of $X$. Moreover, 
\[
A \cap \left(F\setminus U\right) = (F\cap A)\setminus U 
= \varnothing.
\]
Thus, $A\cap \ccc\left(F\setminus U\right)=\varnothing$ by Lemma~\ref{lem: some properties of P-sets and almost P-spaces}(\ref{lem Item: characterizing a P-set via F sigma sets}), and hence  
\begin{eqnarray*}
B\cap\ccc(F) &=& B\cap\ccc\left((F\cap U)\cup(F\setminus U)\right) \\
&=& B\cap\left(\ccc(F\cap U)\cup\ccc(F\setminus U)\right) \\
&=& \left(B\cap\ccc(F\cap U)\right)\cup\left(B\cap\ccc(F\setminus U)\right) \\
&\subseteq& \left(B\cap\ccc(U)\right)\cup\left(A\cap\ccc(F\setminus U)\right) 
= \varnothing. 
\end{eqnarray*}
\end{proof}

It follows from \cite[Thm.~12]{Dow19} that already in ZFC each tree is an interior image of $\omega^*$. Under CH, the First Mapping Lemma implies that we may assume such interior maps are well suited.  

\begin{lemma}[First Mapping Lemma]\label{lem: Jans 0.7-good interior mapping of omega star onto a finite tree}
\emph{(CH)} 
Let $\mathfrak T = (W,\le)$ be a tree. Then there is a ws-mapping $f:\omega^* \to \mathfrak T$.
\end{lemma}

\begin{proof}
The proof is by strong induction on the depth of $\mathfrak T$. 

{\bf Base case:} Suppose the depth of $\mathfrak T$ is 1. Then $W$ consists of a single point, say $W=\{w\}$, and there is exactly one function $f:\omega^*\to W$, which is clearly an onto interior map that satisfies 
\cref{lem: OLD ws-condition 2}.
Thus, $f$ is a ws-map.

{\bf Inductive case:} Suppose the depth of $\mathfrak T$ is greater than $1$ and 
there is a ws-map of $\omega^*$ onto any tree of depth less than the depth of $\mathfrak T$. By \cite[Lem.~1.4.3]{vMi84}, there is a closed nowhere dense $P$-set $A$ in $\omega^*$ that is homeomorphic to $\omega^*$. Let $\mathfrak S=(V,\le)$ be the subtree of $\mathfrak T$ where $V=W\setminus\mathrm{max}(\mathfrak T)$. Because the depth of $\mathfrak S$ is less than that 
of $\mathfrak T$ and $A$ is homeomorphic to $\omega^*$, 
the inductive hypothesis yields a ws-map $g:A\to V$. 

Let $v\in\mathrm{max}(\mathfrak S)$. Since $g$ is a ws-map, $g^{-1}({\downarrow}v)$ is a closed $P$-set in $A$ homeomorphic to $\omega^*$. Because $A$ is closed nowhere dense in $\omega^*$,  so is $g^{-1}({\downarrow}v)$. Moreover, \cref{lem: closed P-set is a transitive property for normal spaces} yields that $g^{-1}({\downarrow}v)$ is a 
$P$-set in $\omega^*$. 
Since $A=\bigcup\{g^{-1}({\downarrow}v)\mid v\in\max(\G)\}$, the Second Partition Lemma delivers a pairwise disjoint open cover $\{U_v\mid v\in\mathrm{max}(\mathfrak S)\}$ of $\omega^*\setminus A$ such that $\ccc(U_v)=U_v\cup g^{-1}({\downarrow}v)$, $\ccc(U_v)$ is homeomorphic to $\omega^*$, and $g^{-1}({\downarrow}v)$ is a closed nowhere dense $P$-set in $\ccc(U_v)$ for each $v\in\mathrm{max}(\mathfrak S)$.

For $v\in\mathrm{max}(\mathfrak S)$, the set ${\uparrow}v\setminus\{v\}$ of children 
of $v$
is a nonempty finite subset of $\mathrm{max}(\mathfrak T)$. Since $\ccc(U_v)$ is homeomorphic to $\omega^*$ and $g^{-1}({\downarrow}v)$ is a nonempty closed nowhere dense $P$-set in $\ccc(U_v)$, 
the First Partition Lemma
implies that there is an open partition $\{U_{w}\mid w\in{\uparrow}v\setminus\{v\}\}$ of $\ccc(U_v)\setminus g^{-1}({\downarrow}v)=U_v$ such that $\ccc(U_{w})=U_{w}\cup g^{-1}({\downarrow}v)$ for each $w\in{\uparrow}v\setminus \{v\}$. Because each $w\in\mathrm{max}(\mathfrak T)$ has a unique parent ${\sf p}_w\in\mathrm{max}(\mathfrak S)$, we have an open partition $\{U_w\mid w\in\mathrm{max}(\mathfrak T)\}$ of $\omega^*\setminus A$ such that $\ccc(U_w)=U_w\cup g^{-1}({\downarrow}{\sf p}_w)$ for each $w\in\mathrm{max}(\mathfrak T)$. Hence, we have a well-defined onto map $f:\omega^*\to W$ given by
\[
f(x) = \left\{
\begin{array}{ll}
g(x) & \text{if }x\in A,\\
w & \text{if }x\in U_w\text{ for some }w\in\mathrm{max}(\mathfrak T).
\end{array}
\right.
\]

We show that $f$ is a ws-map. To see that $f$ is continuous, let $w\in W$. 
If $w\in\mathrm{max}(\mathfrak T)$, then 
\[
f^{-1}({\downarrow} w) = f^{-1}(\{w\}\cup{\downarrow}{\sf p}_w) = U_w \cup g^{-1}({\downarrow}{\sf p}_w) = \ccc(U_w)
\]
is closed in $\omega^*$. If $w\not\in\mathrm{max}(\mathfrak T)$, then $w\in V$ and $f^{-1}({\downarrow} w) = g^{-1}({\downarrow}w)$ is closed in $A$, 
hence closed in $\omega^*$. 

To see that $f$ is open, let $C$ be a clopen subset of $\omega^*$, $u\in f(C)$, and $u\le w$. If $u\in\mathrm{max}(\mathfrak T)$, then $w=u\in f(C)$. 
If $u\not\in\mathrm{max}(\mathfrak T)$, then $u\in V$, which implies that $u\in f(C\cap A)=g(C\cap A)$.  
Since $g$ is interior and $C\cap A$ is open in $A$, we have that $g(C\cap A)$ is an upset in $\mathfrak S$. Thus, if $w\not\in\mathrm{max}(\mathfrak T)$, then $w\in V$ yielding that $w\in g(C\cap A)\subseteq f(C)$. 
If $w\in\mathrm{max}(\mathfrak T)$, the parent ${\sf p}_w$ of $w$ is in $\mathrm{max}(\mathfrak S)$ and we must have $u\le{\sf p}_w$ because $\mathfrak T$ is a tree. Therefore, ${\sf p}_w\in g(C\cap A)$ and there is $x\in C\cap A$ such that $g(x)={\sf p}_w\in{\downarrow}{\sf p}_w$, which yields that $C\cap g^{-1}({\downarrow}{\sf p}_w)\neq\varnothing$. Since $g^{-1}({\downarrow}{\sf p}_w)\subseteq \ccc(U_w)$, we obtain that $C\cap \ccc(U_w)\neq\varnothing$. Consequently,
$C\cap U_w\neq\varnothing$, which yields that 
$w\in f(C\cap U_w)\subseteq f(C)$, showing that $f(C)$ is an upset of $\mathfrak T$. Thus, $f$ is an interior mapping. 

To verify \cref{lem: OLD ws-condition 2}(\ref{10.2(1)}), 
let $w\in W$. 
If $w\not\in\mathrm{max}(\mathfrak T)$, then $w\in V$ and so
$f^{-1}({\downarrow}w)=g^{-1}({\downarrow}w)$ is homeomorphic to $\omega^*$ since $g$ is a ws-map. 
If $w\in\mathrm{max}(\mathfrak T)$, then 
\[
f^{-1}({\downarrow}w)=\ccc(f^{-1}(w))=\ccc(U_w)=U_w\cup g^{-1}({\downarrow}{\sf p}_w),
\]
showing that $f^{-1}({\downarrow}w)$ is the closure of a nonempty open subset of $\omega^*$. 
Thus, the subspace $f^{-1}({\downarrow}w)$ is a 
crowded Stone space that is an $F$-space of 
weight $\omega_1$. Because $f^{-1}(w)$ is open in $\omega^*$, it is an almost $P$-space by \cref{lem: some properties of P-sets and almost P-spaces}(\ref{lem Item: almost P-space is an open hereditary property}). Since $g^{-1}({\sf p}_w)$ is a $P$-set in $\omega^*$, it is a $P$-set in $f^{-1}({\downarrow}w)$ by \cref{lem: some properties of P-sets and almost P-spaces}(\ref{lem Item: being a P-set and contained in a closed set is a P-set in the closed subspace}). Therefore, $f^{-1}({\downarrow}w)$ is an almost $P$-space by \cref{lem: jan lemma 0.1}(\ref{lem item: jan lemma 0.1-compact and almost P implies almost P}), 
and hence $f^{-1}({\downarrow}w)$ is homeomorphic to $\omega^*$ by Parovi\v{c}enko's Theorem.

To see that \cref{lem: OLD ws-condition 2}(\ref{10.2(2)}) holds, let $w\in W\setminus\mathrm{max}(\mathfrak T)$. Then $w\in V$ and $f^{-1}({\downarrow}w)=g^{-1}({\downarrow}w)$. Thus, it is sufficient to show that $g^{-1}({\downarrow}w)$ is a $P$-set in $\omega^*$. But his follows from \cref{lem: closed P-set is a transitive property for normal spaces} since 
$g^{-1}({\downarrow}w)$ is a closed $P$-set in $A$. 

Lastly, we show that $f$ satisfies \cref{lem: OLD ws-condition 2}(\ref{10.2(3)}). Let $w,v\in W$ be such that $v<w$. If $w\in V$, then $v\in V$
and the result follows since $g$ is a ws-map. Suppose
that $w\in\mathrm{max}(\mathfrak T)$. Recalling that ${\sf p}_w$ denotes the parent of $w$, we have that $v\le {\sf p}_w$ since $\mathfrak T$ is a tree. Because $g$ is a ws-map and ${\sf p}_w\in V$, we obtain that
$f^{-1}({\downarrow}v)=g^{-1}({\downarrow}v)$ is a closed 
$P$-set in $g^{-1}({\downarrow}{\sf p}_w)$. 
Since $f^{-1}({\downarrow}w)$ is closed in $\omega^*$ and contains $g^{-1}({\downarrow}{\sf p}_w)$, which is a $P$-set in $\omega^*$, we have that $g^{-1}({\downarrow}{\sf p}_w)$ is a $P$-set in $f^{-1}({\downarrow}w)$ by Lemma~\ref{lem: some properties of P-sets and almost P-spaces}(\ref{lem Item: being a P-set and contained in a closed set is a P-set in the closed subspace}). 
Therefore, \cref{lem: closed P-set is a transitive property for normal spaces} yields that $f^{-1}({\downarrow}v)$ is a 
$P$-set in $f^{-1}({\downarrow}w)$. Thus, $f$ is a ws-map. 
\end{proof}

\subsection{The second mapping lemma}

\begin{lemma}
\label{lem: Jan's 0.8-inverse image of downset of subset of tree is homeo to omega star}
\emph{(CH)} Let $\mathfrak T=(W,\le)$ be a 
tree, $f:\omega^*\to W$ a ws-mapping, and $V\subseteq W$ nonempty. Then 
$\bigcup_{w\in V}f^{-1}({\downarrow}w)$ is homeomorphic to $\omega^*$. \label{lem item: inv img of down is homeo omega star}
\end{lemma}

\begin{proof} 
It suffices to show that $\bigcup_{w\in V}f^{-1}({\downarrow}w)$ satisfies the conditions of Parovi\v{c}enko's Theorem.
Since $V$ is a nonempty finite set and $f$ is interior, 
$\bigcup_{w\in V}f^{-1}({\downarrow}w)$ is a nonempty closed subset of $\omega^*$. Therefore, the subspace $\bigcup_{w\in V}f^{-1}({\downarrow}w)$ of $\omega^*$ is a Stone space that is an 
$F$-space. Moreover, since $f$ is a ws-map, we have for each $w\in V$ that the subspace $f^{-1}({\downarrow}w)$ of $\omega^*$ is homeomorphic to $\omega^*$, and hence $f^{-1}({\downarrow}w)$ is a crowded almost $P$-space of weight $\omega_1$. Thus, 
$\bigcup_{w\in V}f^{-1}({\downarrow}w)$ is a crowded space of weight $\omega_1$. Furthermore, since $V$ is finite and $f^{-1}({\downarrow}w)$ is closed in $\bigcup_{w\in V}f^{-1}({\downarrow}w)$ for each $w\in V$, we have that $\bigcup_{w\in V}f^{-1}({\downarrow}w)$ is an almost $P$-space by \cref{lem: Jan 0.3-covered by closed almost P-spaces is an almost P-space}.
\end{proof}

\begin{lemma}\label{lem: Jans 0.9-commuting diagram for good interior maps}
\emph{(CH)} Let $\mathfrak T=(W,\le)$ be a 
tree and $g,h:\omega^* \to \mathfrak T$ ws-maps. Then there is a homeomorphism $\varphi:\omega^* \to \omega^*$ such that $h=g\circ \varphi$.
\[
\begin{tikzcd}
\omega^* 
\arrow[rr, dashed,"\varphi"] 
\arrow[dr, "h"']&& 
\omega^*
\arrow[dl, "g"]  
\\
&  
\mathfrak T 
&
\end{tikzcd}
\]
\end{lemma}

\begin{proof}
The proof is by strong induction on the depth of $\mathfrak T$. 

{\bf Base case:} Suppose the depth of $\mathfrak T$ is 1. Then $W$ consists of a single point, 
so $g=h$ and we may take $\varphi$ to be the identity 
on $\omega^*$. 

{\bf Inductive case:} Suppose the depth of $\mathfrak T$ is $>1$ and the result holds for any tree of depth less than that
of $\mathfrak T$. Let $\mathfrak S=(V,\le)$ be the subtree of $\mathfrak T$ where $V=W\setminus\mathrm{max}(\mathfrak T)$. Set $A=g^{-1}(V)$ and $B=h^{-1}(V)$. 
By \cref{lem: Jan's 0.8-inverse image of downset of subset of tree is homeo to omega star}, both $A$ and $B$ are homeomorphic to $\omega^*$. Let $G:A\to V$ and $H:B\to V$ be the restrictions of $g$ and $h$, respectively. Since $g$ and $h$ are well suited, 
both $G$ and $H$ are ws-mappings onto $\mathfrak S$.
Therefore, the inductive hypothesis yields a homeomorphism $\psi:B\to A$ such that $H=G\circ\psi$.
\[
\begin{tikzcd}
B 
\arrow[rr, 
"\psi"] 
\arrow[dr, "H"']&& 
A
\arrow[dl, "G"]  
\\
&  
\mathfrak S
&
\end{tikzcd}
\] 
Let $w\in\mathrm{max}(\mathfrak T)$ and $v$ be the parent of $w$. Since $g$ 
and $h$ are ws-maps, 
$g^{-1}({\downarrow}w)=\ccc(g^{-1}(w))$ and $h^{-1}({\downarrow}w)=\ccc(h^{-1}(w))$ are closed subsets of $\omega^*$ that are both 
homeomorphic to $\omega^*$, 
$g^{-1}({\downarrow}v)$ is a closed nowhere dense $P$-set in $g^{-1}({\downarrow}w)$, 
and $h^{-1}({\downarrow}v)$ is a closed nowhere dense $P$-set in $h^{-1}({\downarrow}w)$.  
Because ${\downarrow}v\subseteq V$, we have 
\[
\psi^{-1}(g^{-1}({\downarrow}v))=\psi^{-1}(G^{-1}({\downarrow}v))=H^{-1}({\downarrow}v)=h^{-1}({\downarrow}v).
\]
Thus, the homeomorphism obtained by restricting the domain of $\psi$ to $g^{-1}({\downarrow}v)$ and codomain to $h^{-1}({\downarrow}v)$ can be extended by the Homeomorphism Extension Theorem
to a homeomorphism $\psi_w:h^{-1}({\downarrow}w)\to g^{-1}({\downarrow}w)$. 
Define $\varphi:\omega^*\to \omega^*$ by
\[
\varphi(x) = \left\{
\begin{array}{ll}
\psi(x) & \text{if }x\in B,\\
\psi_w(x) & \text{if } x\in h^{-1}(w)\text{ for some }w\in\mathrm{max}(\mathfrak T).
\end{array}
\right.
\]
Then $\varphi$ is a well-defined bijection since $\{B\}\cup\{h^{-1}(w)\mid w\in\mathrm{max}(\mathfrak T)\}$ and  $\{A\}\cup\{g^{-1}(w)\mid w\in\mathrm{max}(\mathfrak T)\}$ are partitions of $\omega^*$, 
$\psi$ is a bijection such that $\psi(B)=A$, and $\psi_w$ is a bijection such that $\psi_w(h^{-1}(w))=g^{-1}(w)$ for each $w\in\mathrm{max}(\mathfrak T)$. Moreover, if $x\in B$ then 
$\varphi(x)=\psi(x)\in A$, which yields that
\[
h(x) = H(x) = 
G(\psi(x)) = g(\psi(x)) = g(\varphi(x)). 
\]
Suppose that $x\in h^{-1}(w)$ for some $w\in\mathrm{max}(\mathfrak T)$. Then $\varphi(x)=\psi_w(x)\in g^{-1}(w)$, and so
\[
h(x)= w = g(\varphi(x)). 
\]
Thus, $h=g\circ\varphi$. 

To see that $\varphi$ is continuous, let $C$ be a closed subset of $\omega^*$. Because $\mathfrak T$ is finite, we have that 
$
\omega^*=g^{-1}(W)=g^{-1}\left(\bigcup_{w\in\mathrm{max}(\mathfrak T)}{\downarrow}w\right)
=\bigcup_{w\in\mathrm{max}(\mathfrak T)}g^{-1}\left({\downarrow}w\right)
$, which yields that 
\[
\begin{array}{rcl}
\varphi^{-1}(C) &=& \varphi^{-1}\left(C\cap \omega^*\right) \\
&=&\varphi^{-1}\left(C\cap\bigcup\nolimits_{w\in\mathrm{max}(\mathfrak T)}g^{-1}({\downarrow}w)\right)\\
&=&\varphi^{-1}\left(\bigcup\nolimits_{w\in\mathrm{max}(\mathfrak T)}(C\cap g^{-1}({\downarrow}w))\right) \\
&=& 
\bigcup\nolimits_{w\in\mathrm{max}(\mathfrak T)}\varphi^{-1}(C\cap g^{-1}({\downarrow}w)) \\
&=& 
\bigcup\nolimits_{w\in\mathrm{max}(\mathfrak T)}\psi_w^{-1}(C\cap g^{-1}({\downarrow}w)).
\end{array}
\] 
For each $w\in\mathrm{max}(\mathfrak T)$, since $\psi_w$ is a homeomorphism and  $C\cap g^{-1}({\downarrow}w)$ is closed in $g^{-1}({\downarrow}w)$, we have that  
$\psi_w^{-1}(C\cap g^{-1}({\downarrow}w))$ is closed in $h^{-1}({\downarrow}w)$. Hence, $\psi_w^{-1}(C\cap g^{-1}({\downarrow}w))$ is closed in $\omega^*$ because $h^{-1}({\downarrow}w)$ is closed in $\omega^*$. Therefore, $\varphi^{-1}(C)$ is closed in $\omega^*$ as $\varphi^{-1}(C)$ is a finite union of closed subsets of $\omega^*$. Thus, $\varphi$ is a continuous bijection, hence a homeomorphism. 
\end{proof}

\begin{lemma}
[Second Mapping Lemma]
\label{lem: Jans 0.10-extending good interior maps}
\emph{(CH)} Let $\mathfrak T=(W,\le)$ be a 
tree, $V$ a downset of $\mathfrak T$ contained in $W\setminus\mathrm{max}(\mathfrak T)$, and $\mathfrak S$ the subtree $(V,\le)$ of $\mathfrak T$. If $A$ is a closed nowhere dense $P$-set in $\omega^*$ that is homeomorphic to $\omega^*$ and $g:A\to \mathfrak S$ is a ws-map, then $g$ can be extended to a ws-map $f:\omega^*\to \mathfrak T$ such that $f^{-1}(W\setminus V) = \omega^*\setminus A$.
\end{lemma}

\begin{proof}
The First Mapping Lemma
yields a ws-map $h:\omega^*\to \mathfrak T$. Set $B=h^{-1}(V)$. Because $h$ is an onto interior mapping and $V$ is a downset of $\mathfrak T$ that is disjoint from $\mathrm{max}(\mathfrak T)$, we have that $B$ is a closed nowhere dense subset of $\omega^*$. Since $V\subseteq W\setminus\mathrm{max}(\mathfrak T)$ and $h$ is a ws-map, 
$h^{-1}({\downarrow}v)$ is a $P$-set in $\omega^*$ for each $v\in V$. Thus, $B=\bigcup_{v\in V}h^{-1}({\downarrow}v)$ is a $P$-set in $\omega^*$ by Lemma~\ref{lem: some properties of P-sets and almost P-spaces}(\ref{lem Item: finite union of P-sets is a P-set}). Moreover, $B$ is homeomorphic to $\omega^*$ by \cref{lem: Jan's 0.8-inverse image of downset of subset of tree is homeo to omega star}.
 
Let $H:B\to V$ be the restriction of $h$. Then $H$ is a ws-mapping of $B$ onto $\mathfrak S$.  
Because $A$ and $B$ are homeomorphic to $\omega^*$, 
\cref{lem: Jans 0.9-commuting diagram for good interior maps} yields
a homeomorphism $\psi:A\to B$ such that $g=H\circ\psi$. Since $A$ and $B$ are closed nowhere dense $P$-sets in $\omega^*$, the Homeomorphism Extension Theorem
produces a homeomorphism $\varphi:\omega^*\to\omega^*$ that extends $\psi$. 

Set $f=h\circ\varphi$. Because $\varphi$ is a homeomorphism and $h$ is an onto interior mapping, 
$f$ is an interior mapping of $\omega^*$ onto $\mathfrak T$. For each $x\in A$, we have 
\[
f(x) = 
h(\varphi(x)) = h(\psi(x)) =H(\psi(x)) = 
g(x),
\]
showing that $f$ extends $g$. Furthermore,
\begin{eqnarray*}
f^{-1}(W\setminus V) &=& (h\circ \varphi)^{-1}(W\setminus V) = \varphi^{-1}(h^{-1}(W\setminus V)) =\varphi^{-1}(\omega^*\setminus h^{-1}(V)) \\
&=& \varphi^{-1}(\omega^*\setminus B) = \omega^*\setminus\varphi^{-1}(B) = \omega^*\setminus\psi^{-1}(B) = \omega^*\setminus A.
\end{eqnarray*}
Finally, since $h$ is a ws-map and $\varphi$ is a homeomorphism, their composition $f$ is a ws-map.  
\end{proof}

\section{The logic of $\beta(\omega^2)$}\label{logic of beta omega squared}

We are finally ready to show that each willow tree is an interior image of $\beta(\omega^2)$.
We first show that each partially ordered willow tree is an interior image of $\beta(\omega^2)$, and then 
generalize this result to arbitrary willow trees.

\subsection{Partially ordered willow trees}

We start by recalling that for each open subset $U$ of a Tychnoff space $X$, there is the greatest open subset $\mathrm{Ex}(U)$ of $\beta(X)$ whose intersection with $X$ is equal to $U$, namely 
$
\mathrm{Ex}(U)=\beta(X)\setminus\ccc(X\setminus U)
$ 
(see, e.g., \cite[p.~388]{Eng89}). 

\begin{lemma}\label{lem: properties of Ex}
\cite[Lem.~7.1.13]{Eng89}
Let $U,V$ be open subsets of a Tychonoff space $X$.
\begin{enumerate}
\item $\mathrm{Ex}(U\cap V)=\mathrm{Ex}(U)\cap\mathrm{Ex}(V)$.
\item If $X$ is normal, then $\mathrm{Ex}(U\cup V)=\mathrm{Ex}(U)\cup\mathrm{Ex}(V)$.
\end{enumerate}
\end{lemma}

As in \cref{sec: beta omega squared and 3 forbid frames}, 
we identify $\omega^2$ with the product space $(\omega+1)\times\omega$. We also recall that $A$ is the set of isolated points, 
$B$ is the set of limit points, 
$B^*=\ccc(B)\setminus B$, and $V_n=\{n\}\times\omega$ (see \cref{fig: notations for subsets of omega squared}).

\begin{definition}\label{def: partitions Q and P}
Let $n\in\omega$ be nonzero. For each $1\le i\le n$, define
\begin{equation*}
A_i = \bigcup \{V_k\mid k\in\omega \text{ and } k \equiv i \bmod n\}\text{, }U_i=\mathrm{Ex}(A_i)\text{, and }
W_i = U_i\setminus A_i.
\end{equation*}
Set $\mathcal Q=\{A_1,\dots,A_n,B\}$ and $\mathcal P=\{U_1,\dots,U_n,\ccc(B)\}$ (see \cref{fig: partitions of X and beta X} for $n=2$). 
\begin{figure}[h]
\begin{center}
\begin{picture}(212.5,105)(0,10)
\multiput(127.5,22.5)(77.5,0){2}{{\line(0,1){77.5}}}
\multiput(127.5,22.5)(0,77.5){2}{{\line(1,0){77.5}}}
\multiput(122.5,22.5)(-15,0){2}{{\line(0,1){92.5}}}
\multiput(122.5,22.5)(0,92.5){2}{{\line(-1,0){15}}}
\put(127.5,40){\line(1,0){77.5}}
\put(127.5,82.50){\line(1,0){77.5}}
\put(166,60){\makebox(0,0){\small$B^*=\ccc (B)\setminus B$}}
\put(166,90.5){\makebox(0,0){\small$W_1=U_1\setminus A_1$}}
\put(166,31){\makebox(0,0){\small$W_2=U_2\setminus A_2$}}
\put(25,15){\makebox(0,0){$A_2$}}
\put(10,107.5){\makebox(0,0){$A_1$}}
\put(35,15){\line(1,0){47.5}}
\put(20,107.5){\line(1,0){47.5}}
\multiput(45,15)(30,0){2}{\vector(0,1){10}}
\multiput(30,107.5)(30,0){2}{\vector(0,-1){10}}
\multiput(115,30)(0,15){4}{{\makebox(0,0){$\bullet$}}}
\put(115,92.5){\makebox(0,0){$\vdots$}}
\put(115, 107.5){\makebox(0,0){$B$}}
\multiput(30,30)(15,0){4}{\multiput(0,0)(0,15){4}{\makebox(0,0){$\bullet$}}}
\put(96,52.5){\makebox(0,0){$\cdots$}}
\put(96,15){\makebox(0,0){$\cdots$}}
\put(81,107.5){\makebox(0,0){$\cdots$}}
\multiput(30,92.5)(15,0){4}{\makebox(0,0){$\vdots$}}
\end{picture}
\end{center}
\caption{The partitions $\mathcal Q$ and $\mathcal P$ for $n=2$.}
\label{fig: partitions of X and beta X}
\end{figure}
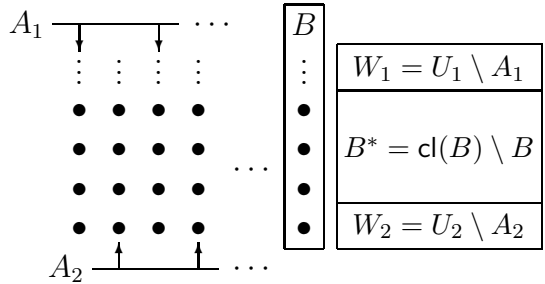
\end{definition}

\begin{lemma}
For each nonzero $n\in\omega$, we have that $\mathcal Q$ is a partition of $\omega^2$ and $\mathcal P$ 
is a partition of $\beta(\omega^2)$.
\end{lemma}

\begin{proof}
It is straightforward to see
that $\mathcal Q$ is a partition of $\omega^2$ such that $A_i$ is open in $\omega^2$ and $\ccc(A_i)\cap\omega^2=A_i\cup B$ for each $i=1,\dots,n$. 
Because $\omega^2$ is normal, 
\cref{lem: properties of Ex} implies that  
\[
\bigcup_{i=1}^nU_i = \bigcup_{i=1}^n \mathrm{Ex}(A_i) = \mathrm{Ex}\left(\bigcup_{i=1}^nA_i\right) 
= \mathrm{Ex}(\omega^2\setminus B) = \beta(\omega^2)\setminus \ccc (B)
\] 
and
\[
U_i\cap U_j = \mathrm{Ex}(A_i)\cap \mathrm{Ex}(A_j) = \mathrm{Ex}(A_i\cap A_j) = \mathrm{Ex}(\varnothing)=
\varnothing
\]
for distinct $i,j$. Thus, 
$\mathcal P$ is a partition of $\beta(\omega^2)$. 
\end{proof}

\begin{lemma}\label{lem: closure of Ui and Wi}
\emph{(CH)} Let $n\in\omega$ be nonzero and $\mathcal P$ as in \cref{def: partitions Q and P}. For each $i=1,\dots,n$ we have: 
\begin{enumerate}
\item $\ccc (U_i) = U_i\cup \ccc (B)$, \label{lem item: closure of Ui}
\item $\ccc(W_i)=W_i\cup B^*$,
\item $\ccc(W_i)$ is homeomorphic to $\omega^*$, 
\item $B^*$ is a closed nowhere dense $P$-set in $\ccc(W_i)$.
\end{enumerate}
\end{lemma}

\begin{proof}
(1) Since $A_i\subseteq U_i$, we have that $B\subseteq\ccc(A_i)\cap \omega^2\subseteq \ccc (A_i) \subseteq \ccc (U_i)$. This yields that $\ccc (B)\subseteq\ccc (U_i)$, so
$U_i\cup \ccc(B)\subseteq \ccc (U_i)$. For the other inclusion, because $\mathcal P$ is a partition of $\beta(\omega^2)$, we have that
\[
\beta(\omega^2)\setminus (U_i\cup \ccc (B)) = U_1\cup\cdots\cup U_{i-1}\cup U_{i+1} \cup\cdots\cup U_n
\]
is open in $\beta(\omega^2)$. Therefore, $U_i\cup \ccc (B)$ is a closed subset of $\beta(\omega^2)$ containing $U_i$. Thus, $\ccc (U_i)\subseteq U_i\cup \ccc (B)$.

(2) We first show
that $W_i\cup B^*$ is closed in $\beta(\omega^2)$. 
Since $\omega^2$ is open in $\beta(\omega^2)$, 
$\mathcal P$ is a partition of $\beta(\omega^2)$, $\{A_i,W_i\}$ is a partition of $U_i$, and $\{B,B^*\}$ is a partition of $\ccc(B)$, we have that
\begin{eqnarray*}
\beta(\omega^2)\setminus \left(W_i\cup B^*\right) &=& U_1\cup\cdots\cup U_{i-1}\cup A_i\cup U_{i+1}\cup\cdots\cup U_n\cup B \\
&=& \omega^2\cup U_1\cup\cdots\cup U_{i-1}\cup U_{i+1} \cup\cdots\cup U_n
\end{eqnarray*} 
is open in $\beta(\omega^2)$. Therefore, $W_i\cup B^*$ is a closed subset of $\beta(\omega^2)$ that contains $W_i$, and hence $\ccc(W_i)\subseteq W_i\cup B^*$. To see the other inclusion, it suffices to show that $B^*\subseteq\ccc(W_i)$.

Let $x\in B^*$ 
and $U$ be a clopen neighborhood of $x$ 
in $\beta(\omega^2)$. 
Since $x\in\ccc(B)\setminus B$ and $\beta(\omega^2)$ is normal, and hence $T_1$, we have that $U\cap B$ is infinite (see, e.g., \cite[Thm.~17.9]{Mun75}). Thus, there is an infinite subset $S$ of $\omega$ such that $U\cap B=\{\langle\omega,m\rangle\mid m\in S\}$. For each $m\in S$ we have that $\langle\omega,m\rangle\in U$, and hence there is $n_m\in\omega$ such that $\langle n_m,m\rangle\in A_i\cap U$ since $B\subseteq\ccc(A_i)$ and $U$ is open. Set $Z=\{\langle n_m,m\rangle \mid \in S\}$. Then $Z$ is an infinite discrete set that is clopen relative to $\omega^2$.
Therefore, $\ccc(Z)$ is clopen in $\beta(\omega^2)$. Since $Z$ is not compact, and hence not closed in $\beta(\omega^2)$, there is $y\in \ccc(Z)\setminus Z$.
Because $B\subseteq \omega^2$, we obtain that 
\[
\ccc(Z)\cap B = \ccc(Z) \cap \omega^2 \cap B = Z\cap B \subseteq A_i\cap B = \varnothing.
\]
Since $\ccc(Z)$ is an open neighborhood of $y$, we have that $y\not\in\ccc(B)$. Because $\mathcal P$ is a partition of $\beta(\omega^2)$, we see that $y\in U_j$ for some $j=1,\dots,n$. Since $U_j$ is open in $\beta(\omega^2)$ and $y\in\ccc(Z)$, we have that $\varnothing\neq U_j\cap Z\subseteq U_j\cap A_i\subseteq U_j\cap U_i$. Therefore, 
$i=j$ and $y\in U_i=A_i\cup W_i$. If $y\in A_i$, then $y\in\ccc(Z)\cap A_i\subseteq
Z$, contradicting $y\not\in Z$. Thus, 
$y\not\in A_i$, and so 
$y\in W_i$. Because $Z\subseteq U$, we have that $y\in \ccc(Z)\subseteq \ccc(U)=U$. Hence, $U\cap W_i\neq\varnothing$, which proves that $B^*\subseteq\ccc(W_i)$. 

(3) To see that $\ccc(W_i)$ is homeomorphic to $\omega^*$, we verify the conditions of Parovi\v{c}enko's Theorem.
Because $\ccc(W_i)$ is a closed subspace of 
$\beta(\omega^2)$, we see that $\ccc(W_i)$ is a Stone space. 
Utilizing CH, it follows from \cite[Thm.~1.2.6]{vMi84} that $\beta(\omega^2)\setminus\omega^2$ is homeomorphic to $\omega^*$. Thus, $W_i$ is crowded since $W_i=\left[\beta(\omega^2)\setminus\omega^2\right]\cap U_i$ is open in $\beta(\omega^2)\setminus\omega^2$. 
Therefore, $\ccc(W_i)$ is also crowded, and it is an $F$-space because it is homeomorphic to a closed subspace of $\omega^*$. As we pointed out before \cref{fig: notations for subsets of omega squared}, $B^*$ and $\omega^*$ are homeomorphic. Since $\ccc(W_i)$ contains $B^*$, 
the weight of $\ccc(W_i)$ is at least $\omega_1$. On the other hand, $\ccc(W_i)$ is a subspace of $\beta(\omega^2)\setminus\omega^2$, whose weight is $\omega_1$.  
Thus, the weight of $\ccc(W_i)$ is exactly $\omega_1$. 

To complete the proof, we show that $\ccc(W_i)$ is an almost $P$-space. Since $\beta(\omega^2)\setminus\omega^2$ is an almost $P$-space and $W_i$ is its open subspace, $W_i$ is an almost $P$-space by \cref{lem: some properties of P-sets and almost P-spaces}(\ref{lem Item: almost P-space is an open hereditary property}).   
By \cite[Lem.~5.1]{vMM80b}, 
$B^*$ is a closed $P$-set in $\beta(\omega^2)\setminus\omega^2$. 
Thus, 
$\ccc(W_i)=W_i\cup B^*$ is an almost $P$-space by \cref{lem: jan lemma 0.1}(\ref{lem item: jan lemma 0.1-compact and almost P implies almost P}). 

(4) This follows from Lemma~\ref{lem: jans lemma 0.2} since $B^*$ is a closed $P$-set in $\beta(\omega^2)\setminus\omega^2$, $W_i$ is an open subset of $\beta(\omega^2)\setminus\omega^2$ that is contained in $\left(\beta(\omega^2)\setminus\omega^2\right)\setminus B^*$ and $\ccc(W_i)=W_i\cup B^*$.
\end{proof}

\begin{lemma}\label{lem: PO unrav WT is image of beta omega squared}
\emph{(CH)} Each partially ordered willow tree is an interior image of $\beta(\omega^2)$.
\end{lemma}

\begin{proof}
Let $\mathfrak F=(W,\le)$ be a partially ordered willow tree with splitting point $s$. 
If $\mathrm{max}(\mathfrak F)=\{s\}$, then it follows from \cite{BH09} (see Cor.~4.10 and Lem.~5.1) that $\F$ is an interior image of $\beta(\omega)$. Thus, since $\beta(\omega)$ is homeomorphic to a clopen subset of $\beta(\omega^2)$, 
we may apply \cref{lem: F image of beta iff F image of open sub of beta} to conclude that $\F$ is an interior image of $\beta(\omega^2)$.

Suppose that $\mathrm{max}(\mathfrak F)\neq\{s\}$. Let $\max(F)=\{m_1,\dots,m_n\}$ and $\mathcal Q,\mathcal P$ be the respective partitions of $\omega^2$ and $\beta(\omega^2)$ (see \cref{def: partitions Q and P}). If $W={\uparrow}s$, then $\mathfrak F$ is a tree of depth 2.
Therefore, it is straightforward to see
that mapping $A_i$ to $m_i$ and $B$ to $s$ yields an interior mapping of $\omega^2$ onto $\F$. Because $\omega^2$ is open in $\beta(\omega^2)$, applying \cref{lem: F image of beta iff F image of open sub of beta} then yields that $\F$ is an interior image of $\beta(\omega^2)$.

If $W\neq{\uparrow}s$, then $T_s={\downarrow}s\setminus\{s\}\neq\varnothing$. Because $B^*$ is homeomorphic to $\omega^*$, 
The First Mapping Lemma implies that there is a ws-mapping $h$ of $B^*$ onto $(T_s,\le)$. We extend $h$ to an interior mapping $f$ of $\beta(\omega^2)$ onto $\mathfrak F$. To do so, we  
recall 
that $\{T_s,T_{i,j}\mid1\le i\le n\text{ and }1\le j\le k_i\}$ is a partition of $W\setminus{\uparrow}s$ 
(see \cref{fig4}). Let $M=\min(W\setminus({\uparrow}s\cup{\downarrow}s))$, and for each $i$, let ${\downarrow}m_i\cap M = \{r_{i,1},\dots,r_{i,k_i}\}$.
Then $T_{i,j}={\uparrow}r_{i,j}\setminus\{m_i\}$,  $\{{\uparrow}s,T_s,T_{i,j}\mid 1\le i\le n\text{ and }1\le j\le k_i\}$ is a partition of $W$, and $w_{i,j}\in T_s$ where $w_{i,j}$ is the parent of $r_{i,j}$. 

By Lemma~\ref{lem: closure of Ui and Wi}, for each $i$ we have that $\ccc(W_i)=W_i\cup B^*$ is homeomorphic to $\omega^*$ and $B^*$ is a closed nowhere dense $P$-set in $W_i\cup B^*$. Because $B^*=\bigcup\{h^{-1}({\downarrow}w)\mid w\in T_s\}$, 
it follows from the Second Partition Lemma
that there is an open partition $\{V_{i,w}\mid  w\in T_s\}$ of $W_i$ such that $\ccc(V_{i,w})=V_{i,w}\cup h^{-1}({\downarrow}w)$ is homeomorphic to $\omega^*$ and $h^{-1}({\downarrow}w)$ is a closed nowhere dense $P$-set in $\ccc(V_{i,w})$ for each $w\in T_s$. Suppose that $k_i\neq 0$ and let $j=1,\dots,k_i$. 
Then ${\downarrow}w_{i,j}$ is a downset in the tree $T_{i,j}\cup{\downarrow}w_{i,j}$ which, since $k_i\neq0$, is disjoint from $\mathrm{max}\left(T_{i,j}\cup{\downarrow}w_{i,j}\right)$. 
Restricting the domain and codomain of $h$ to $h^{-1}({\downarrow}w_{i,j})$ and ${\downarrow}w_{i,j}$ yields a ws-mapping which, by employing the Second Mapping Lemma,
can be extended to a ws-mapping $h_{i,j}:V_{i,w_{i,j}}\cup h^{-1}({\downarrow}w_{i,j}) \to T_{i,j}\cup{\downarrow}w_{i,j}$ such that $h_{i,j}\left(V_{i,w_{i,j}}\right)=T_{i,j}$.
Define $f:\beta(\omega^2)\to \mathfrak F$ by setting
\[
f(x) = \left\{
\begin{array}{ll}
h(x) & \text{if } x\in B^*,\\
s & \text{if } x\in B,\\
h_{i,j}(x) & \text{if } x\in V_{i,w_{i,j}} \text{ for some }i=1,\dots,n \text{ with } k_i\neq0 \text{ and } j=1,\dots,k_i,\\
m_i & \text{if } x\in U_i\setminus\bigcup\nolimits_{j=1}^{k_i}V_{i,w_{i,j}} \text{ for some }i=1,\dots,n.
\end{array}
\right.
\]

To help visualize $f$, \cref{pic of main f} illustrates the sets $V_{i,w}$ and the partition $\{V_{i,w}\mid w\in T_s\}$ of $W_i$ for each $i=1,\dots,n$ (depicted for $i=n$). We have that $h_{i,j}$ exists for each $w_{i,j}$. The case where $v\ne w_{i,j}$ is drawn for $n=1$. The magenta labelling represents $f$.

\begin{figure}[H]
\begin{center}
\begin{picture}(280,130)(-90,5)
\setlength{\unitlength}{.26mm}
\multiput(-90,0)(90,0){2}{\multiput(0,0)(0,125){2}{\line(0,1){65}}}
\multiput(-90,0)(0,65){2}{\multiput(0,0)(0,125){2}{\line(1,0){90}}}
\multiput(-65,20)(0,10){3}{\multiput(0,0)(15,0){2}{\multiput(0,0)(0,125){2}{\makebox(0,0){\tiny$\bullet$}}}}
\multiput(-65,57.5)(15,0){2}{\multiput(0,0)(0,125){2}{\makebox(0,0){\tiny $\vdots$}}}
\multiput(92,100)(-142,0){2}{\makebox(0,0){\tiny $\vdots$}}
\multiput(5,0)(0,190){2}{\line(1,0){25}}
\multiput(5,0)(25,0){2}{\line(0,1){190}}
\multiput(17.5,35)(0,25){5}{\makebox(0,0){\tiny $\bullet$}}
\multiput(17.5,160)(0,125){1}{\makebox(0,0){\tiny $\vdots$}}
\multiput(-30,37.5)(0,125){2}{\makebox(0,0){\tiny$\cdots$}}
\multiput(35,0)(150,0){2}{\line(0,1){190}}
\put(50,131.5){\tiny $V_{1,w_{1,j}}\longmapsto \color{magenta}T_{1,j}$}
\put(50,155.5){\tiny $V_{1,v}\longmapsto \color{magenta}m_1$}
\put(47.5,181){\makebox(0,0){\small $W_1$}}
\put(47.5,55){\makebox(0,0){\small$W_n$}}
\put(87,141){\tiny \color{magenta} $h_{1,j}$}
\put(157,102.5){\scriptsize \color{magenta} $h$}
\put(92,185){\makebox(0,0){\tiny $\vdots$}}
\multiput(92,59)(0,-45){2}{\makebox(0,0){\tiny $\vdots$}}
\put(50,28.5){\tiny $V_{n,w}$}
\put(185,95){\oval(100,190)[l]}
\put(137.5,95){\tiny $B^*\mapsto \color{magenta} T_s$}
\put(138.5,28.5){\tiny $h^{-1}({\downarrow}w)$}
\qbezier(135,20)(160,5)(185,5)
\qbezier(135,45)(160,60)(185,60)
\linethickness{0.6mm}
\multiput(35,65)(0,60){2}{\multiput(0,0)(0,125){1}{\line(1,0){100}}}
\multiput(35,0)(0,190){2}{\multiput(0,0)(0,125){1}{\line(1,0){150}}}
\multiput(35,0)(0,125){2}{\line(0,1){65}}
\linethickness{0.33mm}
\multiput(35,172.5)(0,-140){1}{\line(1,0){100}}
\multiput(35,150)(0,-95){1}{\line(1,0){100}}
\put(35,20){\line(1,0){100}}
\put(35,45){\line(1,0){100}}
\put(-80,55){\makebox(0,0){\small $A_n$}}
\put(-80,180){\makebox(0,0){\small $A_1$}}
\put(12.5,180){\makebox(0,0){\small $B$}}
\put(-65,8){\makebox(0,0){\tiny $V_n$}}
\put(-47,8){\makebox(0,0){\tiny $V_{2n}$}}
\put(-65,133){\makebox(0,0){\tiny $V_1$}}
\put(-45,133){\makebox(0,0){\tiny $V_{n+1}$}}
\color{magenta}
\put(-12,7.5){\makebox(0,0){\small $m_n$}}
\put(-12,132.5){\makebox(0,0){\small $m_1$}}
\put(21,10){\makebox(0,0){\small $s$}}
\end{picture}
\end{center}
\caption{The extension $f:\beta(\omega^2)\to\F$ of $h:B^*\to T_s$.}
\label{pic of main f}
\end{figure}
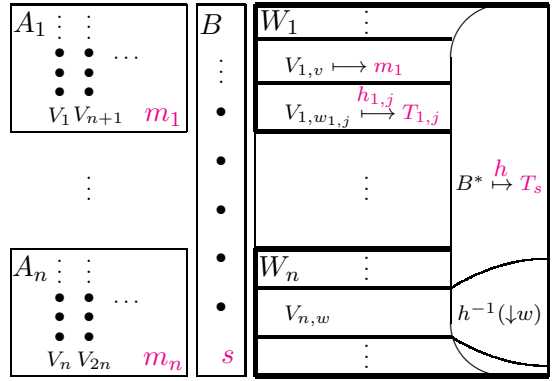

It is straightforward to see that $f$ is a well-defined onto map. 
To see that $f$ is interior, we show that $f^{-1}({\downarrow}w)=\ccc(f^{-1}(w))$ for each $w\in W$. 
We proceed by cases.
 
 \emph{Case 1.} Suppose that $w\in T_s$. Since $B^*$ is closed in $\beta(\omega^2)$ and $h$ is an interior mapping of $B^*$ onto $T_s$, we have 
 \[
f^{-1}({\downarrow}w) = h^{-1}({\downarrow}w) = \ccc(h^{-1}(w)) = \ccc(f^{-1}(w)).
\]

\emph{Case 2.} Suppose that $w=s$. Then ${\downarrow}w={\downarrow}s = \{s\}\cup T_s$, implying that
\[
f^{-1}({\downarrow}w) = f^{-1}(s)\cup f^{-1}(T_s) = B \cup h^{-1}(T_s) = B\cup B^* = \ccc (B) = \ccc(f^{-1}(s)) = \ccc(f^{-1}(w)).
\]

\emph{Case 3.} Suppose that $w\in T_{i,j}$ for some $i=1,\dots,n$ (with $k_i\neq0$) and $j=1,\dots, k_i$. Then ${\downarrow}w = \left(T_{i,j}\cap{\downarrow}w\right)\cup {\downarrow}w_{i,j}$. Because $V_{i,w_{i,j}}\cup h^{-1}({\downarrow}w_{i,j})=\ccc(V_{i,w_{i,j}})$ is closed in $\beta(\omega^2)$ and $h_{i,j}$ is an interior mapping of $V_{i,w_{i,j}}\cup h^{-1}({\downarrow}w_{i,j})$ onto $T_{i,j}\cup {\downarrow}w_{i,j}$ such that $h_{i,j}(x)=h(x)$ for each $x\in h^{-1}({\downarrow}w_{i,j})$, we have 
\begin{eqnarray*}
f^{-1}({\downarrow}w) &=& h_{i,j}^{-1}\left(T_{i,j}\cap{\downarrow}w\right)\cup h^{-1}({\downarrow}w_{i,j}) = h_{i,j}^{-1}\left(T_{i,j}\cap{\downarrow}w\right)\cup h_{i,j}^{-1}({\downarrow}w_{i,j}) \\
&=& h_{i,j}^{-1}({\downarrow}w) = \ccc(h_{i,j}^{-1}(w)) = \ccc(f^{-1}(w)).
\end{eqnarray*}

\emph{Case 4.} Suppose that $w=m_i$ for some $i=1,\dots,n$. Then ${\downarrow}w = \{m_i\}
\cup\left(\bigcup_{j=1}^{k_i}T_{i,j}\right)\cup {\downarrow}s$. It follows from the definition of $f$ that $$f^{-1}(m_i) = U_i\setminus\bigcup\nolimits_{j=1}^{k_i}V_{i,w_{i,j}}$$
and 
\[
f^{-1}\left(\bigcup\nolimits_{j=1}^{k_i}T_{i,j}\right)=\bigcup\nolimits_{j=1}^{k_i}f^{-1}(T_{i,j})=\bigcup\nolimits_{j=1}^{k_i}h_{i,j}^{-1}(T_{i,j})=\bigcup\nolimits_{j=1}^{k_i}V_{i,w_{i,j}}.
\]
(Note that both $\bigcup_{j=1}^{k_i}T_{i,j}$ and $\bigcup_{j=1}^{k_i}V_{i,w_{i,j}}$ are empty if $k_i=0$.) By Lemma~\ref{lem: closure of Ui and Wi}(\ref{lem item: closure of Ui}), 
\begin{eqnarray*}
f^{-1}({\downarrow}w) &=& f^{-1}(m_i)\cup f^{-1}\left(\bigcup\nolimits_{j=1}^{k_i}T_{i,j}\right)\cup f^{-1}({\downarrow}s) \\
&=& \left(U_i\setminus\bigcup\nolimits_{j=1}^{k_i}V_{i,w_{i,j}}\right) \cup \left(\bigcup\nolimits_{j=1}^{k_i}V_{i,w_{i,j}}\right)\cup \ccc(B) = U_i\cup\ccc(B) =\ccc(U_i).
\end{eqnarray*}
To complete the proof, we show that $\ccc(f^{-1}(w))=\ccc(U_i)$; that is, $\ccc(f^{-1}(m_i))=\ccc(U_i)$. Since $A_i\subseteq f^{-1}(m_i)\subseteq U_i$ implies that $\ccc(A_i)\subseteq
\ccc(U_i)$, it suffices to prove that $\ccc(U_i)\subseteq\ccc(A_i)$. 
Let $x\in \ccc(U_i)$. By Lemma~\ref{lem: closure of Ui and Wi}(\ref{lem item: closure of Ui}), 
$x\in \ccc(B)$ or $x\in U_i$. If $x\in\ccc(B)$, then $x\in\ccc(A_i)$ because $B\subseteq\ccc(A_i)$. 
Suppose 
that $x\in U_i$. Let $U$ be an open neighborhood of $x$.
Since $U\cap U_i$ is a nonempty open subset of $\beta(\omega^2)$ and $U_i=\mathrm{Ex}(A_i)$, we have 
\[ 
\varnothing \neq (U\cap U_i)\cap \omega^2 = U\cap (U_i\cap \omega^2) = U\cap A_i.
\]
Thus, $x\in\ccc(A_i)$, completing the proof that $f$ is interior. 
\end{proof}

\subsection{Arbitrary willow trees}

Our next goal is to generalize \cref{lem: PO unrav WT is image of beta omega squared} by proving that an arbitrary willow tree is an interior image of $\beta(\omega^2)$. 
To do so, we recall the well-known topological analogue of a cluster in an {\sf S4}-frame (see \cite[Lem.~5.9]{BBBM17}). Let $X$ be a topological space. A partition of $X$ is {\em dense} if it consists of dense subsets of $X$. The next concept 
is well studied in the literature (see, e.g., \cite{Hew43,Ced64,CG96,Eck97}). For a cardinal $\kappa$, 
call $X$ \emph{$\kappa$-resolvable} provided there is a dense partition of $X$ of cardinality $\kappa$. 

We next recall the standard construction of the skeleton of an {\sf S4}-frame 
(see, e.g., \cite[p.~68]{CZ97}). Let $\mathfrak F=(W,\le)$ be an {\sf S4}-frame. The \emph{skeleton} of $\mathfrak F$ is the partially ordered {\sf S4}-frame $\mathfrak F^*=(V,\preceq)$ where $V:=\{C_w\mid w\in W\}$ is the set of clusters of $\mathfrak F$, 
and for each $w,v\in W$ we have $C_w\preceq C_v$ iff $w\le v$. Then
$\mathfrak F^*$ is a p-morphic image of $\mathfrak F$ via the mapping $\pi:W\to V$ given by $\pi(w)=C_w$. 

\begin{lemma}\label{lem: extending intr map onto PO frame to QO frame}
Let $X$ be a space, $\mathfrak F=(W,\le)$ an {\sf S4}-frame, $\mathfrak F^*=(V,\preceq)$ the skeleton of $\mathfrak F$, 
and $\pi:W\to V$ the p-morphism of $\mathfrak F$ onto $\mathfrak F^*$. 
If $f:X\to\mathfrak F^*$ is an onto interior mapping 
and $f^{-1}(C)$ is $|\pi^{-1}(C)|$-resolvable for each $C\in V$, then there is an onto interior mapping $g:X \to \mathfrak F$ such that $f=\pi\circ g$.
\[
\begin{tikzcd}
&X
\arrow[dl, dashed, "g"']
\arrow[dr, "f"]
&\\
\mathfrak F 
\arrow[rr, "\pi"'] 
&& 
\mathfrak F^*  
&  
\end{tikzcd}
\] 
\end{lemma}

\begin{proof}
Let $C\in V$ and $\kappa_C=|C|$. 
Then $\kappa_C=|\pi^{-1}(C)|$ since $C=\pi^{-1}(C)$. There are an enumeration $\{w_{C,\alpha}\mid \alpha<\kappa_C\}$ of $C$ and a dense partition $\{A_{C,\alpha}\mid \alpha<\kappa_C\}$ of $f^{-1}(C)$. Define $g:X\to W$ by setting $g(x)=w_{C,\alpha}$ whenever $x\in A_{C,\alpha}$. Clearly $g$ is well defined. 
Let $x\in X$. There are unique $C\in V$ and $\alpha<\kappa_C$ such that $x\in A_{C,\alpha}$. Since $A_{C,\alpha}\subseteq f^{-1}(C)$ and $w_{C,\alpha}\in C=\pi^{-1}(C)$,  
we have 
\[ 
f(x) = C =  \pi(w_{C,\alpha}) =\pi(g(x)). 
\]
Therefore, $f=\pi\circ g$.
To see that $g$ is onto, let $w\in W$. Then $w=w_{\pi(w),\alpha}$ for some $\alpha<\kappa_{\pi(w)}$ and there is $x\in A_{\pi(w),\alpha}$ since $A_{\pi(w),\alpha}$ is nonempty. Thus, 
$g(x) = w_{\pi(w),\alpha}=w$, and so $g$ is onto. 

It remains to show that $g$ is interior. To see that $g$ is continuous, let $w\in W$. It follows from the definition of $\pi$ that $\pi({\uparrow}w)$ is an upset of $\mathfrak F^*$ and
${\uparrow}w=\pi^{-1}\pi({\uparrow}w)$. 
Therefore, 
\[
g^{-1}({\uparrow}w) = g^{-1}(\pi^{-1}\pi({\uparrow}w)) = (\pi\circ g)^{-1}(\pi({\uparrow}w)) = f^{-1}(\pi({\uparrow}w))
\]
is an open subset of $X$ since $f$ is continuous. Thus, $g$ is continuous.

To see that $g$ is open, let $U$ be a nonempty open subset of $X$. Since $f$ is open, 
$f(U)$ is an upset of $\mathfrak F^*$. 
Therefore, $\pi^{-1}(f(U))$ is an upset of $\mathfrak F$. Since $f=\pi\circ g$, we have 
\[
g(U) \subseteq \pi^{-1}\pi(g(U)) = \pi^{-1}(f(U)).
\]
To see the other inclusion, let $w\in \pi^{-1}(f(U))$. Then $\pi(w)\in f(U)$, and there is $x\in U$ such that $f(x)=\pi(w)$, which implies that $U\cap f^{-1}(\pi(w))\neq\varnothing$. Because $\{A_{\pi(w),\alpha} \mid \alpha<\kappa_{\pi(w)}\}$ is a dense partition of $f^{-1}(\pi(w))$, we have that $U\cap A_{\pi(w),\alpha}\neq\varnothing$ for each $\alpha<\kappa_{\pi(w)}$. Moreover, there is $\alpha<\kappa_{\pi(w)}$ such that $w=w_{\pi(w),\alpha}$. Taking $y\in U\cap A_{\pi(w),\alpha}$, we have  
\[w=w_{\pi(w),\alpha}=g(y)\in g(U).\]
Thus, $\pi^{-1}(f(U))\subseteq g(U)$, which shows that $g(U)=\pi^{-1}(f(U))$ is an upset of $\mathfrak F$. Consequently, $g$ is open.
\end{proof}

We are now ready to prove that each willow tree is an interior image of $\beta(\omega^2)$.

\begin{lemma}\label{lem: each unrav WT is image of beta omega squared}
\emph{(CH)} Each willow tree is an interior image of $\beta(\omega^2)$.
\end{lemma}

\begin{proof}
Let $\mathfrak F=(W,\le)$ be 
a willow tree with splitting point $s$, 
$\mathfrak F^*=(V,\preceq)$ the skeleton of $\mathfrak F$, and $\pi$ the $p$-morphism of $\mathfrak F$ onto $\mathfrak F^*$ given by $\pi(w)=C_w$. 
Then $\mathfrak F^*$ is a partially ordered 
willow tree. Because the subset ${\uparrow}s$ in $\mathfrak F$ is partially ordered by $\le$, we identify ${\uparrow}s$ with its $\pi$-image 
in $\mathfrak F^*$. 

If $\mathrm{max}(\mathfrak F)=\{s\}$, then $\F$ is an interior image of $\beta(\omega)$, 
and hence 
an interior image of $\beta(\omega^2)$ by \cref{lem: F image of beta iff F image of open sub of beta}.
Suppose that $\mathrm{max}(\mathfrak F)\neq\{s\}$. Then $\mathrm{max}(\mathfrak F^*)\neq\{s\}$ and we may depict $\mathfrak F^*$ as in \cref{fig4}  
where $s,m_i,T_s,T_{i,j}$ are relative to $\mathfrak F^*$, $1\le i\le n$, and $1\le j\le k_i$. Let $f:\beta(\omega^2)\to\mathfrak F^*$ be constructed from the ws-mappings $h:B^*\to V\setminus\{s\}$ and 
\[
h_{i,j}:V_{i,w_{i,j}}\cup h^{-1}({\downarrow}w_{i,j})\to T_{i,j}\cup{\downarrow}w_{i,j}
\]
as in 
the proof of Lemma~\ref{lem: PO unrav WT is image of beta omega squared}. Each of $f^{-1}(s),f^{-1}(m_1),\dots,f^{-1}(m_n)$ is nonempty, and hence $1$-resolvable. Let $C\in V\setminus{\uparrow}s$. Then $f^{-1}(C)$ is open relative to $f^{-1}({\downarrow}C)$. If $C\in T_s$, then 
$f^{-1}({\downarrow}C)=h^{-1}({\downarrow}C)$ is homeomorphic to $\omega^*$ since $h$ is a ws-map. 
If $C\in T_{i,j}$ for some $i=1,\dots,n$ and $j=1,\dots,k_i$, then $f^{-1}({\downarrow}C)=h_{i,j}^{-1}({\downarrow}C)$ is homeomorphic to $\omega^*$ because $h_{i,j}$ is a ws-map. Therefore, $f^{-1}(C)$ is homeomorphic to a nonempty open subspace of $\omega^*$, which implies that $f^{-1}(C)$ is a nonempty crowded locally compact Hausdorff space, and hence $f^{-1}(C)$ is $|\pi^{-1}(C)|$-resolvable by \cite[Lem.~3.6(2)]{BBBM21c}. Thus, $\mathfrak F$ is an interior image of $\beta(\omega^2)$ by Lemma~\ref{lem: extending intr map onto PO frame to QO frame}.
\end{proof}

We have finally arrived at the main result of this paper:

\begin{theorem} \label{thm: P1 for n=2}
The logic of $\beta(\omega^2)$ is the logic of 2-roaches and is the extension of 
${\sf S4.1}$ axiomatized by 
$\chi_1,\chi_2,\chi_3$. 
Therefore, it is finitely axiomatizable, has the finite model property, and is decidable.
\end{theorem}

\begin{proof}
We recall that ${\sf L}_2$ is the logic of $\beta(\omega^2)$, ${\sf L}(\mathcal R_2)$ is the logic of 2-roaches, and ${\sf L}(\mathcal W)$ is the logic of willow trees. By \cref{lem: truth preserving ops}(\ref{lem item: interior image}) and \cref{lem: each unrav WT is image of beta omega squared},
${\sf L_2}\subseteq {\sf L}(\mathcal W)$. Therefore,  
\cref{L(R_2) contained in L(beta(omega_squared))} 
and \cref{thm: logic of unravelled willow trees is L(W)} yield 
\[
{\sf L}(\mathcal R_2) \subseteq {\sf L}_2 \subseteq {\sf L}(\mathcal W) = {\sf L}(\mathcal R_2).
\]
Thus, \cref{thm: the logic of 2-roaches} implies that ${\sf L}_2$ is the extension of ${\sf S4.1}$ axiomatized by $\chi_1,\chi_2,\chi_3$.
\end{proof}

\section{Conclusions}\label{sec: conclusions}

We summarize the current state of knowledge of Shehtman's two problems. \cref{thm: P1 for n=2} provides a solution of {\bf P1} for $n=2$. Our proof uses CH, and it remains open whether CH is necessary. A solution of {\bf P1} for $n\ge 3$ remains a major open problem. We have the following conjecture: 

\begin{conjecture}\label{conj: big conj for L_n}
For $n\ge 3$, ${\sf L}_n$ is the logic ${\sf L}(\mathcal R_n)$ of the class of $n$-roaches.
\end{conjecture}

In particular, we have the following conjecture towards axiomatization of ${\sf L}_3$:

\begin{conjecture}\label{def: forbidden frames for 3-roaches}
The logic ${\sf L}_3$ is 
axiomatizable over {\sf S4.1} 
by the Fine-Jankov formulas of the following nine frames: 
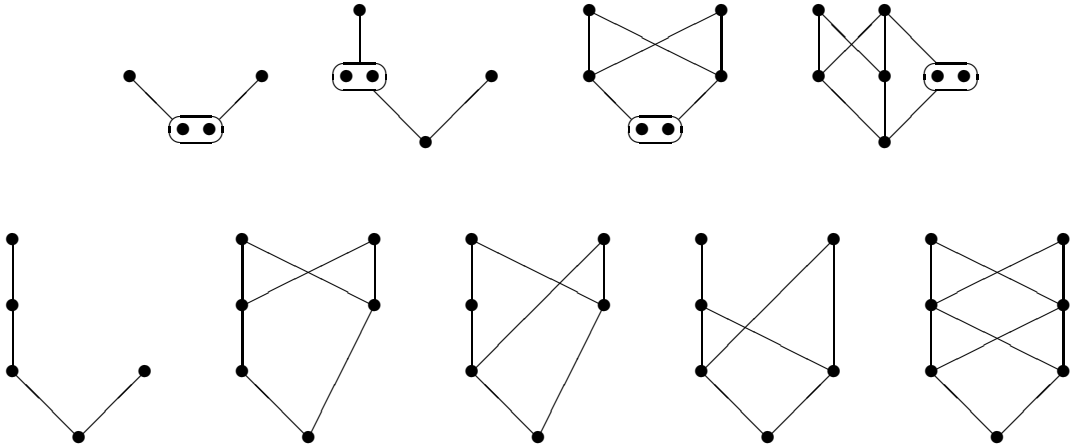
\begin{figure}[H]
\begin{center}

\begin{picture}(50,50)(-25,0)
\multiput(-5,5)(10,0){2}{\makebox(0,0){$\bullet$}}
\put(0,5){\oval(20,10)}
\put(-25,25){\line(1,-1){16}}
\put(25,25){\line(-1,-1){16}}
\multiput(-25,25)(50,0){2}{\makebox(0,0){$\bullet$}}
\end{picture}
\hspace{1cm}
\begin{picture}(50,50)(-25,0)
\put(0,0){\makebox(0,0){$\bullet$}}
\multiput(-30,25)(10,0){2}{\makebox(0,0){$\bullet$}}
\put(-25,25){\oval(20,10)}
\put(0,0){\line(1,1){25}}
\put(0,0){\line(-1,1){20}}
\multiput(-25,50)(50,-25){2}{\makebox(0,0){$\bullet$}}
\put(-25,30){\line(0,1){20}}
\end{picture}
\hspace{1cm}
\begin{picture}(50,50)(-25,0)
\multiput(-5,5)(10,0){2}{\makebox(0,0){$\bullet$}}
\put(0,5){\oval(20,10)}
\put(-25,25){\line(1,-1){16}}
\put(25,25){\line(-1,-1){16}}
\multiput(-25,25)(50,0){2}{\multiput(0,0)(0,25){2}{\makebox(0,0){$\bullet$}}}
\multiput(-25,25)(50,0){2}{\line(0,1){25}}
\put(-25,25){\line(2,1){50}}
\put(25,25){\line(-2,1){50}}
\end{picture}
\hspace{1cm}
\begin{picture}(50,50)(-25,0)
\multiput(20,25)(10,0){2}{\makebox(0,0){$\bullet$}}
\put(25,25){\oval(20,10)}
\put(0,0){\line(1,1){20}}
\put(0,0){\line(-1,1){25}}
\put(0,0){\line(0,1){50}}
\put(0,0){\makebox(0,0){$\bullet$}}
\multiput(-25,25)(25,0){2}{\multiput(0,0)(0,25){2}{\makebox(0,0){$\bullet$}}}
\put(0,0){\line(0,1){50}}
\put(-25,25){\line(0,1){25}}
\put(-25,25){\line(1,1){25}}
\put(20,30){\line(-1,1){20}}
\put(0,25){\line(-1,1){25}}
\end{picture}
\vspace{1.25cm}

\hspace{.55in}
\begin{picture}(50,75)(0,0)
\put(0,0){\makebox(0,0){$\bullet$}}
\put(0,0){\line(1,1){25}}
\put(0,0){\line(-1,1){25}}
\multiput(-25,25)(0,25){3}{\makebox(0,0){$\bullet$}}
\put(-25,25){\line(0,1){50}}
\multiput(25,25)(0,25){1}{\makebox(0,0){$\bullet$}}
\end{picture}
\hspace{1cm}
\begin{picture}(50,75)(0,0)
\put(0,0){\makebox(0,0){$\bullet$}}
\put(0,0){\line(1,2){25}}
\put(0,0){\line(-1,1){25}}
\multiput(-25,25)(0,25){3}{\makebox(0,0){$\bullet$}}
\put(-25,25){\line(0,1){50}}
\multiput(25,50)(0,25){2}{\makebox(0,0){$\bullet$}}
\put(25,50){\line(0,1){25}}
\put(-25,50){\line(2,1){50}}
\put(25,50){\line(-2,1){50}}
\end{picture}
\hspace{1cm}
\begin{picture}(50,75)(0,0)
\put(0,0){\makebox(0,0){$\bullet$}}
\put(0,0){\line(1,2){25}}
\put(0,0){\line(-1,1){25}}
\multiput(-25,25)(0,25){3}{\makebox(0,0){$\bullet$}}
\put(-25,25){\line(0,1){50}}
\multiput(25,50)(0,25){2}{\makebox(0,0){$\bullet$}}
\put(25,50){\line(0,1){25}}
\put(-25,25){\line(1,1){50}}
\put(25,50){\line(-2,1){50}}
\end{picture}
\hspace{1cm}
\begin{picture}(50,75)(0,0)
\put(0,0){\makebox(0,0){$\bullet$}}
\put(0,0){\line(1,1){25}}
\put(0,0){\line(-1,1){25}}
\multiput(-25,25)(0,25){3}{\makebox(0,0){$\bullet$}}
\put(-25,25){\line(0,1){50}}
\multiput(25,25)(0,50){2}{\makebox(0,0){$\bullet$}}
\put(25,25){\line(0,1){50}}
\put(-25,25){\line(1,1){50}}
\put(25,25){\line(-2,1){50}}
\end{picture}
\hspace{1cm}
\begin{picture}(50,75)(0,0)
\put(0,0){\makebox(0,0){$\bullet$}}
\put(0,0){\line(1,1){25}}
\put(0,0){\line(-1,1){25}}
\multiput(-25,25)(0,25){3}{\makebox(0,0){$\bullet$}}
\multiput(-25,25)(50,0){2}{\line(0,1){50}}
\multiput(25,25)(0,25){3}{\makebox(0,0){$\bullet$}}
\multiput(25,25)(0,25){2}{\line(-2,1){50}}
\multiput(-25,25)(0,25){2}{\line(2,1){50}}
\end{picture}
\end{center}
\end{figure}
\end{conjecture}

\cref{thm: partial soln of P2} provides a solution of {\bf P2} for 
ordinals whose Cantor normal form has a special form.
A solution of 
{\bf P2} for an arbitrary ordinal remains open. We have the following conjecture: 

\vspace{2.5mm}

\begin{conjecture} \label{conj}
\begin{enumerate}
\item[]
\item If $\gamma$ has $\alpha_1\ge\omega$ in its Cantor normal form, then ${\sf L}(\beta(\gamma))={\sf L}_\infty$, which is the logic ${\sf L}(\mathcal R_\infty)$ of the class of all roaches.
\item The list of logics of the form ${\sf L}(\beta(\gamma))$ for some ordinal $\gamma$ is obtained by adding
${\sf L}_\infty$ 
to the list in Corollary~\ref{thm: partial soln of P2} $($see \cref{fig: logics of beta gamma}$)$.
\end{enumerate}
\end{conjecture}

\section*{Acknowledgements}
We would like to thank Ilya Shapirovsky for useful discussions regarding axiomatization issues for the logic of $n$-roaches.

\bibliographystyle{amsplain}
\def\cprime{$'$}
\providecommand{\bysame}{\leavevmode\hbox to3em{\hrulefill}\thinspace}
\providecommand{\MR}{\relax\ifhmode\unskip\space\fi MR }

\providecommand{\MRhref}[2]{
  \href{http://www.ams.org/mathscinet-getitem?mr=#1}{#2}
}
\providecommand{\href}[2]{#2}

\end{document}